\documentclass[preprint,3p]{elsarticle}
\usepackage{etoolbox}
\makeatletter
\patchcmd{\ps@pprintTitle}{\footnotesize\itshape
       Preprint submitted to \ifx\@journal\@empty Elsevier
       \else\@journal\fi\hfill\today}{\relax}{}{}
\makeatother
\usepackage{epsfig,amssymb,amsmath,natbib,listings,lineno,xcolor,natbib,bm,algorithm,url,amsthm}
\usepackage{tabulary}
\usepackage{booktabs}
\usepackage[noend]{algpseudocode}
\usepackage[normalem]{ulem}
\usepackage{subfig}
\usepackage{float}
\newtheorem{theorem}{Theorem}[section]

\newtheorem{lemma}[theorem]{Lemma}
\newtheorem{remark}{Remark}[section]
\newtheorem{definition}{Definition}[section]

\DeclareMathOperator*{\argmin}{argmin}
\usepackage{mathtools}
\DeclarePairedDelimiter{\ceil}{\lceil}{\rceil}
\newcommand{\alec}[1]{\textcolor{black}{#1}}

\begin{document}

\begin{frontmatter}

\title{Deterministic matrix sketches for low-rank \alec{compression} of high-dimensional simulation data}

\author[colorado1]{Alec M. Dunton}
\author[colorado2]{Alireza Doostan\corref{cor1}}
\cortext[cor1]{Corresponding author, alireza.doostan@colorado.edu}
\address[colorado1]{Applied Mathematics, University of Colorado, Boulder, CO 80309, USA}
\address[colorado2]{Smead Aerospace Engineering Sciences, University of Colorado, Boulder, CO 80309, USA}

\begin{abstract}
Matrices arising in scientific applications frequently admit linear low-rank approximations due to smoothness in the physical and/or temporal domain of the problem. In large-scale problems, computing an optimal low-rank approximation can be prohibitively expensive. Matrix sketching addresses this by reducing the input matrix to a smaller, but representative matrix via a low-dimensional linear {\it embedding}. If the sketch matrix produced by the embedding sufficiently captures the geometric properties of the original matrix, then a near-optimal approximation may be obtained. Much of the work done in matrix sketching has centered on random projection.~\alec{Alternatively,} in this work, {\it deterministic} matrix sketches which generate coarse representations -- compatible with the corresponding PDE solve -- are \alec{considered in the computation of the singular value decomposition and matrix interpolative decomposition}. \alec{The deterministic sketching approaches in this work have many advantages over randomized sketches. Broadly, randomized sketches are data-agnostic, whereas the proposed sketching methods exploit structures within data generated in complex PDE systems. These deterministic sketches are often faster, require access to a small fraction of the input matrix, and do not need to be explicitly constructed.} A novel single-pass\alec{, i.e., requiring one read over the input,} power iteration algorithm is also presented. The power iteration method is particularly effective \alec{in improving low-rank approximations} when the singular value decay of data is slow. Finally, theoretical error bounds and estimates, as well as numerical results across three application problems, are provided.
\end{abstract}


\end{frontmatter}

\nolinenumbers
\section{Introduction}	
\label{sec:intro}
In complex systems modeled by partial differential equations (PDEs), data generated from the numerical solution of the given equation(s) \alec{is often} written to file in 2D array (matrix) format, with each entry corresponding to a quantity of interest (QoI) measured at a given grid point, for a given input parameter, and/or at a given time-step (see Figure~\ref{fig:A_schematic}). If this array exhibits spatial or temporal smoothness, e.g., in diffusion-dominated problems solved on a Cartesian grid, it may admit an accurate low-rank approximation. In other words, the matrix may be approximated as a product of smaller {\it factor matrices} whose dimensions reflect the extrinsic dimensionality of the \alec{data.} 

Applications of low-rank approximations are ubiquitous in computational mathematics. 
When matrix-vector products must be evaluated many times, for example, in Krylov methods, low-rank approximations enable significant speedup via reduction in computational complexity. In data visualization, low-rank approximations yield useful representations of high-dimensional data by producing embeddings in two or three dimensions.
In lossy data compression, the factor matrices whose product approximates the original matrix require far fewer bytes in storage than the original \alec{matrix.}

\begin{figure}[t]

  \centering
  \includegraphics[width=0.49\textwidth]{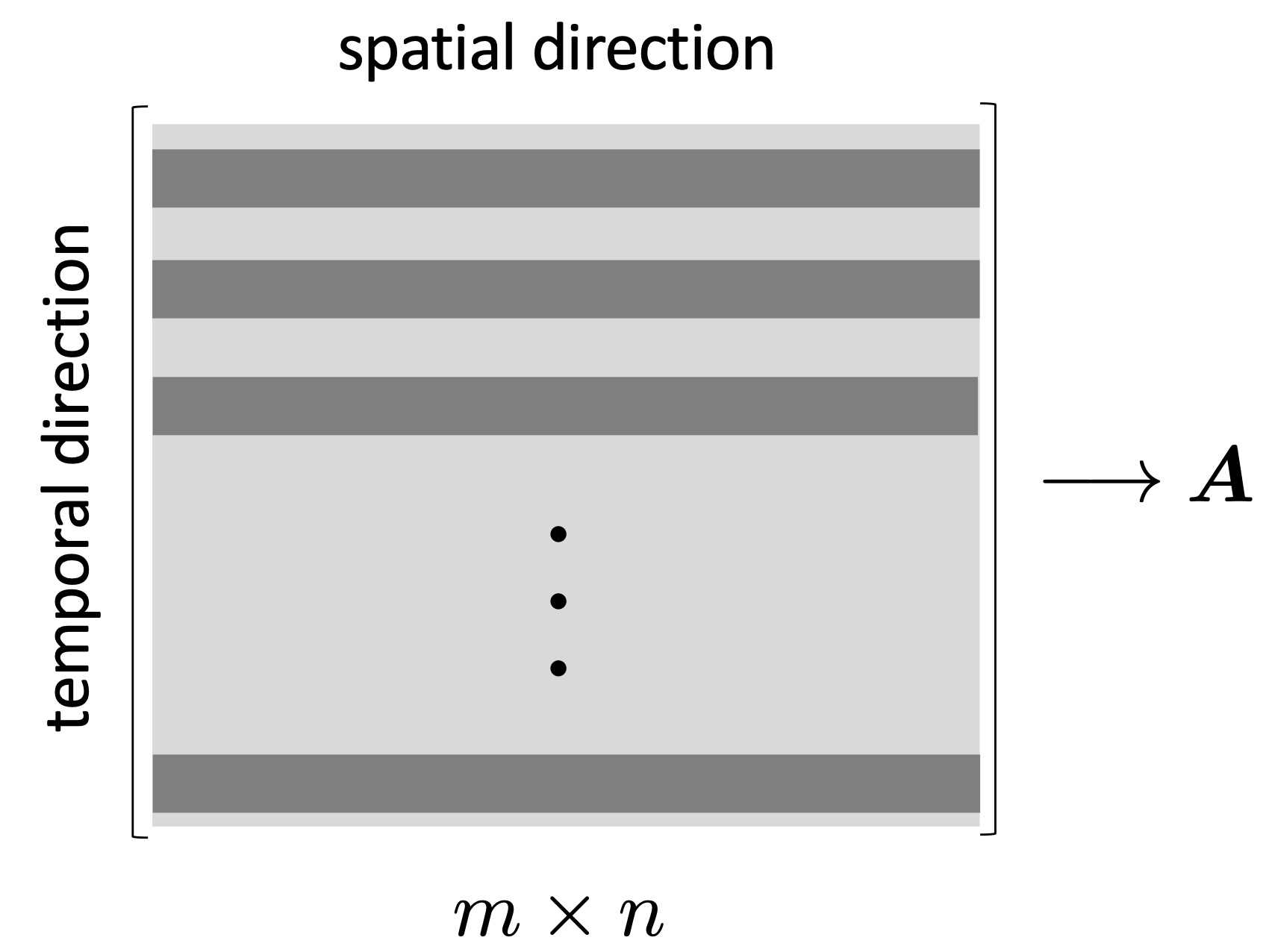}
  \caption{Schematic of a data matrix $\bm A$ with $m$ time solutions (rows) and $n$ spatial degrees of freedom~\cite{dunton2020pass}. }	\label{fig:A_schematic}
\end{figure}


Low-rank matrix approximation can be formulated as the following minimization problem. Let $\bm{A}_k$ be a least-squares optimal rank-$k$ approximation of a matrix $\bm{A}$. Then,
\begin{equation}
    \bm{A}_k = \argmin_{\bm{X}} \{ \Vert \bm{A} - \bm{X} \Vert_{\xi} : \text{rank}(\bm{X}) = k \},
\end{equation}
where $\xi = 2,F$; $\Vert \cdot \Vert_2$ is the matrix 2-norm, while $\Vert \cdot \Vert_F$ is the matrix Frobenius norm. 
By the Eckart-Young theorem~\cite{eckart1936approximation}, $\bm{A}_k$ can be obtained by computing the {\it singular value decomposition} (SVD) of the matrix $\bm{A}$ and storing the $k$ largest singular values and $k$ corresponding left and right singular vectors. This optimal rank-$k$ approximation then satisfies 
\begin{equation}
    \Vert \bm{A} - \bm{A}_k \Vert_2 = \sigma_{k+1}.
\end{equation}

Due to the potential impracticality of computing an optimal rank-$k$ approximation of a matrix, a near-optimal solution is often acceptable.
To this end, approximate rank-$k$ SVD methods are proposed. This takes the form of
\begin{equation}
    \hat{\bm{A}}_k = \tilde{\bm{U}}_k\tilde{\bm{S}}_k\tilde{\bm{V}}_k^T, 
\end{equation}
where $\hat{\bm{A}}_k \in \mathbb{R}^{m \times n}$ is the approximation to $\bm{A}$, $\tilde{\bm{U}}_k \in \mathbb{R}^{m \times k}$ is a matrix whose columns comprise approximations to the first $k$ left singular vectors of $\bm{A}$, $\tilde{\bm{S}}_k \in \mathbb{R}^{k \times k}$ is a diagonal matrix whose entries are approximations to the $k$ largest singular values of $\bm{A}$, and $\tilde{\bm{V}}_k \in \mathbb{R}^{n \times k}$ is a matrix whose columns comprise approximations to the first $k$ right singular vectors of $\bm{A}$. \alec{The factor matrices are approximations to the true SVD factor matrices because the full SVD of the matrix is never computed in this approach.}

\alec{The SVD is an example of low-rank approximation which identifies orthogonal bases for the fundamental subspaces of a matrix. Another class of methods, {\it self-expressive} decompositions, approximate a matrix using, e.g., a row skeleton, column skeleton, or sub-matrix of the input. To this end,}
methods for computing low-rank approximations of $\bm{A}$ via the row interpolative decomposition~\cite{cheng2005compression} (ID) are provided. An ID takes the form
\begin{equation}
    \hat{\bm{A}}_k = \bm{P}\bm{A}(\mathcal{I},:),
\end{equation}
where $\bm{A}(\mathcal{I},:) \in \mathbb{R}^{k \times n}$ is comprised of a subset of the rows of $\bm{A}$ indexed by \alec{$\mathcal{I} \subseteq \lbrack 1, \dots, n \rbrack$}, and $\bm{P} \in \mathbb{R}^{m \times k}$ is a matrix which approximately maps $\bm{A}(\mathcal{I},:)$ to $\bm{A}$.

Although the factor matrices comprising a low-rank approximation of a matrix require far less disk storage than the original matrix, constructing such factor matrices can itself require prohibitively large amounts of working memory (RAM). When the size of a data matrix $\bm{A} \in \mathbb{R}^{m \times n}$ exceeds the system's RAM, instead of storing the entire matrix to memory to then compute a decomposition, a {\it sketch} of the matrix, $\bm{AD} \in \mathbb{R}^{m \times n_c}$, reduced in dimension via right multiplication by a matrix $\bm{D} \in \mathbb{R}^{n \times n_c}$ with $n_c \ll n$, may be stored instead. In an online setting where the snapshots of PDE data arrive in a streaming fashion, sketch matrices may be constructed in a single-pass.

Matrix sketching was, to the best of the authors' knowledge, first used to compute low-rank approximations for latent semantic indexing~\cite{papadimitriou2000latent}. In that work and since, much of the research in matrix sketching for numerical linear algebra has focused on {\it randomized} methods, wherein a data matrix is projected to a lower dimension via linear random embedding. This has taken the form of sparsification techniques, which \alec{yield a sparse output matrix via} \alec{sub-sampling} the entries of a matrix according to a \alec{probability} distribution defined by the matrix entries~\cite{achlioptas2007fast,arora2006fast,gittens2009error,spielman2011graph}. Another popular class of methods is random projection, where a matrix sketch is typically computed by multiplying the data matrix $\bm{A}$ by a matrix with random entries. Examples of random projections include, but are not limited to, the sub-sampled randomized Hadamard transform~\cite{woolfe2008fast,tropp2011improved,gu2015subspace}, sub-sampled random Fourier transform~\cite{liberty2007randomized,woolfe2008fast}, Gaussian matrices~\cite{dasgupta2003elementary,halko2011finding,gu2015subspace}, and the fast Johnson-Lindenstrauss transform (FJLT)~\cite{ailon2009fast}.

\alec{Randomized sketching is used throughout numerical linear algebra, reducing runtimes while maintaining acceptable accuracy in applications such as clustering~\cite{ailon2009fast}, classification~\cite{cannings2015random}, and low-rank approximation~\cite{halko2011finding}. A key result providing conditions under which randomized sketches perform well is the {\it Johnson-Lindenstrauss Lemma}~\cite{johnson1984extensions}. One version of the lemma from~\cite{dasgupta2003elementary} is provided below.
\begin{lemma} Given $0<\epsilon<1$, a set X of m points in $\mathbb{R}^N$, and a positive integer $n > 4(\epsilon^2/2 - \epsilon^3/3)^{-1} \log(m)  \epsilon^2$, there is a linear map f : $\mathbb{R}^N \rightarrow \mathbb{R}^n$ such that
\begin{equation}
    \left(1 - \epsilon\right) \Vert \bm{u} - \bm{v} \Vert_2 \leq \Vert f(\bm{u}) - f(\bm{v}) \Vert_2 \leq \left(1 + \epsilon \right) \Vert \bm{u} - \bm{v} \Vert_2,
\end{equation}
for all $\bm{u},\bm{v} \in X$.
\end{lemma}}

\alec{A consequence of this lemma is that a sketch matrix respresenting the linear map $f$ can be drawn from a suitable probability distribution such that the embedding it yields preserves the geometry of the data, e.g., pairwise distances and inner products between rows, with high probability.
Geometry preservation in low-dimensional embeddings is a desirable property among matrix sketches, and informs the development of grid-compatible deterministic sketches for PDE data. In this work, probabilistic guarantees on the performance of the sketches are not provided. The analysis relies instead on the existence of a deterministic linear mapping between the coarse grid and fine grid data.}

Other approaches for matrix sketching include column subset selection (CSS) methods. The goal of a CSS scheme is to obtain a subset of $k$ columns $\bm{C}$ over all subsets of size $k$ of the columns of $\bm{A}$ such that 

\begin{equation}
\label{eq:cssp}
    \bm{C} = \argmin_{\bm{X}} \Vert \bm{A} - \bm{X}\bm{X}^+\bm{A} \Vert ,
\end{equation}
where $\bm{X}^+$ is the Moore-Penrose pseudoinverse of $\bm{X}$.
Finding the optimal solution to the problem (\ref{eq:cssp}) is NP-complete~\cite{shitov2017column}. Therefore, in many cases an approximate solution to the problem is acceptable. Randomized \alec{row/column} sampling methods for low-rank matrix approximation were pioneered by Frieze, Kannan, and Vempala~\cite{frieze2004fast}, and later improved upon in~\cite{drineas2004clustering,drineas2006fast}. Deterministic schemes for matrix sketching, namely frequent directions, have been developed in recent years~\cite{liberty2013simple,ghashami2016frequent}. The version of frequent directions proposed in~\cite{liberty2013simple} is noteworthy; this method can be implemented in a streaming fashion.

When loading the data matrix into RAM is cost-prohibitive, minimizing the number the number of {\it passes} an algorithm makes over the input matrix is of utmost importance. If this is not taken into consideration, the cost of reading the matrix itself becomes a computational bottleneck. To address such concerns, algorithms which are {\it pass-efficient}, i.e., algorithms which require minimal passes over their input, are prioritized. The methods presented in this work are exclusively single-pass and two-pass algorithms.

The development of single-pass algorithms for the decomposition of matrices can be traced back to work by Clarkson and Woodruff~\cite{clarkson2009numerical}. Another early example of a single-pass algorithm for computing an approximate singular value decomposition is found in Section 5.5 of the review paper by Halko et al.~\cite{halko2011finding}. Their proposal comes with the caveat of relying on an ill-conditioned least-squares problem. Single-pass algorithms for low-rank matrix approximation have since been developed in~\cite{,liberty2013simple,woodruff2014sketching,boutsidis2016optimal,upadhyay2016fast,tropp2017practical,yu2017single,dunton2020pass}.

\subsection{Contribution of this work}
\alec{Randomized sketches have enabled dramatic speedup and performance gains in low-rank approximation methods. A great benefit of these approaches is that they are {\it data-agnostic}. However, in the case of PDE simulation data, there is typically some knowledge of the smoothness or underlying spatial domain. Exploiting this knowledge of the PDE simulations from which the data is generated, the deterministic sketches presented in this work {\it are often faster than randomized sketches}, {\it require access to a fraction of the fine grid data} (randomized sketches typically require access to all of it), and {\it need not be formed explicitly as matrices} (many randomized sketches do).}

\alec{To this end,} several pass-efficient deterministic matrix sketching algorithms for computing the \alec{SVD} and \alec{row ID} of PDE data matrices \alec{are presented}. To provide guarantees on the accuracy of the proposed algorithms, novel theoretical results demonstrating deterministic matrix sketches to be an attractive tool under appropriate assumptions on the data are presented. Then, a single-pass algorithm for computing a low-rank matrix \alec{ID} which generalizes an algorithm from~\cite{dunton2020pass} is presented. This work also presents, to the best of the authors' knowledge, the first algorithm which uses a matrix sketch for power iteration~\cite{halko2011finding} to compute a low-rank SVD approximation. The proposed power iteration scheme yields  substantial improvements in approximation accuracy \alec{and algorithm robustness when the singular value decay of a matrix is slow; it is the first of its kind which requires} \textit{no additional passes over the data matrix}. 

Numerical results demonstrate the efficacy of the proposed schemes in three different application problems. Finally, the proposed low-rank methods are implemented in tandem with three state-of-the-art \alec{lossy scientific data compressors: FPZIP, a predictive coder; ZFP, which computes custom orthogonal transforms on $4^d$ blocks of array data; and SZ-2.0, a predictive coder which generalizes to high-dimensional data tensors. In using the approaches proposed in this work in tandem with these compressors, a low-rank SVD or row ID approximation of the data is first computed. Then, the factor matrices comprising the approximation are passed as input into FPZIP, ZFP, and SZ-2.0. These three compressors are best suited for logically regular arrays whose structure reflects spatial locality in the data. This makes them less well-suited to unstructured data, whereas our methods provably compress unstructured data by identifying low-rank structure in the temporal and parametric domain.} This hybrid workflow consequently yields dramatic improvement in compression ratios with minimal loss of accuracy.

The remainder of this manuscript is outlined as follows. In Section~\ref{sec:proposedschemes}, three deterministic sketching methods for low-rank approximation are proposed. In Section~\ref{sec:twopass}, a theoretical exposition on computing a low-rank SVD approximation of a matrix based on deterministic sketching is provided. In Section~\ref{sec:singlepass}, a single-pass algorithm based on the framework presented in~\cite{yu2017single} is derived. In Section~\ref{sec:SPCID}, the single-pass ID algorithm originally presented in~\cite{dunton2020pass} is generalized. In Section~\ref{sec:coarsegridpoweriteration}, a novel power iteration scheme which relies on approximating the Gramian of the data matrix using a coarse grid sketch is proposed. In Section~\ref{sec:numericalexperiments}, error and runtime experiments for the aforementioned SVD algorithms are conducted. In Section~\ref{sec:numericalexperimentsID}, error and runtime results for the ID algorithms are provided. In Section~\ref{sec:application}, the proposed methods are used with the data compression techniques ZFP~\cite{lindstrom2014fixed}, FPZIP~\cite{lindstrom2006fast}, and SZ~\cite{di2016fast,tao2017significantly,liang2018error} to achieve enhanced \textit{spatio-temporal} compression of data matrices. In Section~\ref{sec:proofs}, proofs of all primary theoretical results are given. In Section~\ref{sec:conclusions}, a brief summary of the results in this manuscript is provided.
\section{Deterministic sketches}
\label{sec:proposedschemes}

\begin{figure}[t]
\centering
\includegraphics[width=0.6\linewidth]{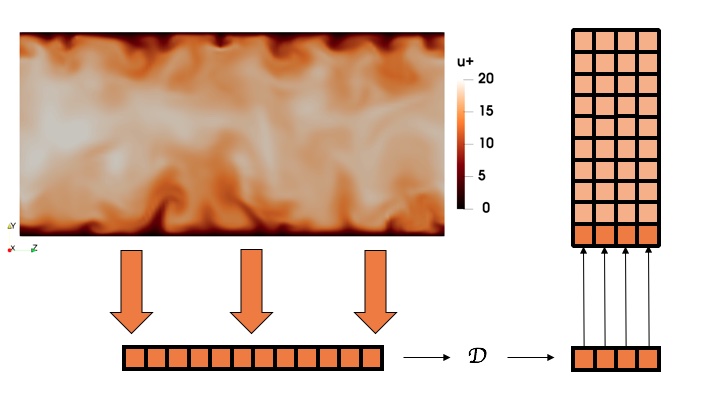}
\caption{Simulation snapshots are read into memory, vectorized, and sketched to form the sketch matrix.}
\label{fig:sketching}
\end{figure}


Throughout this work, input data is assumed to follow the structure given in Figure~\ref{fig:A_schematic}. The ground truth is measured on a fine grid, yielding a corresponding fine-grid data matrix denoted $\bm{A}_f$ where the rows of $\bm{A}_f$ correspond to the snapshots of the solution at different time instances (or solution realizations for different input parameter values).
The coarse grid data matrix $\bm{A}_c \in \mathbb{R}^{m \times n_c}$ is \alec{the} fine grid matrix $\bm{A}_f \in \mathbb{R}^{m \times n}$ deterministically mapped into a lower dimension via a linear operator $\bm{D} \in \mathbb{R}^{n \times n_c}$,
\begin{equation}
\label{eq:generaldeterministicsketch}
\bm{A}_c = \bm{A}_f\bm{D}.
\end{equation}
\alec{In the following subsection, we specify the properties of the matrix $\bm{D}$.}

\subsection{Deterministic sketches for low-rank approximations}
Three different deterministic matrix sketching framework for the dimension reduction of $\bm{A}_f$ are considered. They are referred to as \alec{{\it direct injection}}, \alec{{\it nearest neighbor sketch}}, and \alec{{\it global sketch}}. 
Generally, these sketches may be represented as {\it weighted} averages of the QoI over a \alec{stencil.} The stencil size \alec{determines the} fine grid QoI values used to compute each coarse grid QoI value, which is equivalent to the maximum number of nonzero entries in each column of the sketch matrix $\bm{D}$. 
Let $s = \ceil{n/n_c}$ be the sub-sampling factor, i.e., \alec{the factor by which the feature space dimension is reduced.}

The first proposed sketch is {\it direct injection}, wherein the coarse grid sketch matrix is directly sub-sampled from the fine grid.
\alec{\begin{equation}
 \label{eq:directinjectionsketch}
  \bm{D}_{i,j}=\begin{cases}
    1 & \text{if} \hspace{6pt} i = js + 1, \hspace{6pt} j = 1, \dots, n_c, \\
    0 & \text{otherwise}.
  \end{cases} 
\end{equation}}
 
The second sketch, {\it nearest neighbors}, describes a general framework wherein grid points at the center of each stencil are weighted more heavily than their neighbors. \alec{This is similar to convolutional filters used for smoothing in image processing. For example, in 1-D one could define the nearest neighbor sketch matrix to have a stencil of size 3 and weights of $1/4$, $1/2$, and $1/4$:}
\alec{\begin{equation}
 \label{eq:nearestneighborsketch}
  \bm{D}_{i,j}=\begin{cases}
    1/4 & \text{if} \hspace{6pt} \vert i - js \vert = 1, \hspace{6pt} j = 1, \dots, n_c, \\
    1/2 & \text{if} \hspace{6pt} i = js, \hspace{6pt} j = 1, \dots, n_c, \\
    0 & \text{otherwise}.
  \end{cases} 
\end{equation}}

The third proposed deterministic sketching framework, {\it global sketch}, computes an {\it unweighted} average over every $s$ degrees of freedom.
\alec{\begin{equation}
 \label{eq:globalsketch}
  \bm{D}_{i,j}=\begin{cases}
    1/s & \text{if} \hspace{6pt} i = js + k, \hspace{6pt} j = 1, \dots, n_c, \hspace{6pt} k = 1, \dots, s,\\
    0 & \text{otherwise}.
  \end{cases} 
\end{equation}}

\begin{table}[t]
\centering
\tabcolsep7pt\begin{tabular}{lccc}
\hline
Method & Computational complexity &  Stencil size & Fraction of fine grid used\\
\hline
Direct injection & $\mathcal{O}(1)$  &  1 & $ n_c/n$ \\
Nearest neighbor & $\mathcal{O}(dmn_c)$  &  $2d+1$ & $\min((2d+1) n_c/n,1)$ \\
Global sketch & $\mathcal{O}(mnn_c)$ & $s-1$ & $1$\\
Dense Gaussian matrix & $\mathcal{O}(mnn_c)$ & $n$ & $1$ \\
\hline
\end{tabular}
\caption{Computational complexity of the deterministic sketch operations proposed in this work. $m$ corresponds to the number of rows of the data matrix (number of snapshots from the PDE data), $n$ corresponds to the number of degrees of freedom in the fine grid, $n_c$ corresponds to the reduced dimension of the coarse data, $s = \lceil n/n_c \rceil$ is the sub-sampling factor, and $d$ is the number of neighbors in the nearest neighbor sketch.}
\label{tab:sketches}
\end{table}
In Table~\ref{tab:sketches}, properties of the three sketching frameworks are provided. The computational complexity is the leading order term (up to a constant) of the number of floating-point operations (FLOPs) required to form the coarse grid data matrix $\bm{A}_c$ from $\bm{A}_f$. 
The fastest of the three deterministic sketches is direct injection; it requires $\mathcal{O}(1)$ FLOPs to form the deterministic sketch $\bm{A}_c$. Having the smallest stencil of the three sketches, direct injection also uses the lowest fraction of information from the fine grid data $\bm{A}_f$. The nearest neighbor sketch constitutes the middle ground with respect to both computational complexity and fraction of fine grid used; it has stencil of size $2d+1$. Tuning $d$ allows the user to determine (1) how quickly the sketch can be evaluated and (2) what fraction of the fine grid data is used to construct the coarse grid. Finally, global sketch can be viewed as an averaging filter which requires the largest FLOP count of the three methods, but \alec{uses all of the entries in $\bm{A}_f$ to form $\bm{A}_c$}. 

\begin{remark}
In accessing data from a remote server, direct injection and nearest neighbor may offer significant reduction in memory movement. For example, in sketching 3D data, a direct injection sketch comprised of sub-sampling in each dimension by a factor of 10 would enable an approximate 1000-fold download speedup. Global sketch, and more broadly sketches which require access to entire snapshots, e.g., dense Gaussian sketches, do not offer the same benefit. 
\end{remark}


\section{Singular value decomposition algorithms}
\label{sec:twopass}
\begin{algorithm}[t]
\caption{PROTO-TPC $\bm{A}_f \approx \hat{\bm{A}}_f =  \tilde{\bm{U}}_k\tilde{\bm{S}}_k\tilde{\bm{V}}_k^{T}$}  \label{alg:protosvd}
\begin{algorithmic}[1]
\Procedure{PROTO-TPC}{$\bm{A}_f \in \mathbb{R}^{m \times n}$}
\State $\bm{D} \gets$ deterministic sketch matrix
\State $k \gets$ target rank
\State $p \gets$ oversampling parameter
\State $n_c \gets k + p$ coarse grid size
\State $\bm{A}_c \gets \bm{A}_f\bm{D}$
\State $\bm{U}_c \bm{S}_c\bm{V}_c^T = SVD(\bm{A}_c)$
\State $\tilde{\bm{U}}_k \gets \bm{U}_c(:,1:k)$; $\tilde{\bm{S}}_k \gets \bm{S}_c(1:k, 1:k)$;
\State $\tilde{\bm{V}}_{k} \gets ((\bm{U}_k\bm{S}_k)^{+}\bm{A}_f)^T$
\State $\bm{return}$ $\tilde{\bm{U}}_k, \tilde{\bm{S}}_k, \tilde{\bm{V}}_k$
\EndProcedure
\end{algorithmic}
\end{algorithm} 
This section presents analysis which identifies conditions under which the coarse grid matrix $\bm{A}_c$ retains enough information to enable accurate approximation of the first $k$ left singular vectors of using PROTO-TPC (Algorithm~\ref{alg:protosvd}).
First, the SVD of $\bm{A}_c$ is computed.
\begin{align}
    \bm{A}_c &= \bm{U}_c\bm{S}_c\bm{V}^T_c. 
\end{align}
Then, the following matrices which constitute the final approximation are formed. In MATLAB notation,
\begin{align}
    \tilde{\bm{U}}_k &= \bm{U}_c(:,1:k),\\
    \tilde{\bm{S}}_k &=  \bm{S}_c(1:k,1:k), \\
    \tilde{\bm{V}}_k &= \bm{A}_f^T\tilde{\bm{U}}_k\tilde{\bm{S}}_k^+,\\
    \bm{A}_f & \approx \tilde{\bm{U}}_k \tilde{\bm{S}}_k \tilde{\bm{V}}^T_k \label{eqn:kranksvdapprox}.
\end{align}
The following definition provides the foundation for deriving an error bound on $\Vert \bm{A}_f - \tilde{\bm{U}}_k\tilde{\bm{S}}_k\tilde{\bm{V}}_k^T \Vert_2$. 
\alec{\begin{definition}
\label{def:interpop}
Let $\bm{M} : \mathbb{R}^{n_c} \rightarrow \mathbb{R}^{n}$ be a linear operator such that $\bm{A}_f = \bm{A}_c\bm{M} + \bm{E}_I$, with $\bm{E}_I \in \mathbb{R}^{m \times n}$ the matrix containing the corresponding approximation error.
\end{definition}}

\begin{theorem}
\label{thm:tpsvdcoarseerror}
Let $\bm{A}_c \in \mathbb{R}^{m \times n_c}$ be a coarse grid sketch of $\bm{A}_f$, $\sigma_{c,k+1}$ be the $(k+1)^{th}$ largest
singular value of $\bm{A}_c$, and $\bm{E}_I$ and $\bm{M}$ be as in Definition~\ref{def:interpop}. Then,  the \alec{rank-$k$} approximation \alec{(\ref{eqn:kranksvdapprox})} satisfies 
\begin{equation}
\Vert \bm{A}_f - \tilde{\bm{U}}_k\tilde{\bm{\Sigma}}_k\tilde{\bm{V}}^T_k \Vert_2 \leq \Vert \bm{M} \Vert_2  \sigma_{c,k+1}   + \Vert \bm{E}_I \Vert_2.
\end{equation}

\end{theorem}
\begin{proof}
See Section~\ref{sec:prooftpsvdcoarseerror}.
\end{proof}

This error bound indicates that the performance of the algorithm depends on the singular value decay of the coarse grid data matrix, as well as its approximate mapping to the fine grid \alec{data}. If there exists an accurate mapping of the QoI from \alec{the coarse to fine grid} (quantified in the matrix $\bm{E}_I$), then one should expect that the singular vectors and values computed from $\bm{A}_c$ are representative of those of $\bm{A}_f$ to within error governed by the mapping.


PROTO-TPC \alec{(Algorithm~\ref{alg:protosvd})} is outlined as a theoretical exposition; its practicality is of major concern. First, the approximation to the right singular vectors, $\bm{V}^T_k$, is not an orthonormal matrix. Second, computing the SVD of $\bm{A}_c$ is expensive, and contributes significantly to the overall complexity of the algorithm. To address these concerns, the basic two-pass randomized SVD algorithm \alec{in~\cite{halko2011finding}, presented next,} enables a two-pass coarse grid SVD approximation in which $\tilde{\bm{V}}_k$ is orthogonal.

\subsection{A faster two-pass and single-pass SVD algorithm}
\label{sec:singlepass}
The basic framework for computing an SVD from a matrix sketch presented here is taken from~\cite{halko2011finding}, and shown in Algorithm~\ref{alg:coarsetpsvd}. In this work, instead of using random projection to reduce the dimension of the input matrix, deterministic sketches (with the option of further embedding via random projection) are employed. This is reflected in Step~\ref{step:coarsenproto} of PROTO-TPC, where a Gaussian matrix sketch has been replaced with a deterministic sketch. 
Unlike in PROTO-TPC, the QR decomposition of $\bm{A}_c$ is computed instead of its SVD. This yields the matrix $\bm{B}$ \alec{(shown below in Step~\ref{step:QRSVDB})}. The SVD of $\bm{B}$ then yields a low-rank approximation of $\bm{A}_f$:
\begin{subequations}
\begin{align}
    \bm{A}_c &= \bm{Q}_c\bm{R}_c \label{step:QRSVDcoarse}, \\
    \bm{B} &= \bm{Q}_c^T \bm{A}_f \label{step:QRSVDB}, \\
     \bm{B} &= \bm{U} \bm{S} \bm{V}^T, \\
    \bm{A}_f & \approx \bm{Q}_c\tilde{\bm{U}}_k \tilde{\bm{S}}_k \tilde{\bm{V}}_k^T.
\end{align}
\end{subequations}
\begin{algorithm}[t]
\caption{TPC-SVD $\bm{A} \approx \hat{\bm{A}}_f = \tilde{\bm{U}}_k \tilde{\bm{S}}_k \tilde{\bm{V}}_k^T$ (adapted from~\cite{halko2011finding})} 	\label{alg:coarsetpsvd}
\begin{algorithmic}[1]
\Procedure{TPC-SVD}{$\bm{A}$ $\in \mathbb{R}^{m \times n}$}
\State $\bm{D} \in \mathbb{R}^{n \times n_c} \gets$ deterministic sketch matrix; $n_c > k$
\State $\bm{\Omega} \gets randn(n_c,\ell)$ (Optional random projection matrix); $n_c > \ell > k$
\State $k \gets$ target rank
\State $\bm{A}_c \gets \bm{A}_f\bm{D}$ \label{step:coarsenproto}
\State  $\bm{A}_c = \bm{A}_c \bm{\Omega}$ $\qquad$ \text{(Optional random projection)}
\State $\bm{Q}_c \gets  QR(\bm{A}_c)$ 
\State $\bm{B} \gets \bm{Q}_c^T\bm{A}_f$
\State $\tilde{\bm{U}}$, $\bm{S}$, $\bm{V}
\gets SVD(\bm{B})$
\State $\bm{U} \gets \bm{Q}_c\tilde{\bm{U}}$
\State $\tilde{\bm{U}}_k \gets \bm{U}(:,1:k)$; $\tilde{\bm{S}}_k \gets \bm{S}(1:k,1:k)$; $\tilde{\bm{V}}_k \gets \bm{V}(:,1:k)$
\State $\bm{return}$ $\tilde{\bm{U}}_k$, $\tilde{\bm{S}}_k$, $\tilde{\bm{V}}_k$
\EndProcedure
\end{algorithmic}
\end{algorithm}
\begin{algorithm}[t]
\caption{SPC-SVD $\bm{A}_f \approx \hat{\bm{A}}_f  = \tilde{\bm{U}}_k\tilde{\bm{S}}_k\tilde{\bm{V}}_k^{T}$ (adapted from~\cite{yu2017single})}  \label{alg:coarsespsvd}
\begin{algorithmic}[1]
\Procedure{SPC-SVD}{$\bm{A}_f \in \mathbb{R}^{m \times n}$}
\State $k \gets$ target rank
\State $p \gets$ oversampling parameter
\State $n_c \gets k + p$  coarse grid size
\State $\bm{D} \in \mathbb{R}^{n \times n_c} \gets$ deterministic sketch matrix; $n_c > k$
\State $\ell \gets$ optional random projection dimension; $\ell < n_c$
\State $\bm{\Omega} \gets randn(n_c,\ell)$ (Optional random projection matrix); $n_c > \ell > k$
\State instantiate $\bm{A}_c$
\State $\bm{H} \gets zeros(n,\min(n_c,\ell))$
\State $\bm{while}$ $\bm{A}_f$ is not entirely read through $\bm{do}$
\State $\qquad$ read the next set of rows $\bm{a}_f$ in RAM 
\State $\qquad$ $\bm{a}_c \gets \bm{a}_f\bm{D}$
\State $\qquad$ $\bm{a}_c \gets \bm{a}_c\bm{\Omega}$ $\qquad$ \text{(Optional random projection)} 
\State $\qquad$ $\bm{A}_c \gets [\bm{A}_c; \bm{a}_c]$
\State $\qquad$ $\bm{H} \gets \bm{H} + \bm{a}_f^{T}\bm{a}_c$
\State $\bm{end}$ $\bm{while}$
\State $\bm{Q}_c,\bm{R}_c \gets QR(\bm{A}_c)$ 
\State $\bm{H} \gets \bm{H}\bm{R}_c^{-1}$
\State  $\tilde{\bm{U}}, \bm{S}, \bm{V} \gets SVD(\bm{H}^T)$;
\State $\bm{U} \gets \bm{Q}_c\tilde{\bm{U}}$;
\State $\tilde{\bm{U}}_k \gets \bm{U}(:,1:k)$; $\tilde{\bm{V}}_k \gets \bm{V}(:,1:k)$; $\tilde{\bm{S}}_k \gets \bm{S}(1:k, 1:k)$
\State $\bm{return}$ $\tilde{\bm{U}}_k, \tilde{\bm{S}}_k, \tilde{\bm{V}}_k$ 
\EndProcedure
\end{algorithmic}
\end{algorithm} 
To reduce the number of passes from two \alec{(required above in~\ref{step:QRSVDcoarse} and~\ref{step:QRSVDB})} to one, a single-pass coarse grid algorithm for computing a rank-$k$ approximation of a data matrix, SPC-SVD is presented. The algorithmic framework proposed in~\cite{yu2017single} is used to construct a low-rank approximation in a single-pass over the data matrix; the method is presented in this work as Algorithm~\ref{alg:coarsespsvd}. \alec{The key change made to the algorithm in~\cite{yu2017single} is the sketch used and the omission of matrix blocking and re-orthonormalization in the QR step.}

This algorithm relies on the formation of two sketch matrices in a single-pass over the fine grid data; the procedure for computing these sketches online is given in the while loop in Algorithm~\ref{alg:coarsespsvd}:
\begin{align}
    \bm{A}_c &= \bm{A}_f \bm{D},\\ 
    \bm{H} &= \bm{A}^T_f \bm{A}_c.
\end{align}
The QR decomposition of $\bm{A}_c$ is then used to update $\bm{H}$:
\begin{align}
    \bm{A}_c &= \bm{Q}_c \bm{R}_c, \\
    \bm{H} &= \bm{H} \bm{R}_c^{-1}, \\
    &= \bm{A}^T_f \bm{A}_c  \bm{R}_c^{-1},
    \\
    &= \bm{A}^T_f \bm{Q}_c\bm{R}_c\bm{R}_c^{-1}, \\
    &= \bm{A}^T_f\bm{Q}_c \label{step:Rinvfullrank}.
\end{align}
Computing the SVD of $(\bm{H}\bm{R}_c^{-1})^T$, 
\begin{equation}
    (\bm{H}\bm{R}_c^{-1})^T = \bm{Q}_c^T\bm{A}_f = \bm{U}\bm{S}\bm{V}^T,
\end{equation}
which yields a rank-$k$ approximation for $\bm{A}_f$ via
\begin{equation}
    \bm{A}_f \approx \bm{Q}_c\bm{Q}_c^T\bm{A}_f = \bm{Q}_c(\bm{H}\bm{R}_c^{-1})^T \approx \bm{Q}_c\tilde{\bm{U}}_k\tilde{\bm{S}}_k\tilde{\bm{V}}_k^T.
\end{equation}
Having outlined SPC-SVD, the main theoretical result for this section is now presented.
\begin{theorem}
\label{thm:SPCSVD}
Let $\bm{A}_c \in \mathbb{R}^{m \times n_c}$ be a coarse grid sketch with matrices $\bm{M}$ and $\bm{E}_I$ as in Definition~\ref{def:interpop}. Then, a rank-$k$ approximation $\hat{\bm{A}}_f$ for a fine grid data matrix $\bm{A}_f \in \mathbb{R}^{m \times n}$ generated using SPC-SVD \alec{(Algorithm~\ref{alg:coarsespsvd})} satisfies 
\begin{align}
    \Vert \bm{A}_f - \hat{\bm{A}}_f \Vert_2 \leq \sigma_{f,k+1} + \Vert \bm{E}_I \Vert_2.
\end{align}
\end{theorem}
\begin{proof}
See Section~\ref{sec:proofspcsvd}.
\end{proof}
This error bound indicates that the sub-optimality of SPC-SVD is bounded above by the error incurred from mapping the coarse grid data on the fine grid. 
\begin{remark}
This error bound can also be applied to TPC-SVD (Algorithm~\ref{alg:coarsetpsvd}). 
\end{remark}

\section{A single-pass interpolative decomposition algorithm}
\label{sec:SPCID}
\begin{algorithm}[t]
\caption{TPC-ID $\bm{A}_f \approx \hat{\bm{A}}_f = \bm{P}_c \bm{A}_f(\mathcal{I}_c,:)$~\cite{dunton2020pass}} 	\label{alg:subid}
\begin{algorithmic}[1]
\Procedure{TPC-ID}{$\bm{A}$ $\in \mathbb{R}^{m \times n}$}
\State $\bm{D} \in \mathbb{R}^{n \times n_c} \gets$ deterministic sketch matrix
\State $\bm{\Omega} \gets randn(n_c,\ell)$ (Optional random projection matrix); $n_c > \ell > k$
\State $k \gets$ target rank
\State $\bm{A}_c \gets \bm{A}_f\bm{D}$
\State  $\bm{A}_c = \bm{A}_c \bm{\Omega}$ $\qquad$ \text{(Optional random projection)}
\State  $\bm{P}_c, \mathcal{I}_c \gets ID(\bm{A}_c^T)$
\State $\bm{P}_c \gets \bm{P}^T_c$
\State $\bm{return}$ $\bm{P}_c$, $\bm{A}_f(\mathcal{I}_c,:)$
\EndProcedure
\end{algorithmic}
\end{algorithm}

\begin{algorithm}[t]
\caption{SPC-ID $\bm{A}_f \approx \hat{\bm{A}}_f = \bm{P}_c \bm{A}_c(\mathcal{I}_c,:)\bm{T}_r$} 	\label{alg:spcid}
\begin{algorithmic}[1]
\Procedure{SPC-ID}{$\bm{A}$ $\in \mathbb{R}^{m \times n}$}
\State $\bm{D} \in \mathbb{R}^{n \times n_c} \gets$ deterministic sketch matrix
\State $\bm{\Omega} \gets randn(n_c,\ell)$ (Optional random projection matrix); $n_c > \ell > k$
\State $k \gets$ target rank
\State $r \gets$ lifting operator rank; $r \geq k$
\State $\bm{H} \gets zeros(n,l)$
\State $\bm{while}$ $\bm{A}_f$ is not entirely read through $\bm{do}$
\State $\qquad$ read the next row into RAM $\bm{a}_f$
\State $\qquad$ $\bm{a}_c \gets \bm{a}_f\bm{D}$
\State $\qquad$ $\bm{A}_c \gets [\bm{A}_c; \bm{a}_c]$
\State $\qquad$ $\bm{H} \gets \bm{H} + \bm{a}_f^{T}\bm{a}_c$
\State $\bm{end}$ $\bm{while}$
\State $\bm{U}_c, \bm{S}_c, \bm{V}_c \gets  SVD(\bm{A}_c)$ 
\State $\tilde{\bm{U}}_r \gets \bm{U}_c(:,1:r)$
\State $\bm{T}_r \gets \bm{V}_c\bm{S}_c^{+}\bm{U}^T_c\tilde{\bm{U}}_r\tilde{\bm{U}}^T_r \bm{U}_c \bm{S}_c^{+}\bm{V}_c^T \bm{H}^T $
\State $\bm{delete}$ $\bm{H}, \tilde{\bm{U}}_r, \bm{U}_c, \bm{S}_c, \bm{V}_c $ $\qquad$ \text{(If necessary)}
\State  $\bm{A}_{c} = \bm{A}_c \bm{\Omega}$ $\qquad$ \text{(Optional random projection)}
\State  $\bm{P}_c, \mathcal{I}_c \gets ID(\bm{A}_c)$ \text{or} $\bm{P}_c, \mathcal{I}_c \gets ID(\bm{U}_c(:,1:k))$
\State $\bm{P}_c \gets \bm{P}^T_c$
\State $\bm{return}$ $\bm{P}_c$, $\bm{A}_c(\mathcal{I}_c,:)$, $\bm{T}_r$
\EndProcedure
\end{algorithmic}
\end{algorithm}
The row interpolative decomposition (row ID) represents a matrix $\bm{A}_f \in \mathbb{R}^{m \times n}$ as a product of a subset of its rows $ \bm{A}_f(\mathcal{I}_f,:) \in \mathbb{R}^{k \times n}$ and a coefficient matrix $\bm{P}_f \in \mathbb{R}^{m \times k}$ such that 
\begin{equation}
\label{eqn:ID}
\bm{A}_f \approx \bm{P}_f\bm{A}_f(\mathcal{I}_f,:),
\end{equation}
with  $\mathcal{I}_f\subseteq\{1,\dots,m\}$, $\vert\mathcal{I}_f\vert=k$. Further, $\bm{P}_f(\mathcal{I}_f,:) = \bm{I}$, with $\bm{I} \in \mathbb{R}^{k \times k}$ an identity matrix. Row ID earns its name from the fact that it {\it interpolates} the rows of $\bm{A}_f$ using a basis consisting of a subset of its rows.

To obtain the row ID, the rank-$k$ column-pivoted QR decomposition of $\bm{A}_f^T$ is employed,
\begin{equation}
\bm{A}_f^T\bm{Z} \approx \bm{QR},
\end{equation}
where $\bm{Z}\in\mathbb{R}^{m\times m}$ is a permutation matrix, $\bm{Q}\in\mathbb{R}^{n\times k}$ has orthonormal columns, and $\bm R\in\mathbb{R}^{k\times m}$ is upper triangular. Separating $\bm{R}$ into two sub-matrices $\bm{R} = \left[\bm{R}_{1} \hspace{2pt}\vert\hspace{2pt} \bm{R}_{2} \right]$, where $\bm{R}_{1}\in\mathbb{R}^{k\times k}$ and $\bm{R}_{2}\in\mathbb{R}^{k\times(m-k)}$, and solving $\bm{R}_{2}\approx\bm{R}_{1}\bm{C}$ \alec{for $\bm{C}$ via least squares} yields the final approximation
\begin{equation}
\bm{A}_f^T \approx \bm{Q}\bm{R}_{1} \left[\bm{I} \hspace{2pt}\vert\hspace{2pt} \bm{C}\right]\bm{Z}^{T} 
= \bm{A}_f^T(:,\mathcal{I}_f)\left[\bm{I} \hspace{2pt}\vert\hspace{2pt} \bm{C} \right]\bm{Z}^{T}
= \bm{A}_f^T(:,\mathcal{I}_f)\bm{P}^T,
\end{equation}
and hence the equation (\ref{eqn:ID}). From~\cite{cheng2005compression}, there exists a rank-$k$ row ID of any $m \times n$ real matrix such that 
\begin{equation}
\label{eqn:id_bound_orig}
\Vert \bm{A}_f - \bm{P}_f\bm{A}_f(\mathcal{I}_f,:)\Vert_2 \leq \sqrt{1 + k(m-k)}\sigma_{f,k+1},
\end{equation} 
\alec{where $\sigma_{f,k+1}$ is the $(k+1)^{th}$ largest singular value of $\bm{A}_f$.}
Depending on the QR scheme used, row ID may feature a worst-case computational complexity of $\mathcal{O}(mn\min(m,n))$. The complexity is closer to $\mathcal{O}(mnk)$, in particular when using the modified Gram-Schmidt QR scheme from~\cite{golub2012matrix}.

To reduce the runtime of row ID, the index vector and coefficient matrix may be computed on a sketch $\bm{A}_c$ of the matrix $\bm{A}_f$ instead. First, the sketch $\bm{A}_c = \bm{A}_f\bm{D}$ and its rank-$k$ row ID is computed
\begin{equation}
     \bm{A}_c \approx  \bm{P}_c \bm{A}_c(\mathcal{I}_c,:). 
\end{equation}
The index vector $\mathcal{I}_c$ and coefficient matrix $\bm{P}_c$ - computed on the coarse grid data matrix $\bm{A}_c$ - may then be used to approximate the fine grid data matrix $\bm{A}_f$ by {\it lifting} the row skeleton of $\bm{A}_c$ to $\bm{A}_f$ 
\begin{equation}
     \bm{A}_f \approx \bm{P}_c \bm{A}_f(\mathcal{I}_c,:). 
\end{equation}
This procedure constitutes \alec{the two-pass coarse grid ID (TPC-ID)} (Algorithm 5 in~\cite{dunton2020pass} and Algorithm~\ref{alg:subid} in this work). The algorithm requires a first pass over $\bm{A}_f$ to form the sketch $\bm{A}_c$ and compute $\mathcal{I}_c$ and $\bm{P}_c$, and a second pass to obtain the \alec{row} skeleton $\bm{A}_f(\mathcal{I}_c,:)$.

The single-pass ID in~\cite{dunton2020pass} skips the second pass over $\bm{A}_f$ by interpolating the coarse grid \alec{row} skeleton $\bm{A}_c(\mathcal{I}_c,:)$ on the fine grid to obtain an {\it approximation} to the fine grid \alec{row} skeleton, here denoted $\hat{\bm{A}}_f(\mathcal{I}_c,:)$. Let $\bm{M}$ be an operator which maps the coarse grid data on the fine grid (see Definition~\ref{def:interpop}). Then, the single-pass ID approximation is given by 
\begin{equation}
    \bm{A}_f \approx \bm{P}_c\hat{\bm{A}}_f(\mathcal{I}_c,:) = \bm{P}_c\bm{A}_c(\mathcal{I}_c,:)\bm{M}. 
\end{equation}
Error bounds and computational complexities for TPC-ID and this single-pass ID are provided in~\cite{dunton2020pass}.

A major limitation to this existing single-pass ID algorithm is the large errors associated with the interpolation.
To this end, an improved single-pass algorithm for computing the row ID, SPC-ID, is presented.  

SPC-ID (Algorithm~\ref{alg:spcid}) is a generalization of the method presented in~\cite{dunton2020pass}, is based on the bi-fidelity approximation algorithm presented in~\cite{hampton2018practical}, and utilizes the sketching procedure from the single-pass SVD algorithm presented in~\cite{yu2017single}. \alec{SPC-ID is able to compute a low-rank approximation in a single-pass by forming the {lifting operator} $\bm{T}_r$ with $\bm{A}_f \approx \bm{A}_c \bm{T}_r$ in a single-pass over the input}. First, as in Algorithm~\ref{alg:coarsespsvd}, $\bm{A}_f$ is read row-by-row into working memory to form the matrices
\begin{subequations}
\begin{align}
    \bm{A}_c &= \bm{A}_f\bm{D}, \\
    \bm{H} &= \bm{A}_f^T\bm{A}_c.
\end{align}
\end{subequations}
\alec{The {\it lifting operator} $\bm{T}_r$, which maps the coarse grid data on the fine grid, can now be constructed.} To form $\bm{T}_r$, the matrix $\tilde{\bm{U}}_r$ is introduced:
\begin{subequations}
\begin{align}
\bm{A}_c &= \bm{U}_c\bm{S}_c\bm{V}_c^T,\\
\tilde{\bm{U}}_r &= \bm{U}_c(:,1:r),
\end{align}
\end{subequations}
\alec{where $r$ can be interpreted as a regularization parameter which may take on any value between the target rank $k$ and the rank of the coarse grid matrix $\bm{A}_c$.
Now, the lifting operator $\bm{T}_r$ \alec{is constructed:}
\begin{align}
\label{eq:augmentliftingop}
\bm{T}_r
&= \bm{V}_c\bm{S}_c^{+}\bm{U}^T_c\tilde{\bm{U}}_r\tilde{\bm{U}}^T_r \bm{U}_c \bm{S}_c^{+}\bm{V}_c^T\bm{H}^T .
\end{align}
Now, because $\bm{H}^T = \bm{A}_c^T\bm{A}_f$, we have,
\begin{align}
    \bm{H}^T = \bm{V}_c\bm{S}_c\bm{U}_c^T\bm{A}_f.
\end{align}
And thus, 
\begin{subequations}
\begin{align}
\bm{T}_r
&= \bm{V}_c\bm{S}_c^{+}\bm{U}^T_c\tilde{\bm{U}}_r\tilde{\bm{U}}^T_r \bm{U}_c \bm{S}_c^{+}\bm{V}_c^T\bm{V}_c\bm{S}_c\bm{U}_c^T\bm{A}_f, \\
&= \bm{V}_c\bm{S}_c^{+}\bm{U}^T_c\tilde{\bm{U}}_r\tilde{\bm{U}}^T_r \bm{U}_c \bm{S}_c^{+}\bm{S}_c\bm{U}_c^T\bm{A}_f, \\
&= \bm{V}_c\bm{S}_c^{+}\bm{U}^T_c\tilde{\bm{U}}_r\tilde{\bm{U}}^T_r \bm{U}_c\bm{U}_c^T\bm{A}_f, \\
&= \bm{V}_c\bm{S}_c^{+}\bm{U}^T_c\tilde{\bm{U}}_r\tilde{\bm{U}}^T_r\bm{A}_f, \\
&= \bm{A}_c^+\tilde{\bm{U}}_r\tilde{\bm{U}}^T_r\bm{A}_f.
\end{align}
\end{subequations}
Note the similarity between this operator and the classic least-squares solution of $\bm{A}_c\bm{X} \approx \bm{A}_f$ for $\bm{X}$, $\bm{A}_c^+\bm{A}_f$. The key distinction is the inclusion of the projection $\tilde{\bm{U}}_r\tilde{\bm{U}}^T_r$, wherein $r$ acts as a regularization parameter. The addition of this regularization echoes work done using the truncated SVD to solve least-squares problems~\cite{hansen1987truncatedsvd}. Introducing this regularization has benefits when the coefficient matrix, e.g., $\bm{A}_c$, is ill-conditioned and has well-defined numerical rank, which is precisely the case in this work. In particular, the results in~\cite{hansen1987truncatedsvd} show the truncated SVD achieves accuracy on par with the classically regularized least-squares approach with an optimally selected regularization parameter. These observations further support this specific construction of the operator $\bm{T}_r$.}

\alec{The ID of $\bm{A}_c$ is then computed, yielding the final approximation:}
\begin{align}
    \bm{A}_c & \approx \bm{P}_c \bm{A}_c(\mathcal{I}_c,:), \\
    \bm{A}_f & \approx  \bm{P}_c \bm{A}_c(\mathcal{I}_c,:) \bm{T}_r.
\end{align}

\begin{remark}
SPC-ID as presented requires computing the ID and SVD of $\bm{A}_c$. Fast algorithms for converting an SVD into an ID or vice versa are available~\cite{liberty2007randomized,halko2011finding}. These approaches can be used to speed up SPC-ID. Further exploration of similar ideas is left to a future work.
\end{remark}

Following the analysis presented in~\cite{hampton2018practical}, the parameter $\epsilon(\tau)$ is defined to be
\begin{equation}
    \epsilon(\tau) = \lambda_{\max} (\bm{A}_f\bm{A}_f^T - \tau \bm{A}_c \bm{A}_c^T ),
    \label{eqn:epstau_main_test}
\end{equation}
\alec{where $\tau > 0$ and $\lambda_{\max}$ is the maximum eigenvalue of a given matrix. This parameter measures how well the scaled Gramian of the coarse grid data matrix approximates its fine grid counterpart. This provides the groundwork for the key theoretical result from this section.} 

\begin{theorem}
\label{thm:SPCIDbound}
Let $\hat{\bm{A}}_f$ be the rank-$k$ approximation of $\bm{A}_f$ generated by SPC-ID, $\bm{A}_c$ the coarse grid data matrix, $\bm{P}_c \bm{A}_c(\mathcal{I}_c,:)$ its rank-$k$ row ID approximation, $\sigma_{c,j}$ the $j$th largest singular value of $\bm{A}_c$, and $\epsilon(\tau)$ defined as in (\ref{eqn:epstau_main_test}). Then, 
\begin{align}
    \Vert \bm{A}_f -  \hat{\bm{A}}_f \Vert_2
    &\leq \min_{\tau,k \leq r \leq \text{rank}(\bm{A}_c)} \left(\epsilon(\tau) + \tau \sigma_{c,r+1}^2 \right)^{1/2} + \Vert \bm{A}_c - \bm{P}_c \bm{A}_c(\mathcal{I}_c,:) \Vert_2 \left(\tau + \epsilon(\tau) \sigma_{c,r}^{-2}\right)^{1/2}. 
\end{align}
\end{theorem}
\begin{proof}
See Section~\ref{sec:proofSPCID}.
\end{proof}

\alec{The first term indicates the importance that $\epsilon(\tau)$, $\tau$, and $\sigma_{c,r+1}$ are all small. The second term suggests an accurate approximation if the low-rank ID approximation error on the coarse grid, $\epsilon(\tau)$, and $\tau$ are small. It is also critical that $\sigma_{c,r}^{-2}$ is not too large; i.e., that $\sigma_{c,r}$ is not too small. This last observation, coupled with the suggestion that $\sigma_{c,r+1}$ should be small, indicates that identifying a value of $r$ corresponding to a steep dropoff in the singular value decay of $\bm{A}_c$ will therefore help ensure an accurate SPC-ID approximation.}

The error bound given as Theorem~\ref{thm:SPCIDbound} is nearly identical to Theorem 1 of~\cite{hampton2018practical}, which bounds the error of TPC-ID. Notably, Theorem~\ref{thm:SPCIDbound} is a potentially smaller bound \alec{(the first term has a smaller coefficient than that of the bound in~\cite{hampton2018practical})}, even though SPC-ID requires one fewer pass than TPC-ID. Another important distinction is that the rank \alec{$r$} of the lifting operator is minimized over values between $k$ and $\text{rank}(\bm{A}_c)$ in Theorem~\ref{thm:SPCIDbound}; in Theorem 1 of~\cite{hampton2018practical}, $r$ is allowed to vary between $1$ and $\text{rank}(\bm{A}_c)$. 
\begin{remark}
The rank of $\bm{T}_r$ is not necessarily the same as the rank of the ID approximation of $\bm{A}_c$. For consistency, $r \geq k = \vert \mathcal{I}_C \vert$; the approximation rank would otherwise be reduced to $r = \min(r,k) < k$.
\end{remark}

\subsection{Single-pass error estimation for SPC-ID}
\begin{algorithm}[t]
\caption{SPC-ID-ERR $\hat{\epsilon}(\tau) \approx \epsilon(\tau)$ (Algorithm 1 in~\cite{hampton2018practical})}  \label{alg:onlineerrestimatespcid}
\begin{algorithmic}[1]
\Procedure{SPC-ID-ERR}{$\bm{A}_f \in \mathbb{R}^{m \times n}$}
\State $\bm{D} \in \mathbb{R}^{n \times n_c} \gets$ deterministic sketch matrix
\State $m_c \gets$ target temporal dimension
\State $c =  m/m_c$
\State $\bm{for}$ $i = 1:m$
\State $\qquad$ read the $i^{th}$ row  $\bm{a}_f$ into RAM
\State $\qquad$ $\bm{if}$ $\mod(i,c) == 0$
\State $\qquad$ $\qquad$ $\bm{a}_c \gets \bm{a}_f\bm{D}$
\State $\qquad$ $\qquad$ $\bm{B}_f \gets [\bm{B}_f; \bm{a}_f]$
\State $\qquad$ $\qquad$ $\bm{B}_c \gets [\bm{B}_c; \bm{a}_c]$
\State $\qquad$ $\bm{end}$ $\bm{if}$
\State $\bm{end}$ $\bm{for}$
\State $\hat{\epsilon}(\tau) =  \min_{\tau} c \lambda_{\max} (\bm{B}_f\bm{B}_f^T - \tau \bm{B}_c \bm{B}_c^T )$
\EndProcedure
\end{algorithmic}
\end{algorithm} 

A RAM-efficient single-pass method for estimating the error bound presented in Theorem~\ref{thm:SPCIDbound} is provided. 
Following the structure of Algorithm 1 in~\cite{hampton2018practical}, the matrices $\bm{A}_f$ and $\bm{A}_c$ are \alec{sub-sampled} as they are read into RAM, indexed by $\mathcal{J}$ with 
$\mathcal{J} \subseteq\{1,\dots,m\}$ and $\vert\mathcal{J}\vert=m_c$, \alec{where $m_c$ must be chosen large enough to ensure an accurate estimate, but small enough to avoid overuse of working memory}.
\begin{align}
    \bm{B}_f &= \bm{A}_f(\mathcal{J},:), \\
    \bm{B}_c &= \bm{A}_c(\mathcal{J},:).
    \label{eqn:epstauestimate}
\end{align}
The procedure hinges on computing the following estimate for $\epsilon(\tau)$. Let 
\begin{equation}
    \hat{\epsilon}(\tau) = c \lambda_{\max} (\bm{B}_f\bm{B}_f^T - \tau \bm{B}_c \bm{B}_c^T ),
    \label{eqn:epstau}
\end{equation}
\alec{Using (\ref{eqn:epstau_main_test}) error estimation} and low-rank approximation may be achieved in parallel with acceptable RAM usage (see Table~\ref{tab:algsummary} for more details). 


\section{Coarse grid power iteration}
\label{sec:coarsegridpoweriteration}
\begin{algorithm}[t]
\caption{C-PWR}  \label{alg:coarsespsvdpower}
\begin{algorithmic}[1]
\Procedure{C-PWR}{$\bm{A}_f \in \mathbb{R}^{m \times n}$}
\State $\bm{D} \in \mathbb{R}^{n \times n_c} \gets$ deterministic sketch matrix
\State $\mathcal{I} \gets$ sub-sampling index vector to form approximate Gram matrix
\State $\bm{A}_c \gets \bm{A}_f(:,\mathcal{I})$
\State $q \gets$ number of power iterations
\State $\bm{Q}_{np} \gets QR(\bm{A}_f\bm{D})$
\State $G_q \gets \bm{A}_f\bm{D}$
\State $\bm{for}$ $i = 1:q$
\State $\qquad$ $\bm{G}_q \gets \bm{A}_c^T \bm{G}_q$
\State \text{(Optional)}  $\qquad$ $\bm{Q} \gets QR(\bm{G}_q)$  $\qquad$ 
\State \text{(Optional)} $\qquad$ $\bm{G}_q \gets \bm{Q}$   $\qquad$ 
\State $\qquad$ $\bm{G}_q \gets \bm{A}_c \bm{G}_q$ 
\State  \text{(Optional)} $\qquad$ $\bm{Q} \gets QR(\bm{G}_q)$  $\qquad$ 
\State  \text{(Optional)} $\qquad$ $\bm{G}_q \gets \bm{Q}$ $\qquad$ 
\State $\bm{end}$ $\bm{for}$
\State $\bm{return}$ $\bm{G}_q$, $\bm{Q}_{np}$
\EndProcedure
\end{algorithmic}
\end{algorithm} 

\alec{In low-rank approximations generated using randomized sketching approaches where singular value decay is slow, the final accuracy can be degraded~\cite{martinsson2010normalized,halko2011finding}. In order to mitigate this, a power iteration method for computing a rank-$k$ SVD approximation is presented in~\cite{halko2011finding}.}
With $\bm{A}_f$ the input matrix, \alec{standard} power iteration excluding the right multiplication by a sketching operator entails computing the product 
\begin{equation}
\bm{A}^q_f = (\bm{A}_f\bm{A}_f^T)^q \bm{A}_f,
\end{equation}
where $q$ denotes the number of power iterations. \alec{Following the formation of $\bm{A}^q_f$, its SVD is given by:}  

{\begin{align}
    \bm{A}^q_f &=  \bm{U}_{f} \bm{S}_{q,f} \bm{V}_{f}^T,\\
    \bm{S}_{q,f} &= \bm{S}_f^{2q+1} \label{eqn:pwritersingvals}.
\end{align}}

The equation (\ref{eqn:pwritersingvals}) indicates that the singular values of $\bm{A}^q_f$ are $\sigma_{f,i}^{2q+1}$, where the $\sigma_{f,i}$ are the singular values of $\bm{A}_f$. The left and right singular vectors of $\bm{A}_f$, $\bm{U}_f$ and $\bm{V}_f$, respectively, are left unchanged. This enables improvement in accuracy while preserving key geometry from the input matrix.

\alec{A significant drawback of this approach is that} it requires $2q+2$ passes over the data matrix~\cite{halko2011finding}. 
To reduce the number of passes from $2q+2$ to 1, C-PWR (Algorithm~\ref{alg:coarsespsvdpower}), a single-pass power iteration algorithm, is presented.
First, an approximation to $\bm{A}^q_f$ is formed in a single pass:
\begin{equation}
 \bm{A}^q_b =    (\tau\bm{A}_c\bm{A}_c^T)^q \bm{A}_f,
\end{equation}
where $\bm{A}_c = \bm{A}_f(:,\mathcal{I})$, and $\tau$ is chosen to scale the entries of the coarse Gram matrix $\bm{A}_c\bm{A}_c^T$ to approximate those of $\bm{A}_f\bm{A}_f^T$; \alec{the optimal value of $\tau$ will minimize $\rho(\tau)$ (see Definition~\ref{def:rho_tau} below).}

\alec{Algorithm~\ref{alg:coarsespsvdpower} is used to form the QR decomposition of the input matrix, which is then used to compute a $k$-rank SVD in the follow steps:}

\begin{align}
    \bm{A}^q_b &= \bm{Q}_b\bm{R}_b, \label{step:QRstepCPWR}\\
    \bm{B} &= \bm{Q}_b^T\bm{A}_f, \\
    \bm{B} &= \bm{U}\bm{S}\bm{V}^T, \\
    \hat{\bm{A}}_f &= \bm{Q}_b\tilde{\bm{U}}_k\tilde{\bm{S}}_k\tilde{\bm{V}}^T_k.
\end{align}
In practice, the QR decomposition in the equation (\ref{step:QRSVDcoarse}) is computed on a sketch of the matrix $\bm{A}^q_b$:
\begin{align}
    \bm{A}^q_b\bm{D} = \bm{Q}_b\bm{R}_b. \label{eqn:CPWR_sketch}
\end{align}

To improve performance, the authors of~\cite{halko2011finding} rely on re-orthonormalization in their power iteration scheme. Re-orthonormalization may also be used in C-PWR. That is, instead of explicitly computing the product $(\tau\bm{A}_c\bm{A}_c^T)^q \bm{A}_f\bm{D}$, the columns are repeatedly orthonormalized via a QR decomposition before multiplying by $\bm{A}_c$ or $\bm{A}_c^T$ (see the steps marked `Optional' in Algorithm~\ref{alg:coarsespsvdpower}). Re-orthornormalization improves the accuracy by reducing error \alec{accrued} when directly evaluating $(\tau\bm{A}_c\bm{A}_c^T)^q \bm{A}_f\bm{D}$~\cite{halko2011finding}.

\begin{remark}
When the coarse grid data matrix $\bm{A}_c = \bm{A}_f(:,\mathcal{I})$ is formed, it is necessary that $\vert \mathcal{I} \vert > k$.
\end{remark}
To analyze the error incurred using C-PWR, the following definition is provided.

\begin{definition}
\label{def:rho_tau}
Let $\rho(\tau) =  \Vert \bm{A}_f\bm{A}_f^T -  \tau\bm{A}_c\bm{A}_c^T\Vert_2$. 
\end{definition}
The following theorem provides an upper bound on the error due exclusively to C-PWR and excluding the contribution of the sketching step in the equation (\ref{eqn:CPWR_sketch}).
\begin{theorem}
\label{thm:coarsepwritr}
A rank-$k$ approximation $\hat{\bm{A}}_f$ for a fine grid data matrix $\bm{A}_f$ generated using C-PWR without sketching satisfies 
\begin{align}
    \label{eq:pwritrgenbound}
    \Vert \bm{A}_f - \hat{\bm{A}}_f \Vert_2 &\lesssim \sigma_{f,k+1} + C(\tau,q)^{1/(2q+1)}, \\
    C(\tau,q) &= q \tau^{q-1}\Vert \bm{A}_c \Vert_2^{2q-2} \Vert \bm{A}_f \Vert_2 \rho(\tau) + \rho^2(\tau),
\end{align}
where $q$ is the number of power iterations used and $\rho(\tau)$ is defined in Definition~\ref{def:rho_tau}. Moreover, the best possible approximation satisfies
\begin{align}
    \label{eq:bestpwritrapprox}
    \Vert \bm{A}_f - \hat{\bm{A}}_f \Vert_2 \lesssim  \sigma_{f,k+1} +\min_{\tau,q} C(\tau,q)^{1/(2q+1)}.
\end{align}
\end{theorem}
\begin{proof}
See Section~\ref{sec:proofcoarsepoweriteration}.
\end{proof}

\alec{The result of Theorem~\ref{thm:coarsepwritr} demonstrates that the sub-optimality $C^{1/(2q+1)}$ depends on: $\rho(\tau)$, i.e, how well $\tau\bm{A}_c\bm{A}_c^T$ approximates $\bm{A}_f\bm{A}_f$; how large $\tau$ is, which is directly determined by how aggressively the columns of $\bm{A}_f$ are sub-sampled; and the number of power iterations $q$. }


\alec{In the present work,} sketches of the form $\bm{A}_c = \bm{A}_f(:,\mathcal{I})$ are used for approximating Gramians in \alec{the} power iteration scheme. Sketches such as well as random projections -- in particular Gaussian matrix-based Johnson-Lindenstrauss transforms -- performed worse empirically in initial tests and were consequently omitted. The theoretical analysis of this \alec{observation} is left to \alec{a} future work.

\subsection{Summary of algorithms}
\label{sec:algsummary}
\begin{table}[t]
\centering
\tabcolsep7pt\begin{tabular}{lccc}
\hline
Method & Computational complexity &  Max RAM Usage & Number of Passes\\
\hline
PROTO-TPC & $\mathcal{O}(mnn_c)$ &  $(m+n)n_c$ & 2\\
TPC-SVD~\cite{halko2011finding}  & $\mathcal{O}(mnn_c)$ & $(m+2n + n_c + 1)n_c$ & 2 \\
SPC-SVD  & $\mathcal{O}(mnn_c)$ & $(m + n + n_c + 1)n_c + \max(m,n)n_c$ &  1 \\
TPC-ID~\cite{dunton2020pass} & $\mathcal{O}(mnn_c)$ & $mn_c + k(m+n_c)$ & 2 \\
SPC-ID & $\mathcal{O}(mnn_c)$ & $(m + n + n_c + 1)n_c + \max(m,n)n_c$ & 1 \\
SPC-ID-ERR & $\mathcal{O}(mnn_c + m_c^2 n + m n_c^2)$ & $n(m_c + n_c) + m_c n_c$ & 1 \\
C-PWR & $\mathcal{O}(mnn_c + q(m^2n_c + mn_c^2))$ & $3mn_c + \max(m,n)n_c$ & 1 \\
\hline
\end{tabular}
\caption{Left column: computational complexity of the 7 algorithms described in this work. Center column: maximum RAM usage, defined as the maximum number of matrix entries required to be stored in working memory throughout execution of each algorithm. Right column: number of passes over the input matrix required by each method.}
\label{tab:algsummary}
\end{table}
The algorithms presented in this worked are summarized in Table~\ref{tab:algsummary}. In the second column from the left, the computational complexity of each method is provided. PROTO-TPC, TPC-SVD, SPC-SVD, TPC-ID, and SPC-ID (Algorithms~\ref{alg:protosvd},~\ref{alg:coarsetpsvd},~\ref{alg:coarsespsvd},~\ref{alg:subid}, and~\ref{alg:spcid}, respectively), all have complexities which depend asymptotically on $mnn_c$ due to the sketching step(s) in each method. SPC-ID-ERR (Algorithm~\ref{alg:onlineerrestimatespcid}) has complexity similar to that of SPC-ID, though slightly large due to the formation of the approximate Gram matrices. C-PWR (Algorithm~\ref{alg:coarsespsvdpower}) has the same leading term; its complexity also depends on the \alec{number $q$ of power iterations.} 

Of the single-pass methods, C-PWR requires the most RAM due to the storage of sketch matrices for power iteration, sketching, and orthonormalization. SPC-SVD and SPC-ID are more RAM efficient if $n < 2m$. Finally, the online error estimation procedure SPC-ID-ERR requires the least RAM, but in practice is meant to be run simultaneously with SPC-ID.

\section{Numerical experiments: SPC-SVD}
\label{sec:numericalexperiments}
In the following two sections, the error and wall-clock runtime of TPC-SVD (Algorithm~\ref{alg:coarsetpsvd}), SPC-SVD (Algorithm~\ref{alg:coarsespsvd}), TPC-ID (Algorithm~\ref{alg:subid}), SPC-ID (Algorithm~\ref{alg:spcid}), and C-PWR (Algorithm~\ref{alg:coarsespsvdpower}) are measured. All results are generated in MATLAB using a conventional laptop with an Intel i7 @1.80GHz CPU and 16 GB of RAM. 

Approximation error for the $k$-rank matrix \alec{$\hat{\bm{A}}_{f}$} is reported in terms of the Relative Frobenius Error: 
\begin{equation}
\label{eq:relativefrobeniuserror}
    \text{Relative Frobenius Error} = \frac{\Vert \bm{A}_f - \alec{\hat{\bm{A}}_{f}} \Vert_F}{\Vert \bm{A}_f \Vert_F}.
\end{equation}
For low-rank approximations generated using C-PWR (Algorithm~\ref{alg:coarsespsvdpower}), errors are reported in terms of a metric referred to herein as the Max Ratio of Error to Oracle (MREO). Let $\bm{A}_{f,k}$ be the best rank-$k$ approximation of the matrix $\bm{A}_f$ given by the truncated rank-$k$ SVD, i.e., the oracle solution. Then, 
\begin{subequations}
\begin{align}
\label{eq:maximumsuboptimality}
    \text{MREO} &= \max_{\text{all trials}} \frac{\Vert \bm{A}_f - \alec{\hat{\bm{A}}_{f}}\Vert_F}{\Vert \bm{A}_f - \bm{A}_{f,k} \Vert_F}, \\
\label{eq:eckartyoung}
    \Vert \bm{A}_f - \bm{A}_{f,k} \Vert_F &= \left(\sum_{i=k+1}^{\min(m,n)} \sigma_{f,i}^2\right)^{1/2},
\end{align}
\end{subequations}
where (\ref{eq:eckartyoung}) is the Eckart-Young theorem~\cite{eckart1936approximation}. The MREO (\ref{eq:maximumsuboptimality}) is a measure of the robustness of a randomized algorithm. A low value indicates that C-PWR improves low-rank approximations generated using randomized schemes; this is precisely the purpose of power iteration as presented in~\cite{martinsson2010normalized,halko2011finding}.

In this section, hybrid deterministic-randomized sketches, \alec{i.e., sketches with a deterministic dimension reduction step followed by a random projection step}, are implemented. When randomized sketches are used in the test cases involving C-PWR, they are composed with a direct injection deterministic sketch of the fine grid data (see the `Optional' random sketching steps in Algorithms~\ref{alg:coarsetpsvd},~\ref{alg:coarsespsvd},~\ref{alg:subid}, and~\ref{alg:spcid}). SPC-SVD+C-PWR and TPC-SVD+C-PWR indicate that C-PWR is used in conjunction with SPC-SVD or TPC-SVD, respectively. Finally, the coarsening factor of a deterministic sketch is defined to be
\begin{equation}
\label{eq:coarseningfactor}
    \text{Coarsening Factor} = n/n_c,
\end{equation}
where $n$ is the size of the fine grid snapshot and $n_c < n$ is the size of the corresponding coarse grid snapshot. In C-PWR test cases, the oversampling parameter is set equal to the target rank, while in hybrid deterministic-randomized sketches, the oversampling parameter is set to $10$. One coarse grid power iteration is used in all cases; the MREO is evaluated over $100$ independent trials. In the tests involving C-PWR, the fully randomized SVD is not used as a benchmark because hybrid deterministic-randomized sketches were shown to perform as well for low target rank values.
\begin{remark}
In some of the following numerical experiments, single-pass algorithms are slower than their two-pass counterparts in terms of wall-clock runtime. It is critical to emphasize that in scenarios where a second pass entails a second PDE solve or loading a prohibitively large matrix from disk to RAM, these differences in wall-clock runtime become negligible.
\end{remark}
\subsection{NACA-4412 airfoil \alec{data}}
\begin{figure}[t]
\centering
\includegraphics[width=0.5\linewidth]{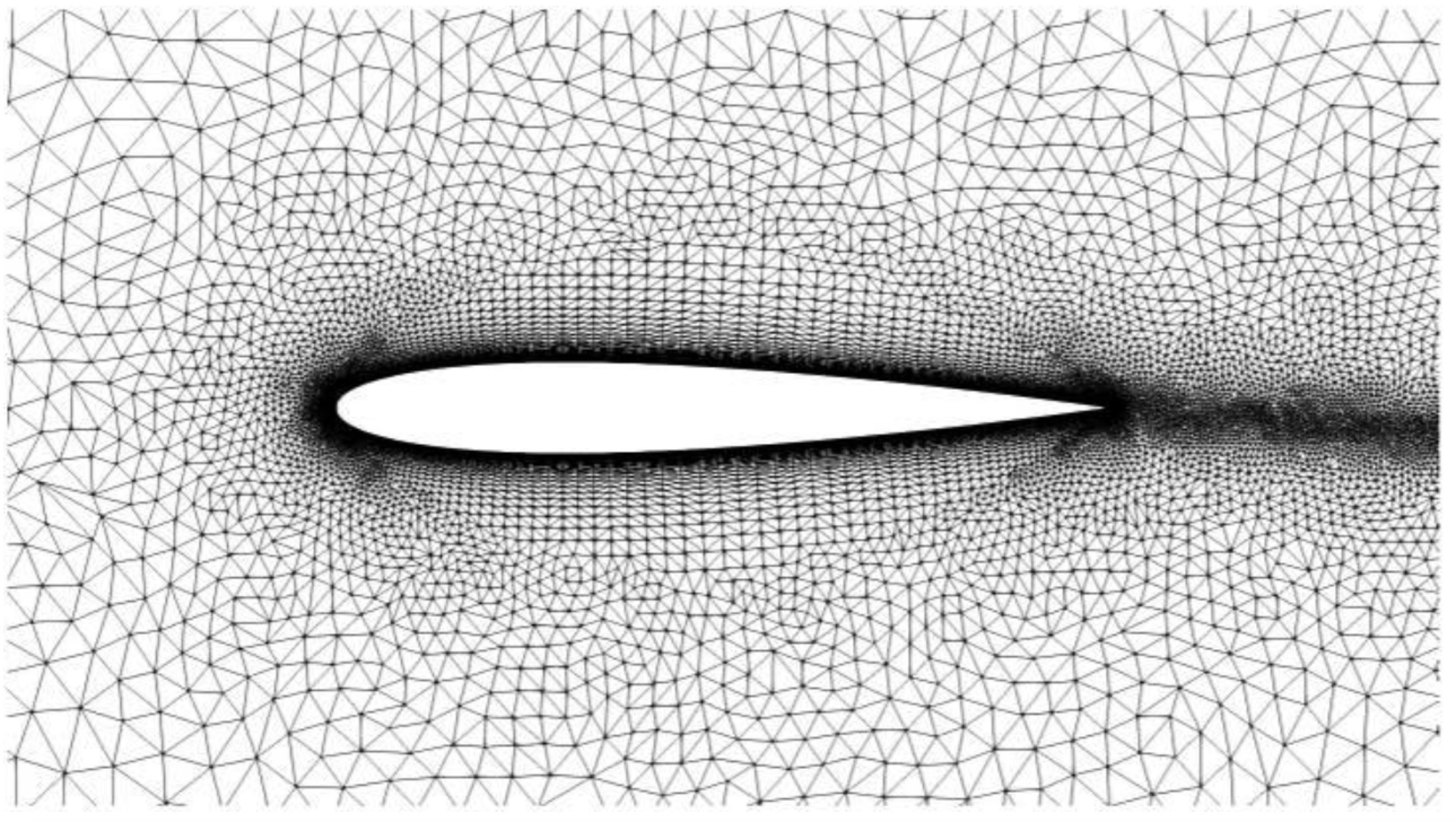}
\caption{Symmetric NACA 0012 mesh used for fine grid Reynolds Averaged Numerical Simulation to model pressure coefficient response of a two dimensional NACA 4412 airfoil in a steady, incompressible flow with Reynolds number $1.52 \times 10^6$~\cite{skinner2019reduced}.}
 \label{fig:NACA_airfoil}
\end{figure}
\begin{figure}[t]
\centering
\includegraphics[width=0.49\linewidth]{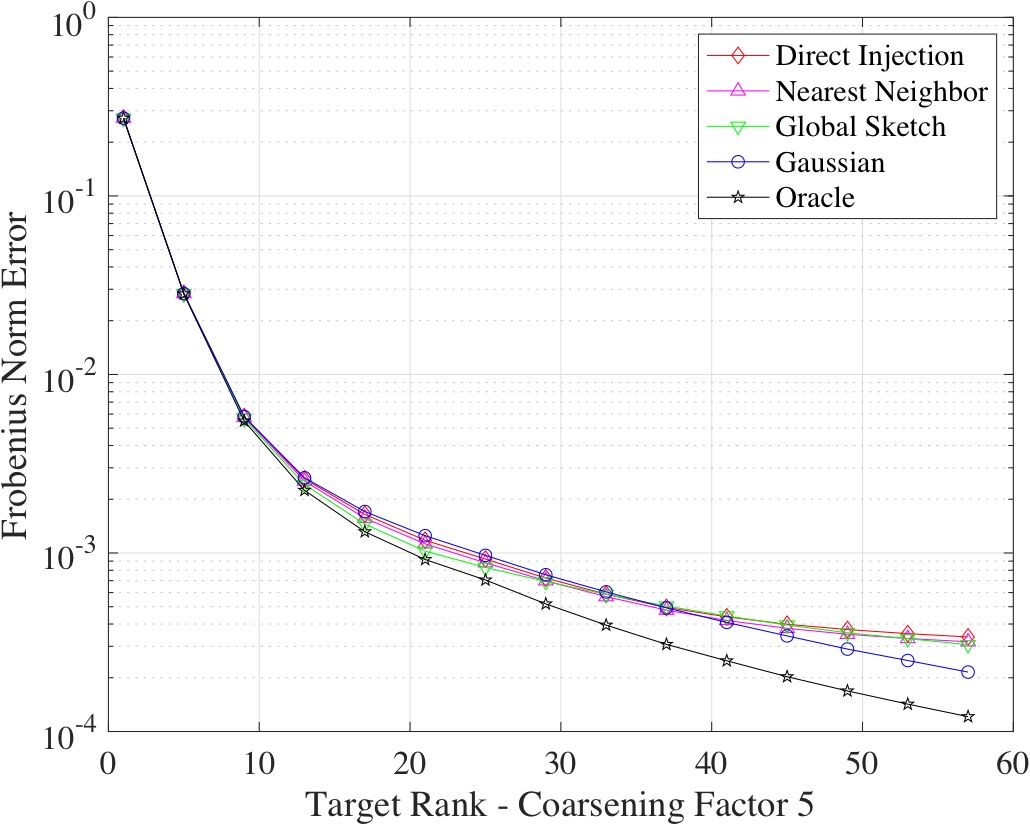}
\includegraphics[width=0.48\linewidth]{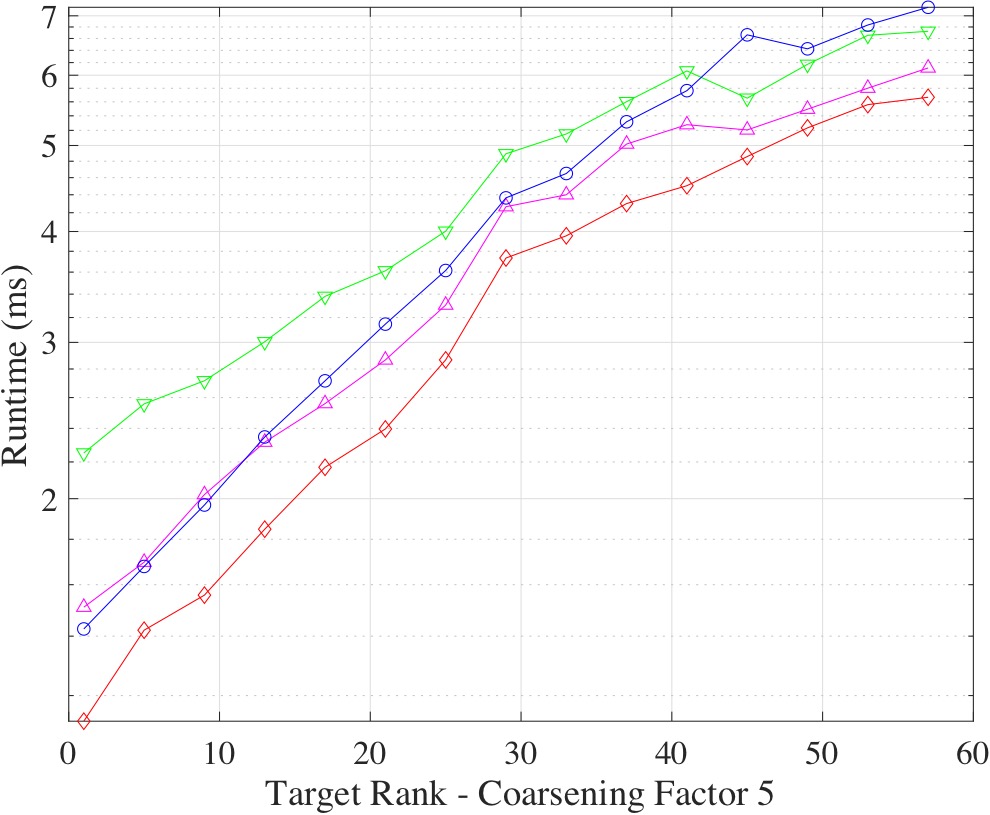}
\caption{Left: Relative Frobenius error of three proposed sketches used in SPC-SVD (Algorithm~\ref{alg:coarsespsvd}) and a simplified version of the single-pass algorithm from~\cite{yu2017single} compared against the optimal error given by the Eckart-Young theorem on the NACA-4412 airfoil dataset. Right: Sketch times of the four methods.}
 \label{fig:ranks_AIAA_sketches}
\end{figure}
\begin{figure}[t]
\centering
\includegraphics[width=0.49\linewidth]{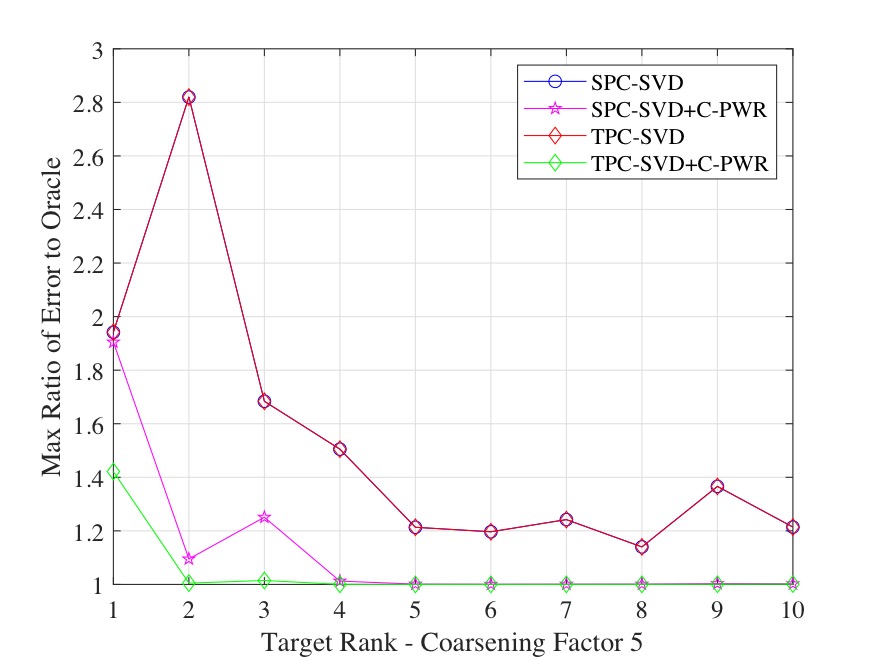}
\includegraphics[width=0.49\linewidth]{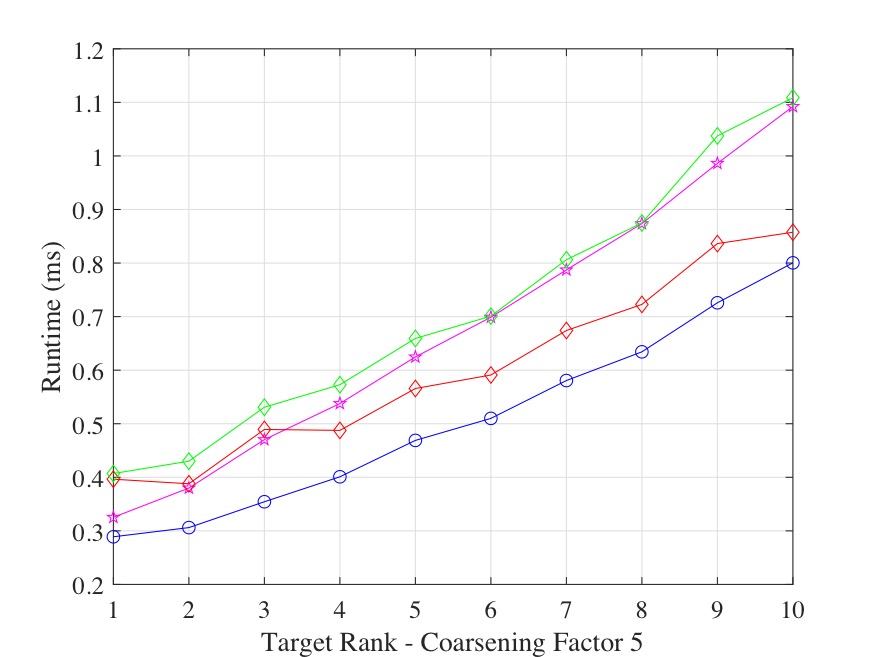}
\includegraphics[width=0.49\linewidth]{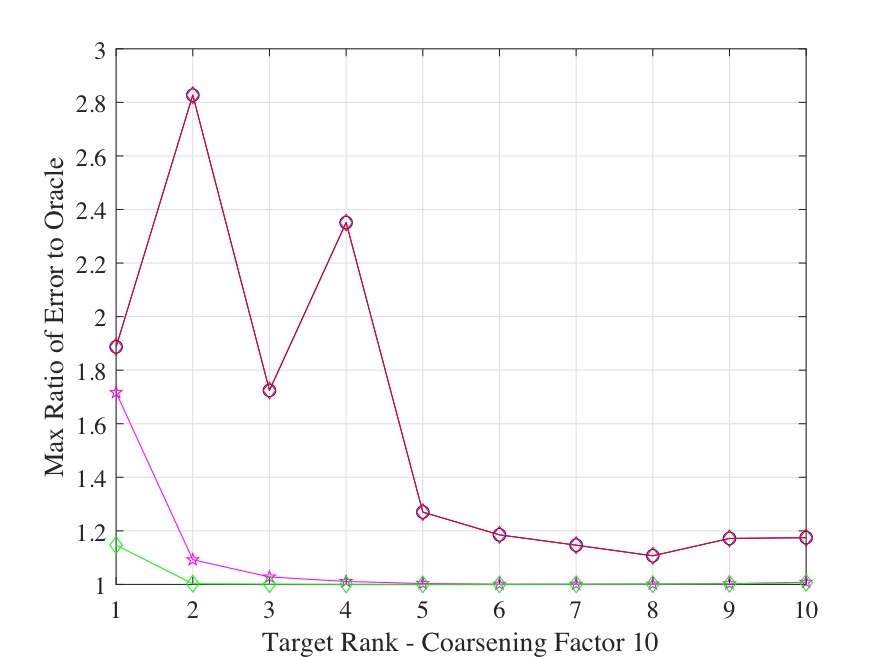}
\includegraphics[width=0.49\linewidth]{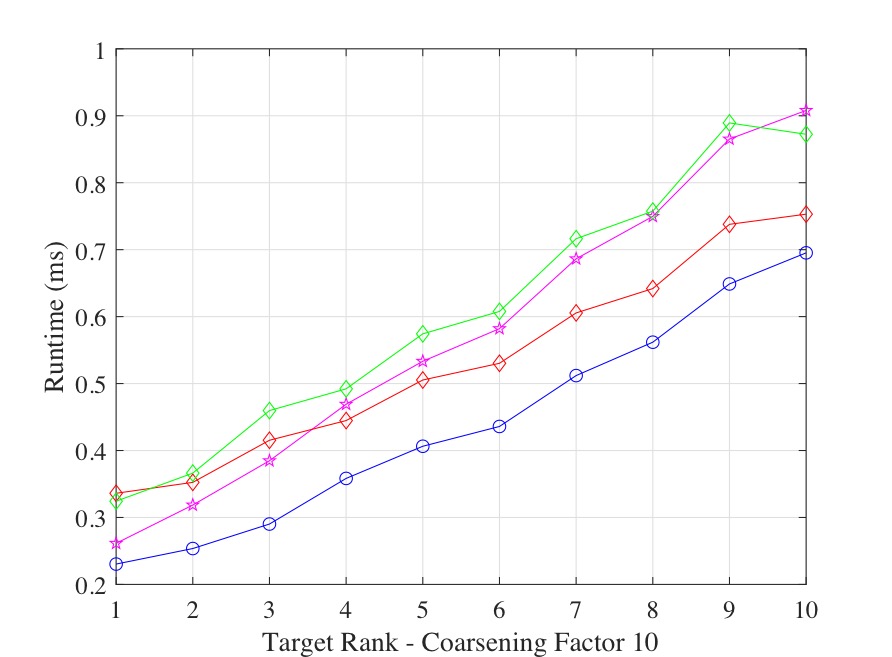}
\caption{Left: Maximum ratio over 100 independent trials of Frobenius error of schemes relative to the lower bound given by the Eckart-Young theorem \alec{on the NACA-4412 airfoil dataset}. Right: Average runtimes over the same 100 trials.}
 \label{fig:ranks_AIAA}
\end{figure}

The first dataset is generated by measuring the coefficient of pressure $C_p$ along $768$ points on the surface of a NACA-4412 airfoil for $500$ steady-state solutions to the Navier Stokes equations solved using a grid-independent Spalart-Allmaras simulation~\cite{spalart1992one}, yielding a data matrix of size $500 \times 768$ (Figure~\ref{fig:NACA_airfoil}). Each row corresponds to a steady-state solution for a combination of four stochastic parameters: the angle of attack, maximum camber, location of maximum camber, and maximum thickness~\cite{skinner2017evaluation,skinner2019reduced}.

Direct injection, nearest-neighbors with weights $\lbrack 1/4 \hspace{6pt} 1/2 \hspace{6pt} 1/4 \rbrack$ and $d=1$, global sketch, and a dense Gaussian sketch \alec{with oversampling parameter 10} (used in~\cite{yu2017single}), are compared to the oracle solution. For the first test, reported in Figure~\ref{fig:ranks_AIAA_sketches}, the coarsening factor is set to 5. In the left panel the relative error is \alec{rather similar} for all sketches. In the right-panel, the average sketch time for each method is presented. The slowest of the schemes is the global sketch; for larger ranks direct injection is the fastest of the four methods, followed by nearest neighbors and Gaussian sketch. This reflects the complexity analysis presented in Section~\ref{sec:algsummary}.

C-PWR (Algorithm~\ref{alg:coarsespsvdpower}) is incorporated into TPC-SVD (Algorithm~\ref{alg:coarsetpsvd}) and SPC-SVD (Algorithm~\ref{alg:coarsespsvd}), with results reported in Figure~\ref{fig:ranks_AIAA}.  We first consider the case with coarsening factor set to $5$ (top two panels of Figure~\ref{fig:ranks_AIAA}). For target rank 2, applying C-PWR improves the MREO by a factor of almost 3 for TPC-SVD and SPC-SVD. Moreover, for all target rank values greater than 1, TPC-SVD+C-PWR and SPC-SVD+C-PWR achieve near-optimal rank-$k$ approximations with one power iteration.
In the top right panel of Figure~\ref{fig:ranks_AIAA}, the average runtime over the same $100$ independent trials is reported. The slowest method of the four is TPC-SVD+C-PWR, followed by SPC-SVD+C-PWR, TPC-SVD, and finally SPC-SVD. These results are not surprising; adding a power iteration step and second pass should increase runtime.  

In the final numerical experiment for the NACA-4412 airfoil dataset, the coarsening factor is set to 10 with the same experimental setup as before. In the bottom left panel of Figure~\ref{fig:ranks_AIAA}, the MREO of the four algorithms for target ranks between 1 and 10 is provided.  In TPC-SVD and SPC-SVD, applying one iteration of C-PWR reduces the MREO by a factor of almost 3 for a target rank of 2; for all ranks greater than 1, TPC-SVD+C-PWR and SPC-SVD+C-PWR achieve near-optimal rank-$k$ approximations. The runtimes shown in the bottom right panel of Figure~\ref{fig:ranks_AIAA} are slightly faster than those reported in the top right panel of Figure~\ref{fig:ranks_AIAA} due to the doubling of the coarsening factor.
\begin{figure}[t]
  \centering
  \includegraphics[width=0.51\textwidth]{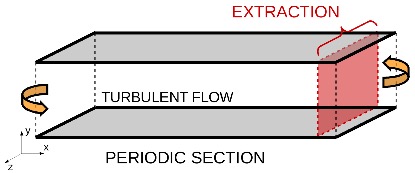}
  \includegraphics[width=0.47\textwidth]{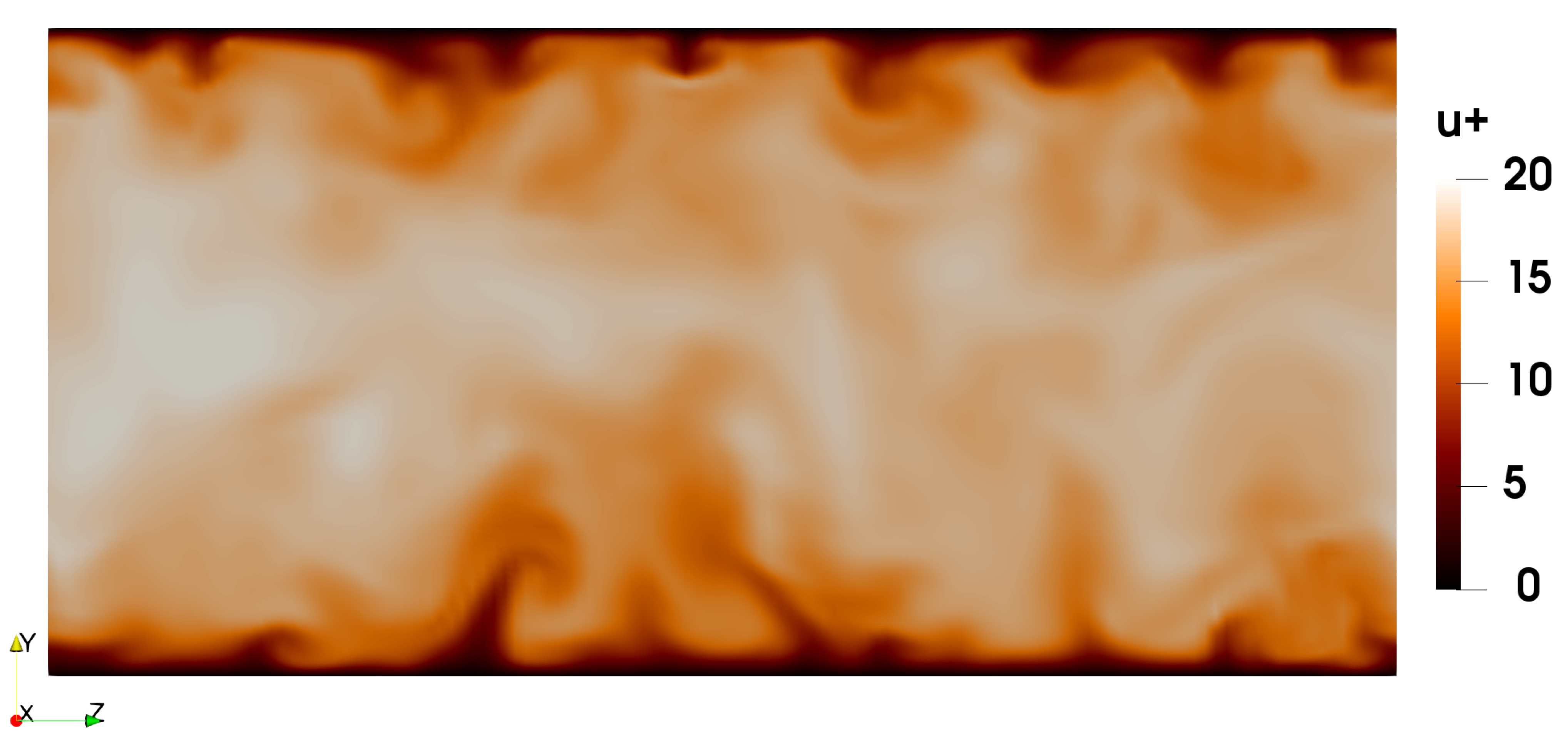}
  \caption{Left: Data extraction computational setup for compression and reconstruction. Right: Stream-wise velocity (wall units) on the extraction plane (instantaneous, spanwise $y$-$z$ plane snapshot).}	\label{fig:data_capture}
\end{figure}

\subsection{Turbulent channel flow \alec{data}}

\begin{figure}[t]
\centering
\includegraphics[width=0.49\linewidth]{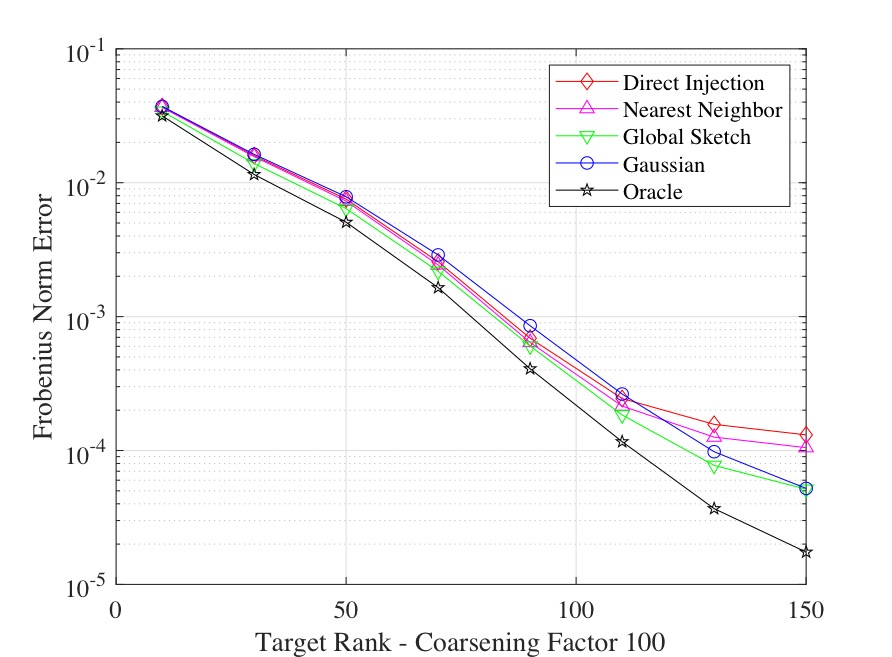}
\includegraphics[width=0.49\linewidth]{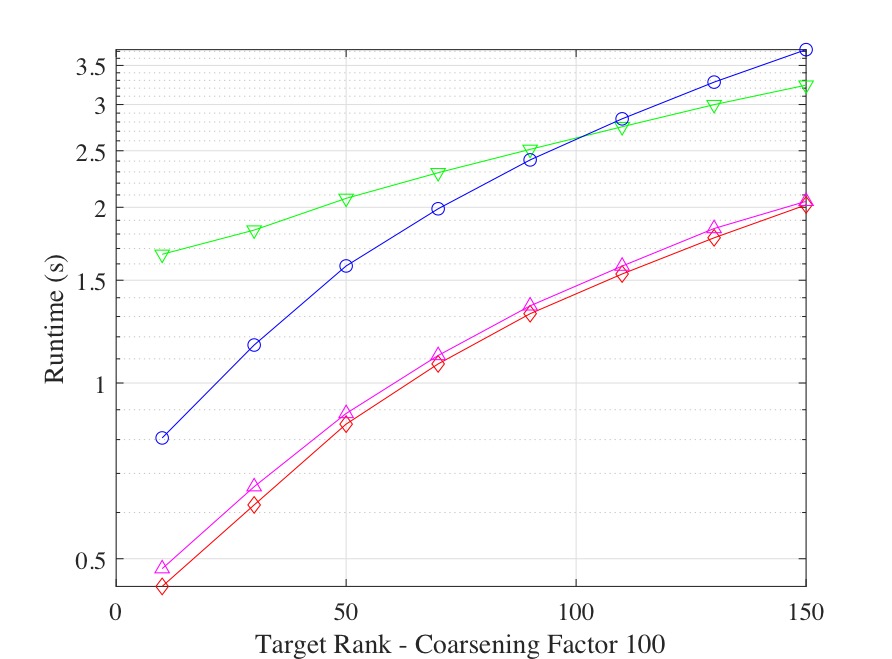}
\caption{Left: Relative Frobenius error of three proposed sketches used in SPC-SVD (Algorithm~\ref{alg:coarsespsvd}) and the Gaussian sketch from the single-pass algorithm from~\cite{yu2017single} compared against the optimal error given by the Eckart-Young theorem on the turbulent channel flow dataset. Right: Sketch times of the four methods (one data point removed as an outlier).}
 \label{fig:ranks_Uvel}
\end{figure}

\begin{figure}[t]
\centering
\includegraphics[width=0.49\linewidth]{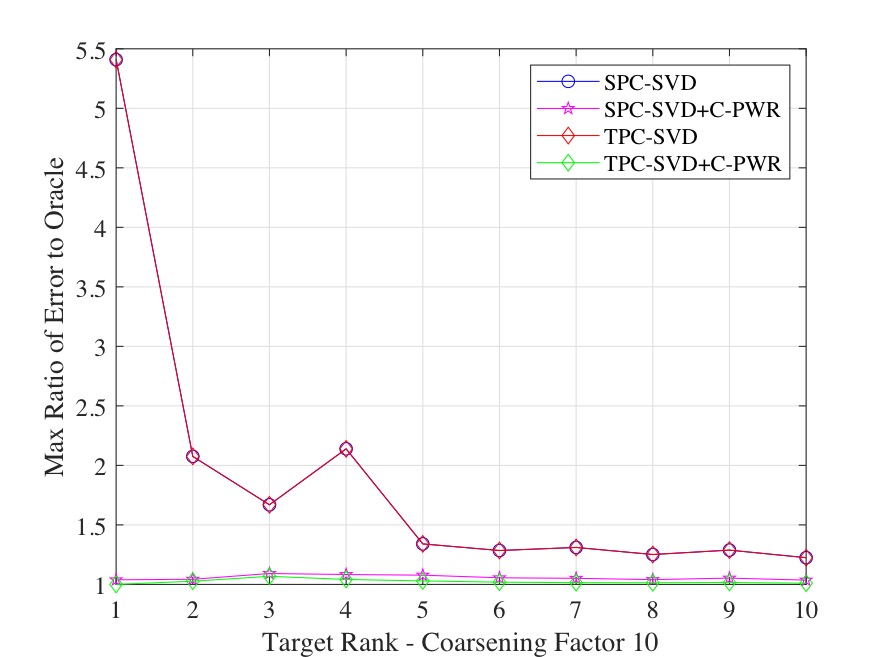}
\includegraphics[width=0.49\linewidth]{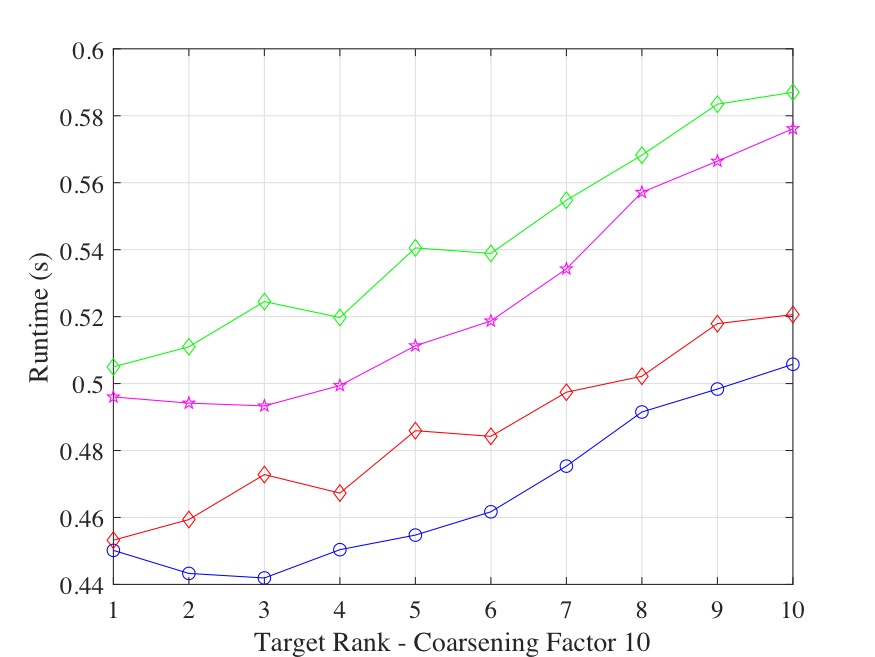}
\includegraphics[width=0.49\linewidth]{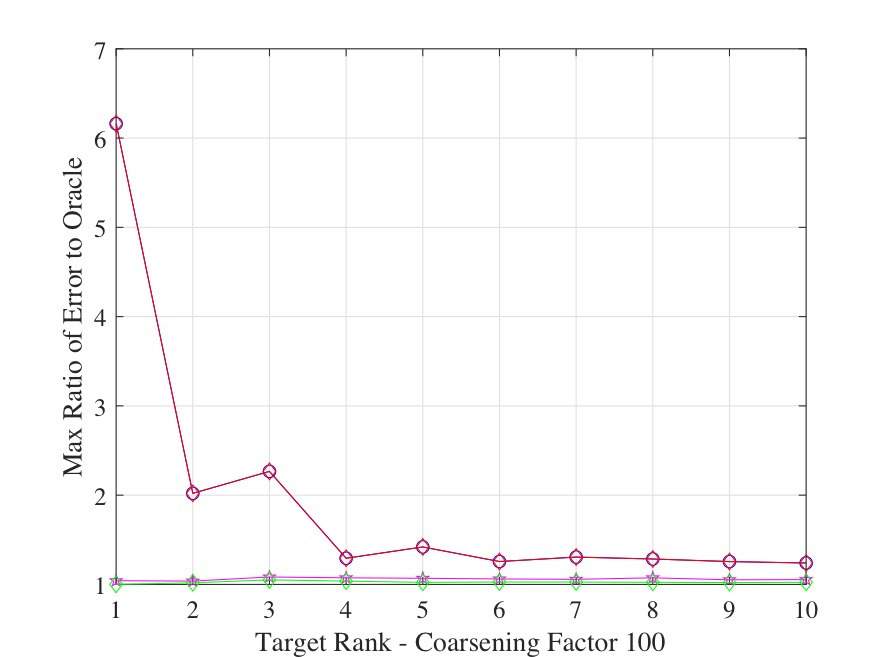}
\includegraphics[width=0.49\linewidth]{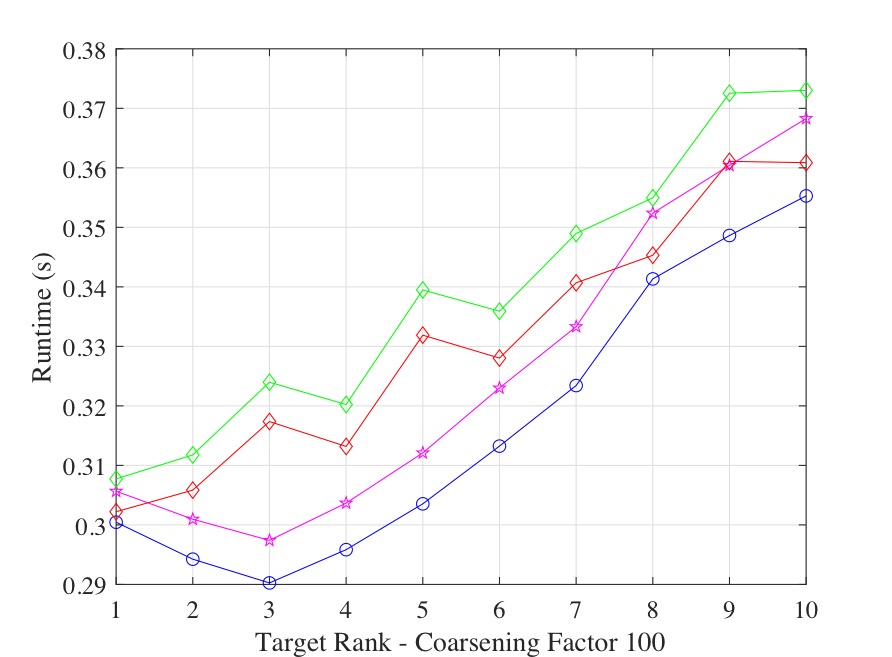}
\caption{Left: Maximum ratio over 100 independent trials of Frobenius error of schemes relative to the lower bound given by the Eckart-Young theorem \alec{on the turbulent channel flow dataset}. Right: Average runtimes over the same 100 trials.}
 \label{fig:ranks_channel}
\end{figure}

The second dataset is extracted from the direct numerical simulation (DNS) of \alec{a} wall-bounded particle-laden turbulent flow at frictional Reynolds number $Re_{\tau} = 180$ using the Soleil-MPI low-Mach-number flow solver~\cite{esmaily2018scalable}. After the system has reached turbulent steady-state conditions, the stream-wise $u$ velocity field is collected on a $130 \times 130$ grid at the outlet for an entire flow through time as shown in Figure~\ref{fig:data_capture}, thereby generating $25100$ snapshots corresponding to a $25100 \times 16900$ data matrix~\cite{moser1999direct,dunton2020pass}.

The first test demonstrates errors incurred in SPC-SVD using coarsening factors of 100 for: (1) direct injection, (2) nearest-neighbors with weights $\lbrack 1/4 \hspace{6pt} 1/2 \hspace{6pt} 1/4 \rbrack$ and $d=1$, (3) global sketch, (4) a dense Gaussian matrix. In the left panel of Figure~\ref{fig:ranks_Uvel}, all four sketches are shown to perform similarly relative to the oracle error (\ref{eq:eckartyoung}) for smaller target rank values. As the target rank increases beyond 100, the global sketch and dense Gaussian sketch outperform the direct injection and nearest-neighbor sketches. This result is unsurprising; global sketch uses more fine grid data than all other deterministic sketches.

The right panel of Figure~\ref{fig:ranks_Uvel} displays the sketch times of the four sketches. Recalling Table~\ref{tab:sketches}, Gaussian and global sketch have the asymptotically largest complexity of the four sketches. This is reflected empirically; as the rank is increased, dense Gaussian sketch is the slowest of the four, while global sketch is the second worst. Further, nearest neighbors and direct injection are the two fastest sketches.

TPC-SVD+C-PWR and SPC-SVD+C-PWR are now compared to TPC-SVD and SPC-SVD. In the first test, the coarsening factor is set to 10, and the performances of the four algorithms are shown in Figure~\ref{fig:ranks_channel}. In the top left panel of Figure~\ref{fig:ranks_channel}, the MREO over $100$ independent \alec{trials} per target rank value is provided. When $k=1$, a single coarse grid power iteration ($q=1$ in Algorithm~\ref{alg:coarsespsvdpower}) reduces the MREO from over 5 to nearly 1 for both SPC-SVD and TPC-SVD. For all target rank values, approximations generated without C-PWR display noticeable sub-optimality; those incorporating C-PWR do not. The top right panel of Figure~\ref{fig:ranks_channel} displays the average runtime of the four methods. SPC-SVD is the fastest of the four methods, followed by TPC-SVD, SPC-SVD+C-PWR, and TPC-SVD+C-PWR. 

In the second test (results displayed in the bottom two panels of Figure~\ref{fig:ranks_channel}), the coarsening factor is increased to $100$.  \alec{Despite a larger coarsening factor}, C-PWR again improves the rank-$1$ approximation by a factor close to 6 for TPC-SVD and SPC-SVD; both methods achieve near optimal error for all trials over all target rank values. In the bottom right panel of Figure~\ref{fig:ranks_channel}, the average runtime of the four methods over the same $100$ trials is provided. TPC-SVD and TPC-SVD+C-PWR are slightly slower than SPC-SVD and SPC-SVD+C-PWR for most target rank values. Incorporating C-PWR into the existing methods with sufficiently large coarsening factor introduces negligible computational cost.

\section{Numerical experiments: SPC-ID}
In this section SPC-ID (Algorithm~\ref{alg:spcid}) and SPC-ID+C-PWR are compared to their two-pass counterparts TPC-ID (Algorithm~\ref{alg:subid})~\cite{dunton2020pass} and TPC-ID+C-PWR using the same two datasets as in Section~\ref{sec:numericalexperiments}. 
Re-orthonormalization (see optional steps in Algorithm~\ref{alg:spcid}) is not used in accordance with the power iteration ID algorithm presented in~\cite{martinsson2019randomized}. 

\label{sec:numericalexperimentsID}
\subsection{NACA-4412 airfoil \alec{data}}
\begin{figure}[t]
\centering
\includegraphics[width=0.49\linewidth]{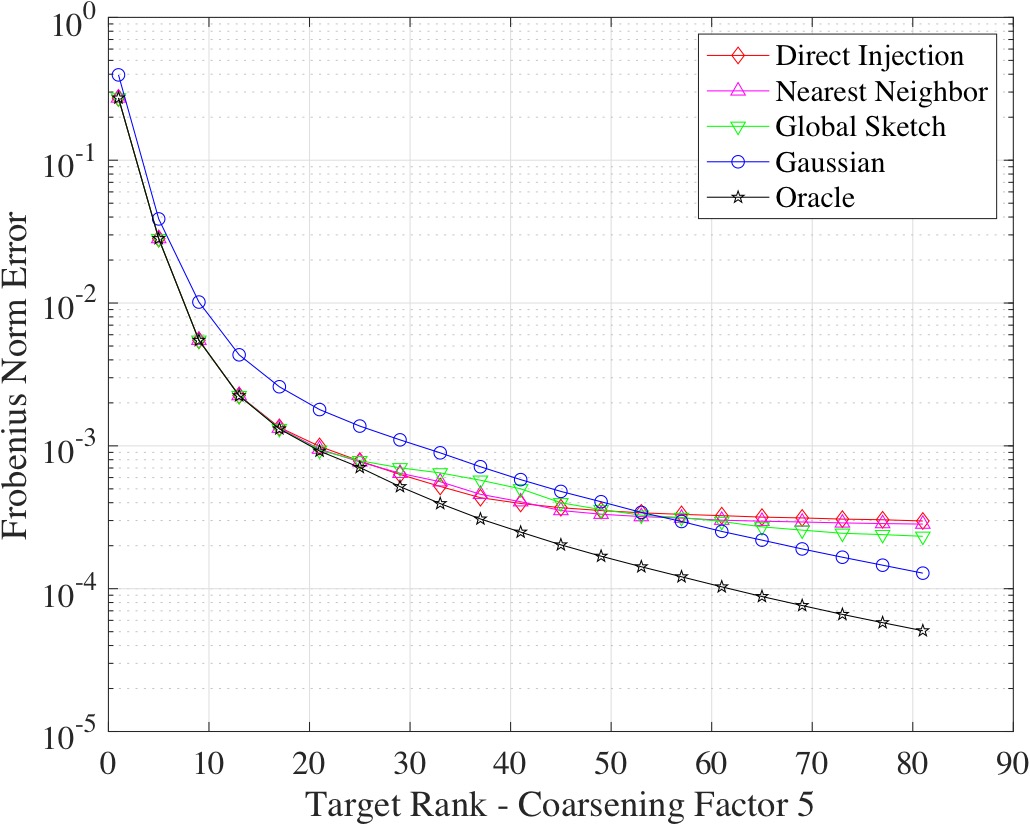}
\includegraphics[width=0.48\linewidth]{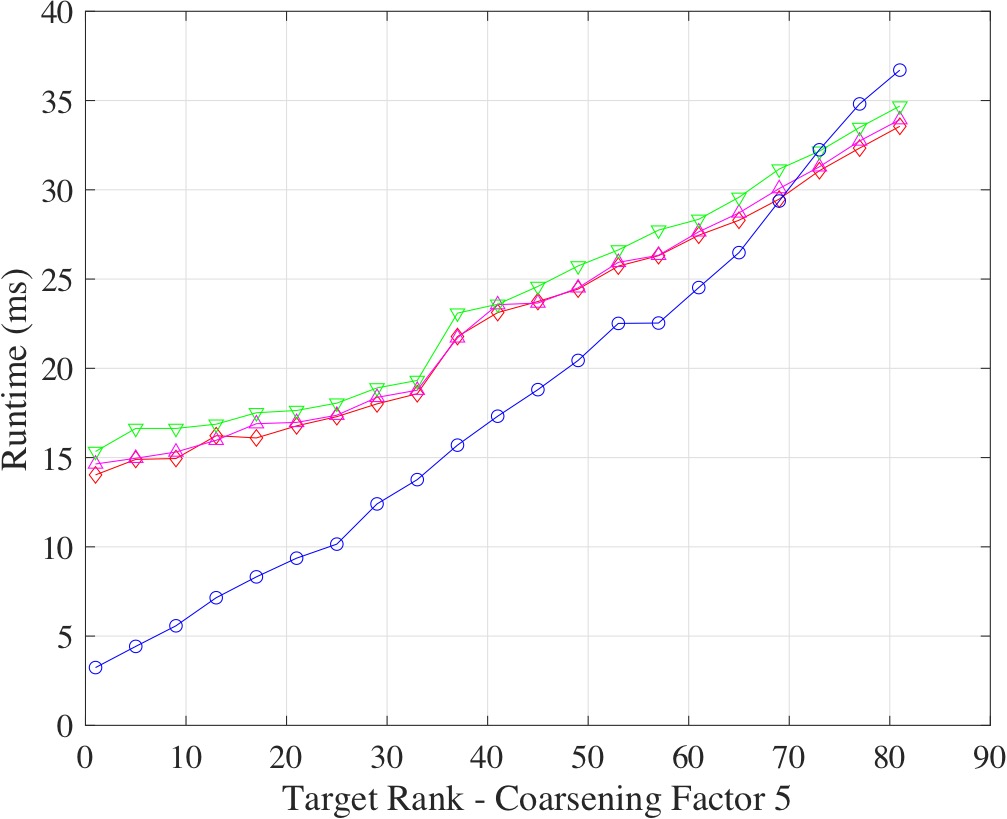}
\caption{Left: Relative Frobenius error of three proposed sketches used in SPC-ID (Algorithm~\ref{alg:spcid}) compared against the optimal error given by the Eckart-Young theorem on the NACA-4412 airfoil dataset. Right: Sketch times of the four methods.}
 \label{fig:ranks_AIAA_sketches_ID}
\end{figure}
\begin{figure}[t]
\centering
\includegraphics[width=0.49\linewidth]{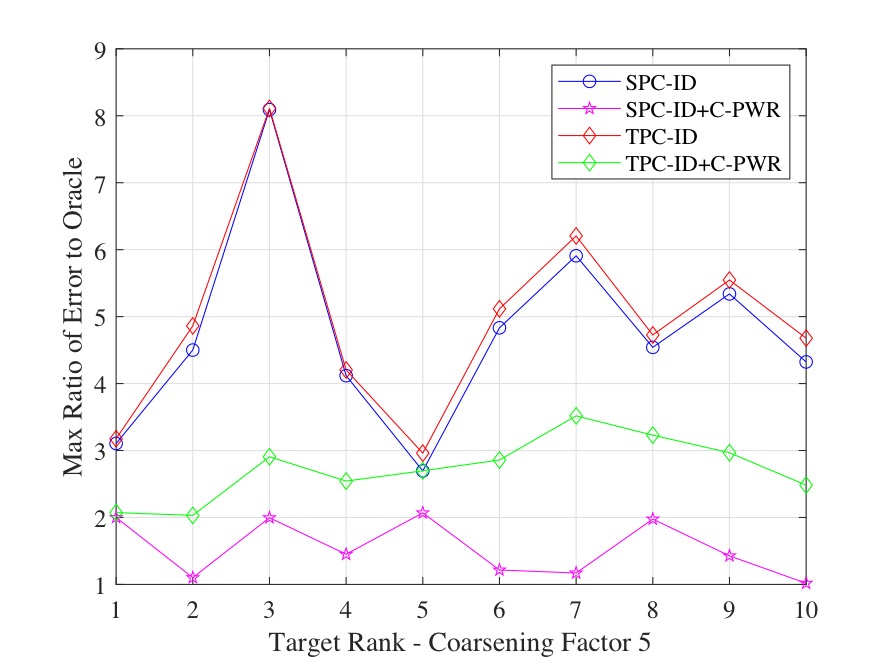}
\includegraphics[width=0.49\linewidth]{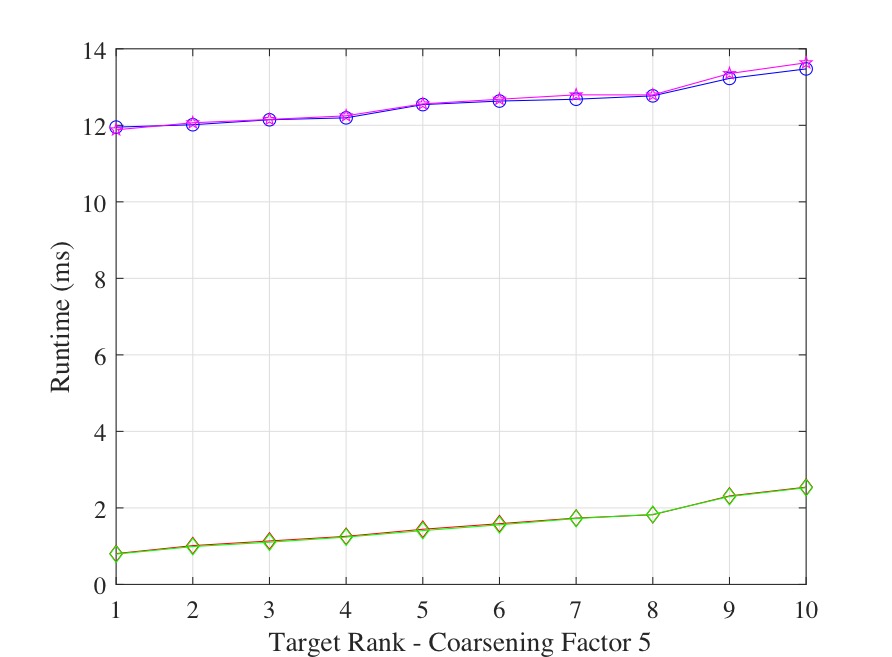}
\includegraphics[width=0.49\linewidth]{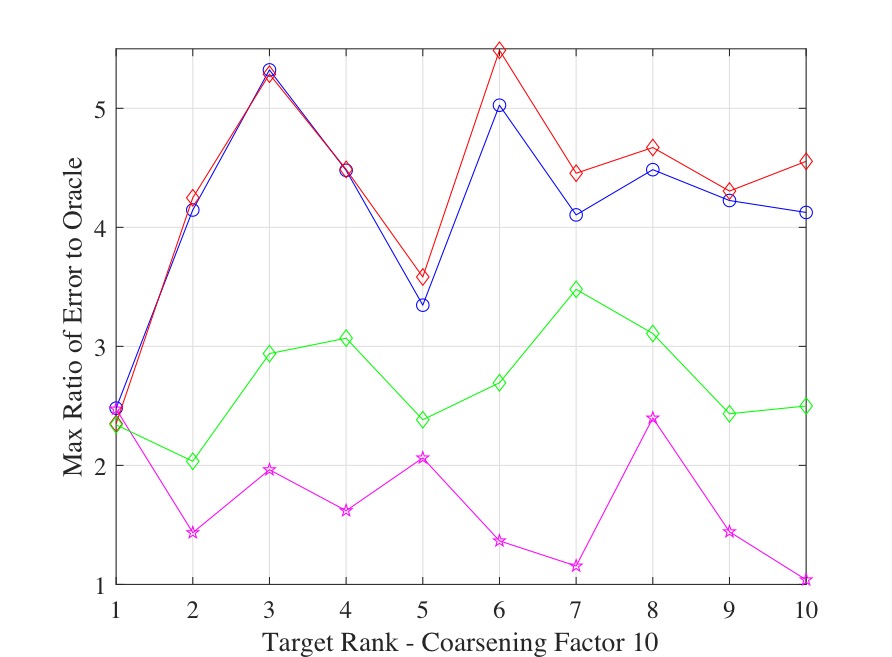}
\includegraphics[width=0.49\linewidth]{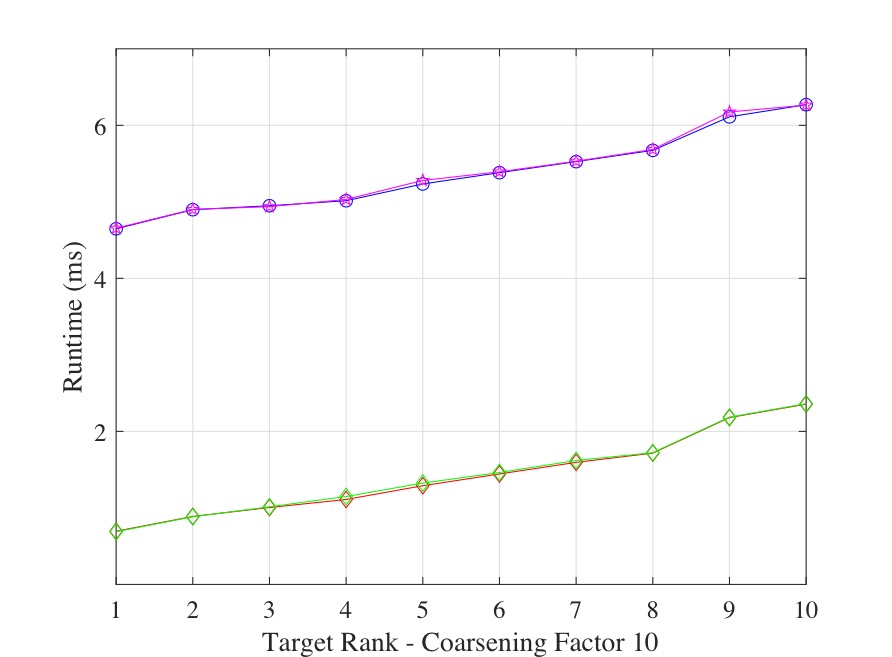}
\caption{Left: Maximum ratio over 100 independent trials of Frobenius error of schemes relative to the lower bound given by the Eckart-Young theorem \alec{on the NACA-4412 airfoil dataset}. Right: Average runtimes over the same 100 trials.}
 \label{fig:ranks_AIAA_ID}
\end{figure}
In the NACA airfoil test case, all three deterministic sketches perform roughly as well as one another, and noticeably better than the dense Gaussian sketch in terms of relative Frobenius error. \alec{Moreover, the deterministic approaches match the oracle solution for target ranks less than or equal to 20. The deterministic approaches are slower than the Gaussian sketch by less than an order of magnitude (see Figure~\ref{fig:ranks_AIAA_sketches_ID}).} However, the deterministic sketches have smaller slope than the Gaussian sketch, reflecting their superior computational complexity. In the top left panel of Figure~\ref{fig:ranks_AIAA_ID}, SPC-ID, TPC-ID, SPC-ID+C-PWR, and TPC-ID+C-PWR are compared. SPC-ID and TPC-ID achieve similar MREO for all target rank values. Applying one coarse grid power iteration with coarsening factor $n/n_c = 5$ gives significant error reduction for both approaches, but yields noticeably better results for SPC-ID+C-PWR than TPC-ID+C-PWR. 

Across all target rank values, TPC-ID is less accurate than SPC-ID. This is not surprising; although TPC-ID is a two-pass algorithm whereas SPC-ID is single-pass, Theorem~\ref{thm:SPCIDbound} when compared to Theorem 2.1 of~\cite{dunton2020pass} suggests the methods should be comparable in accuracy.

As seen in the top right panel of Figure~\ref{fig:ranks_AIAA_ID}, SPC-ID and SPC-ID+C-PWR are roughly an order of magnitude slower than TPC-ID and TPC-ID+C-PWR. This can be attributed to the computational expense of computing the lifting operator. To ameliorate this, one may construct the lifting operator in parallel to the ID approximation of the coarse grid data.

In the bottom left panel of Figure~\ref{fig:ranks_AIAA_ID}, the coarsening factor is increased from $n/n_c =5$ to 10. Across all target rank values from 1 to 10 for 100 independent trials, SPC-ID and TPC-ID achieve similar MREO. Power iteration is beneficial in all test cases; for target rank 6, one coarse grid power iteration reduces the MREO of TPC-ID from over 5 to almost 2.5, while reducing that of SPC-ID from roughly 5 to almost 1. For target rank values greater than 1, SPC-ID+C-PWR is the most accurate method, followed by TPC-ID+C-PWR. In the bottom right panel of Figure~\ref{fig:ranks_AIAA_ID} the same trend as in top right panel is observed; SPC-ID and SPC-ID+C-PWR are roughly half an order of magnitude slower than TPC-ID and TPC-ID+C-PWR.

\subsection{Turbulent channel flow \alec{data}}
\begin{figure}[t]
\centering
\includegraphics[width=0.49\linewidth]{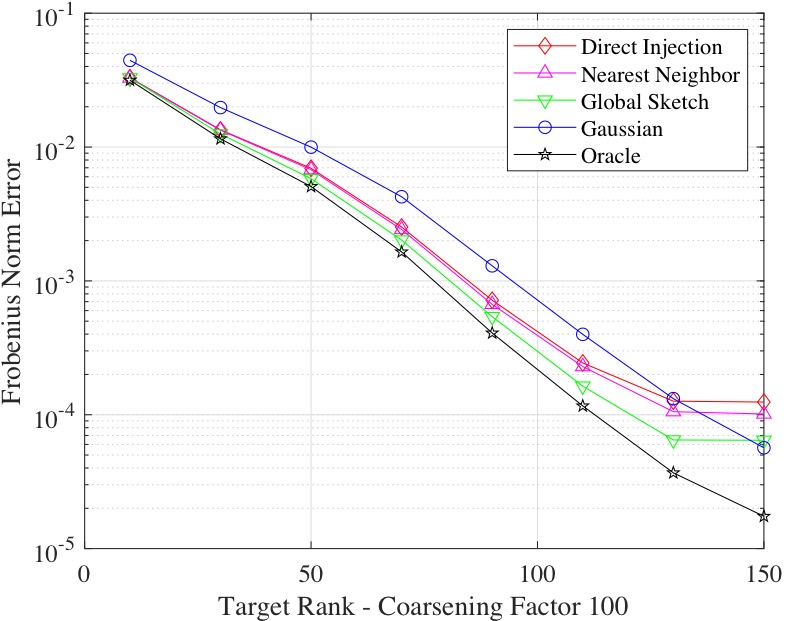}
\includegraphics[width=0.49\linewidth]{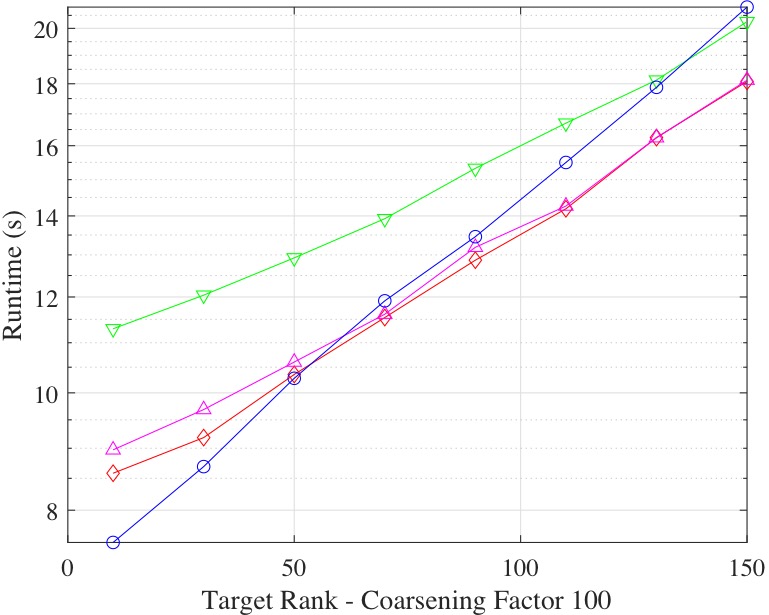}
\caption{Left: Relative Frobenius error of three proposed sketches used in SPC-ID (Algorithm~\ref{alg:coarsespsvd}) and a simplified version of the single-pass algorithm from~\cite{yu2017single} compared against the optimal error given by the Eckart-Young theorem on the turbulent channel flow dataset. Right: Sketch times of the four methods.}
 \label{fig:ranks_ChanFlow_sketches}
\end{figure}
\begin{figure}[t]
\centering
\includegraphics[width=0.49\linewidth]{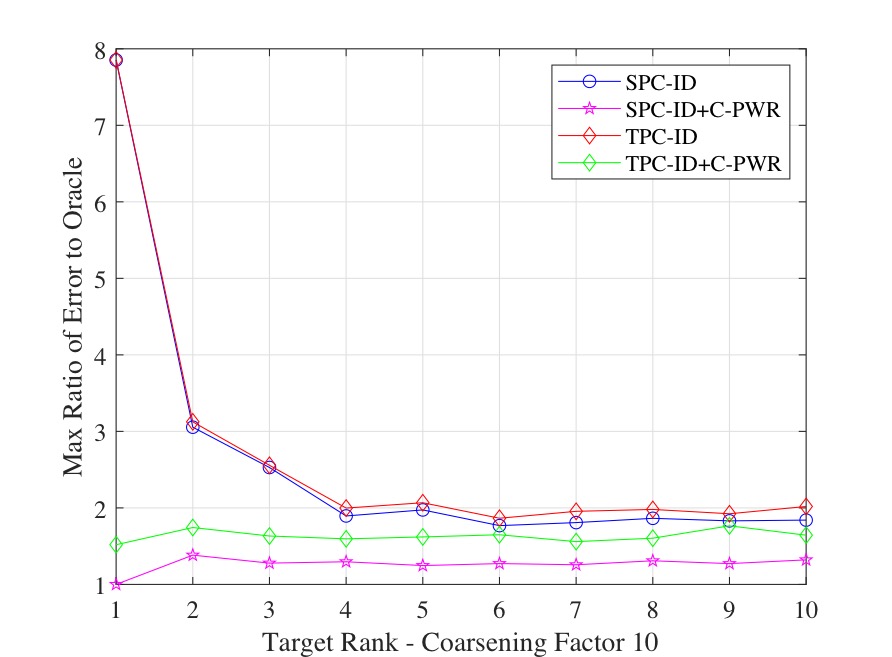}
\includegraphics[width=0.49\linewidth]{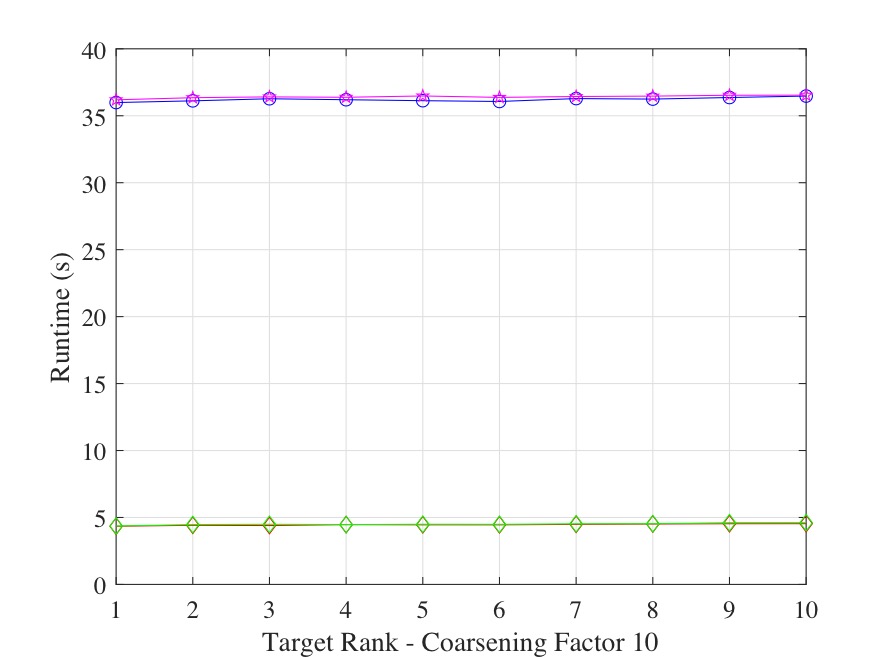}
\includegraphics[width=0.49\linewidth]{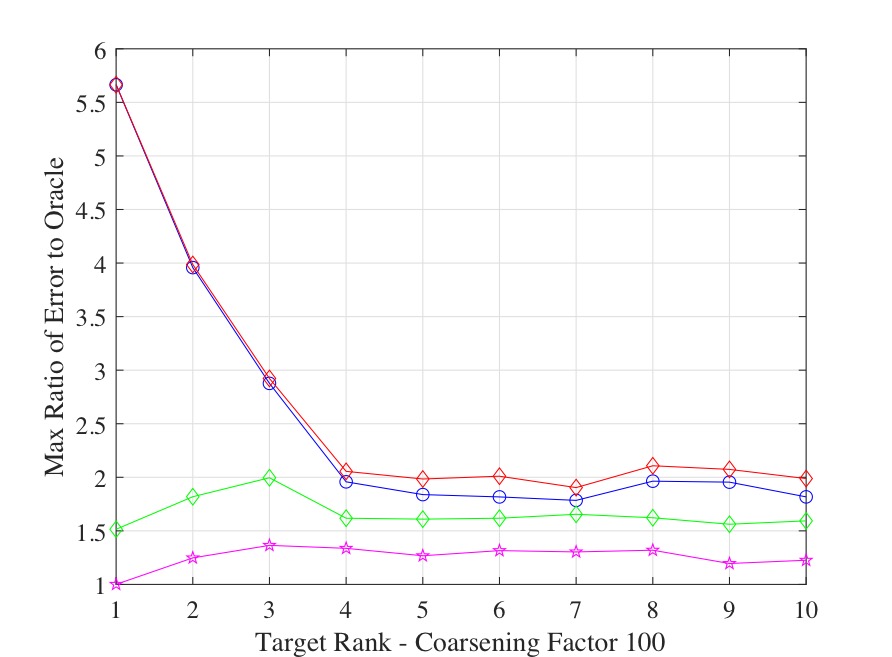}
\includegraphics[width=0.49\linewidth]{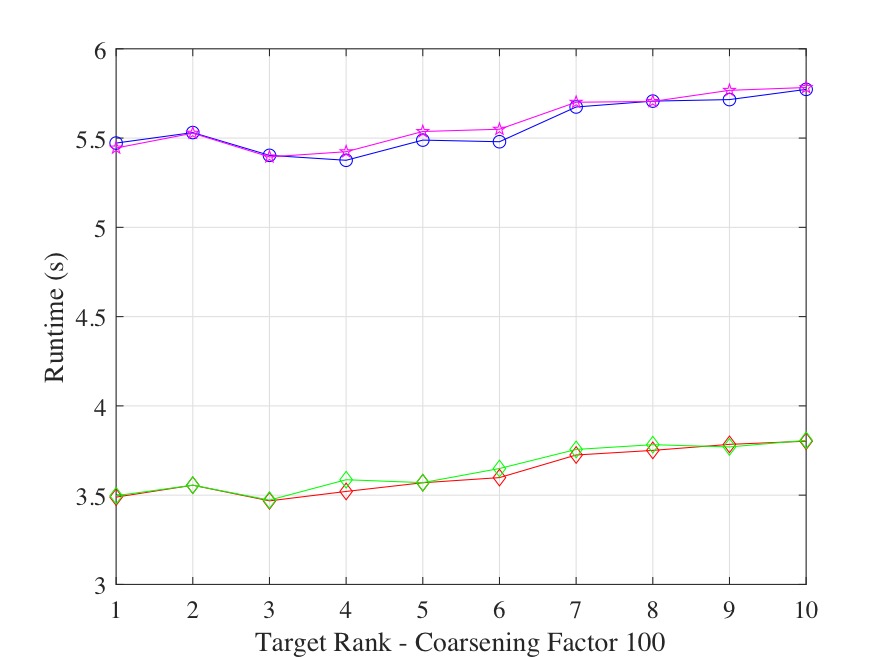}
\caption{Left: Maximum ratio over 100 independent trials of Frobenius error of schemes relative to the lower bound given by the Eckart-Young theorem \alec{on the turbulent channel flow dataset}. Right: Average runtimes over the same 100 trials.}
 \label{fig:ranks_ChanFlow_ID}
\end{figure}
The turbulent channel flow is now used as a test case for the ID algorithms. Examining the left panel of Figure~\ref{fig:ranks_ChanFlow_sketches}, for a coarsening factor of 100, all three sketches outperform the Gaussian sketch in terms of Frobenius norm error for target ranks between 10 and 130. As the target rank is increased beyond 130, direct injection performs the worst, followed by nearest neighbor and global sketch. Examining the right panel of Figure~\ref{fig:ranks_ChanFlow_sketches}, all sketches achieve similar runtimes for all target ranks tested; this is due to the dominant cost of forming the sketch matrix $\bm{H} = \bm{A}_f^T\bm{A}_c$ in Algorithm~\ref{alg:spcid}.
 
With the coarsening factor set to $n/n_c = 10$, the MREO (\ref{eq:maximumsuboptimality}) over $100$ independent trials for SPC-ID, SPC-ID+C-PWR, TPC-ID, and TPC-ID+C-PWR is provided in the top left panel of Figure~\ref{fig:ranks_ChanFlow_ID}. Across all target rank values, SPC-ID+C-PWR achieves the smallest error values, followed by TPC-ID+C-PWR, SPC-ID, and TPC-ID. SPC-ID+C-PWR achieves an MREO value well under 2 in all test cases; the probability the scheme generates a poor approximation is  low. For target rank 1, C-PWR reduces the MREO from approximately 8 to $1.5$ for TPC-ID and to approximately 1 for SPC-ID.

Shown in the top right panel of Figure~\ref{fig:ranks_ChanFlow_ID}, SPC-ID and SPC-ID+C-PWR are roughly 9 times slower than TPC-ID and TPC-ID+C-PWR. Constructing the lifting operator dominates the computational cost of SPC-ID, while C-PWR is inexpensive in comparison. Although SPC-ID is slower than TPC-ID, the cost of forming the lifting operator (and therefore the serial runtime of SPC-ID) is negligible compared to a second pass over the input.

The coarsening factor $n/n_c$ is increased to $100$ with results shown in the bottom panels of Figure~\ref{fig:ranks_ChanFlow_ID}. Across all target rank values, SPC-ID+C-PWR outperforms TPC-ID+C-PWR, SPC-ID, and TPC-ID. For target rank $1$, C-PWR reduces the MREO over $100$ independent trials of the approximations from almost 6 to 1 for SPC-ID and $1.5$ for TPC-ID. For target ranks 6 to 10, SPC-ID competes with TPC-ID+C-PWR in terms of accuracy. In the bottom right panel of Figure~\ref{fig:ranks_ChanFlow_ID}, SPC-ID and SPC-ID+C-PWR again achieve the slowest runtimes compared to their two-pass counterparts.

\section{\alec{In situ spatio-temporal compression of forced isotropic turbulence data}}
\label{sec:application}
\alec{This section presents the application of SPC-SVD and SPC-ID in tandem with three lossy data compression algorithms: SZ-2.0 (SZ)~\cite{di2016fast,tao2017significantly,liang2018error}, ZFP-0.5.0 (ZFP)~\cite{lindstrom2014fixed}, and FPZIP-1.1.4 (FPZIP)~\cite{lindstrom2006fast}. Whereas SPC-SVD and SPC-ID exploit the low-rank structure of data matrices, FPZIP, SZ, and ZFP do not do so directly. Lossy compressors like FPZIP, ZFP, and SZ are \alec{well} suited for spatial compression; they \alec{can exploit, e.g.,} the continuity of physical fields in simulation data to use local approximations to achieve compression. Datasets in this work are assumed to be inherently spatio-temporal; they are order 2 tensors whose column dimension corresponds to space and whose row dimension corresponds to time. Low-rank approximation methods - which are suited for temporal compression - can be used in tandem with SZ, FPZIP, and ZFP to yield {\it spatio-temporally} compressed data. This can be done with the method seeing each individual snapshot only once; to use SZ, ZFP, or FPZIP to compress the entire temporal evolution of the data would require multiple snapshots and therefore cannot be implemented in the same manner as the sketching-based low-rank approximation methods.} 

\alec{Spatio-temporal compression is achieved by first computing two factor matrices from a low-rank approximation, then using one of SZ, FPZIP, or ZFP to compress each factor matrix individually.} Using low-rank approaches combined with these three compressors may lead to enhanced compression factors with a minimal loss of accuracy relative to the existing low-rank approximation error. Moreover, this hybrid approach enables {\it in situ} implementation of the approach. The low-rank method is single-pass and sees each snapshot once; compressing the factor matrices downstream using FPZIP, ZFP, or SZ requires no additional passes over the original input data. Brief background on the three spatial compressors is now provided.

FPZIP~\cite{lindstrom2006fast} is a compressor developed at Lawrence Livermore National Laboratory which enables lossless and lossy compression of floating point data arrays. In the lossless setting, FPZIP \alec{achieves} compression ratios of about 1.5-4~\cite{lindstrom2006fast}; in a lossy setting it can achieve significantly greater data reduction. The algorithm relies on the {\it Lorenzo predictor}, which estimates values at a corner of a cube based on the values at the other corners~\cite{li2018data}.  

ZFP~\cite{lindstrom2014fixed}, also developed at Lawrence Livermore, relies on a custom orthogonal transform on $4^d$-size blocks (where $d$ is the order of the data tensor) and encoding the corresponding coefficients. It allows for both lossless and lossy compression of floating point data, and frequently outperforms FPZIP in lossy compression on benchmark problems~\cite{lindstrom2014fixed}. Moreover, ZFP allows the user to select a fixed size and a fixed accuracy (maximum allowable point-wise absolute error). 

SZ~\cite{liang2018error} is a compression algorithm developed by Argonne National Laboratory. SZ exclusively enables lossy compression, allows the user to determine a maximum allowable error, and relies on multiple prediction schemes to estimate data values based on those of their neighbors. SZ generalizes the Lorenzo predictor used in FPZIP to cubes of arbitrary size, $n^d$, where $d$ is the dimensionality of the data and $n$ is the number of data points along each dimension of each cube~\cite{li2018data}.

When using SPC-SVD and SPC-ID in tandem with FPZIP, ZFP, and SZ, \alec{the low-rank approximation is formed first. Then, the dense} factor matrices are compressed \alec{individually}, yielding further compressed files which are denoted $\mathcal{S}(\bm{B})$ and $\mathcal{S}(\bm{C})$. \alec{In the case of SPC-SVD, $\bm{B} = \tilde{\bm{U}}_k$ and $\bm{C} = \tilde{\bm{S}}_k\tilde{\bm{V}}_k^T$; for SPC-ID, $\bm{B} = \bm{P}_c$ and $\bm{C} = \bm{A}_f(\mathcal{I}_c,:)\bm{T}_r$.}
 \alec{The factor matrices are not reshaped into lower or higher dimensional tensors; they are compressed as 2D arrays}. Let
\begin{equation}
    \text{Spatio-Temporal CF} = \frac{\text{no. bytes}(\bm{A}_f)}{\text{no. bytes}(\mathcal{S}({\bm{B}})) + \text{no. bytes}(\mathcal{S}({\bm{C}}))}.
    \label{eq:spatiotemporalCF}
\end{equation}

Each factor matrix $\bm{B}$ and $\bm{C}$ is compressed using a given compressor $\mathcal{S}$ such that the decompressor, abusively denoted $\mathcal{S}^{-1}$, yields a reconstructed factor matrix such that the errors in the factor matrices are much smaller than the overall error.
These factor matrices can be used to form the final approximation:
\begin{equation}
    \bm{A}_f \approx \mathcal{S}^{-1}(\mathcal{S}(\bm{B})) \mathcal{S}^{-1}(\mathcal{S}(\bm{C})).
\end{equation}
Let $\epsilon_1$ and $\epsilon_2$ be the reconstruction errors due to the compression of the factor matrices 
\begin{subequations}
\begin{align}
\epsilon_1 &= \Vert \mathcal{S}^{-1}(\mathcal{S}(\bm{B})) - \bm{B} \Vert_2 ,\\
\epsilon_2 &= \Vert \mathcal{S}^{-1}(\mathcal{S}(\bm{C})) - \bm{C} \Vert_2. \\
\label{eqn:spatiotemporalerror_epsilons}
\end{align}
\end{subequations}
Then,
\begin{align}
\Vert \bm{B} \bm{C} - \mathcal{S}^{-1}(\mathcal{S}(\bm{B})) \mathcal{S}^{-1}(\mathcal{S}(\bm{C})) \Vert_2 &\leq \Vert \bm{B} \Vert_2 \epsilon_2 + (\Vert \bm{C} \Vert_2 + \epsilon_2) \epsilon_1  .
\label{eqn:spatiotemporalerror}
\end{align}
If $\epsilon_1,\epsilon_2 \ll \Vert \bm{A}_f - \bm{B}\bm{C} \Vert_2 $, there is negligible impact on the overall approximation error. 
\begin{remark}
In ZFP and SZ, $\epsilon_1$ and $\epsilon_2$ can be controlled directly, allowing for modulation of the final approximation error due to compression of the factor matrices $\bm{B}$ and $\bm{C}$.
\end{remark}

The dataset in this application is a $100 \times 128^3$ matrix extracted from a DNS of forced isotropic turbulence on a $1024^3$ periodic grid, simulated using a pseudo-spectral parallel code for the Navier-Stokes equations from the Johns Hopkins Turbulence Databases (JHTDB)~\cite{perlman2007data,li2008public}. After the system has reached a statistically stationary state, 100 snapshots of the \alec{pressure} are generated and stored to file. The pressure of the flow is extracted at $128^3$ points (stride of $8$ on the original $1024^3$ grid) on a $x,y,z-$domain of $2 \pi \times 2 \pi \times 2 \pi$ for all $100$ timesteps between $0$ and $0.0198$ seconds. A snapshot of the solution taken at the $42^{nd}$ timestep is provided in Figure~\ref{fig:ground_truth_JHU}.

In this section, compression tests for \alec{FPZIP} and ZFP are carried out using \alec{their respective Python wrappers}, while the C implementation of SZ is called from Python using its built-in interface.
Errors are reported in terms of \alec{relative Frobenius norm} in all test cases. In the spatial compression experiments, the following parameters are used. In FPZIP, the fixed precision is set to 20. In ZFP, the tolerance is set to $10^{-2}$. In SZ, the absolute square error tolerance is set to $5 \times 10^{-1}$. These settings were found to yield approximations with relative Frobenius norm error close to $10^{-3}$.

The following parameters are used in the spatio-temporal compression experiments. In FPZIP used in tandem with ID (FPZIP+ID), the fixed precision is set to 24, while in FPZIP+SVD, the fixed precision is set to 20. In ZFP+ID, the fixed accuracy is set to $10^{-6}$ times the norm of the full data matrix. In ZFP+SVD, the fixed accuracy is set to $10^{-9}$ times the norm of the full data matrix. In SZ+ID, the fixed accuracy is set to $10^{-2}$, while in SZ+SVD the tolerance is set to $10^{-3}$. These settings were found to yield similar reconstructions accuracies for the different compression approaches. \alec{In SPC-ID and SPC-SVD the target rank is set to $4$, yielding a temporal compression factor of 25.  SPC-SVD and SPC-ID are implemented using a direct injection sketch, one power iteration, and an oversampling parameter of $10$. For the \alec{C-PWR implementation} the sub-sampling factor is set to $10$.}

 \begin{figure}[t]
\centering
\includegraphics[width=0.6\linewidth]{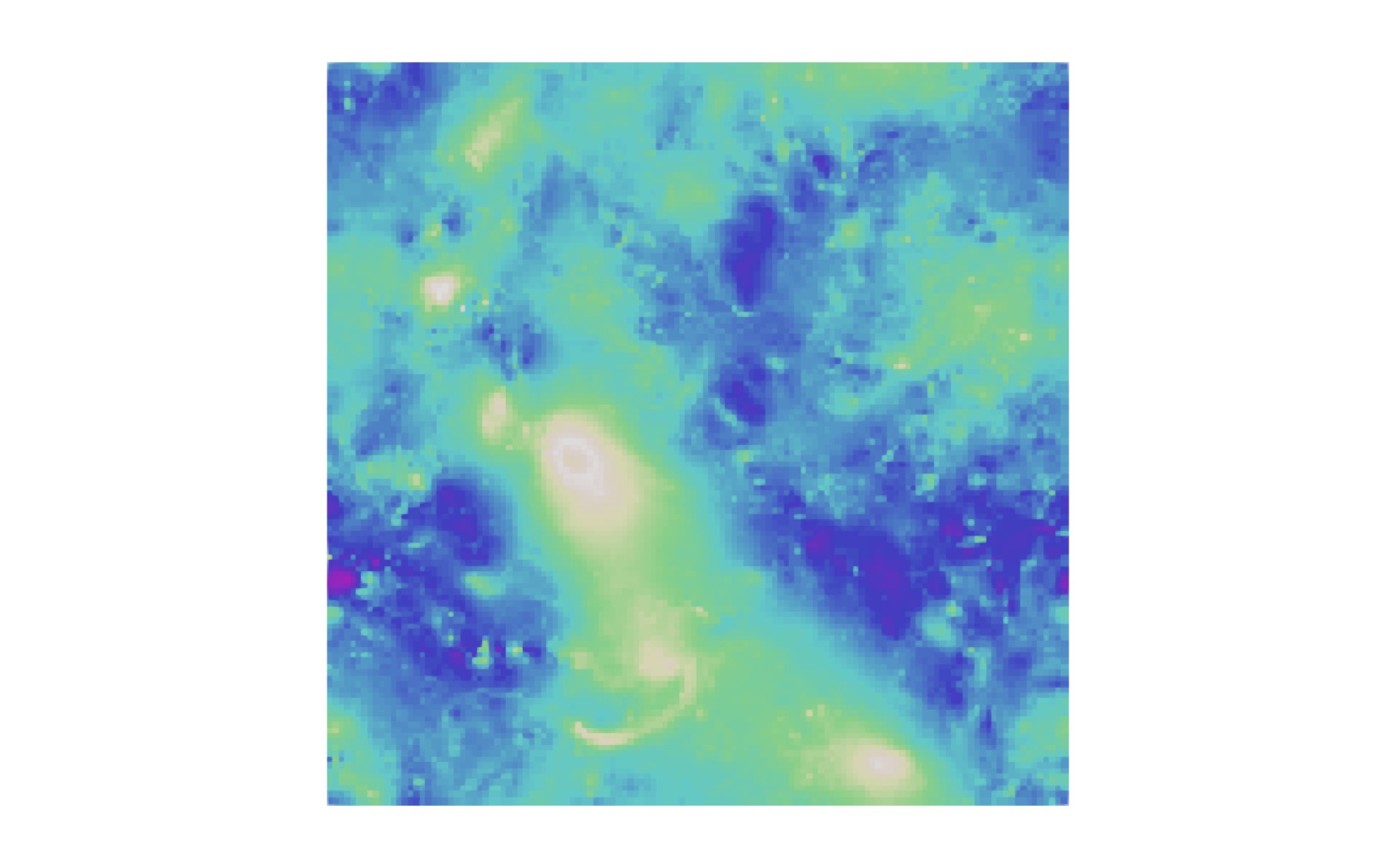}
\caption{Volumetric $x$-$z$ plane snapshot of the ground truth pressure field taken at the $42^{nd}$ time-step.}
 \label{fig:ground_truth_JHU}
 \end{figure}

 \begin{figure}[ht]
\centering
\includegraphics[width=0.49\linewidth]{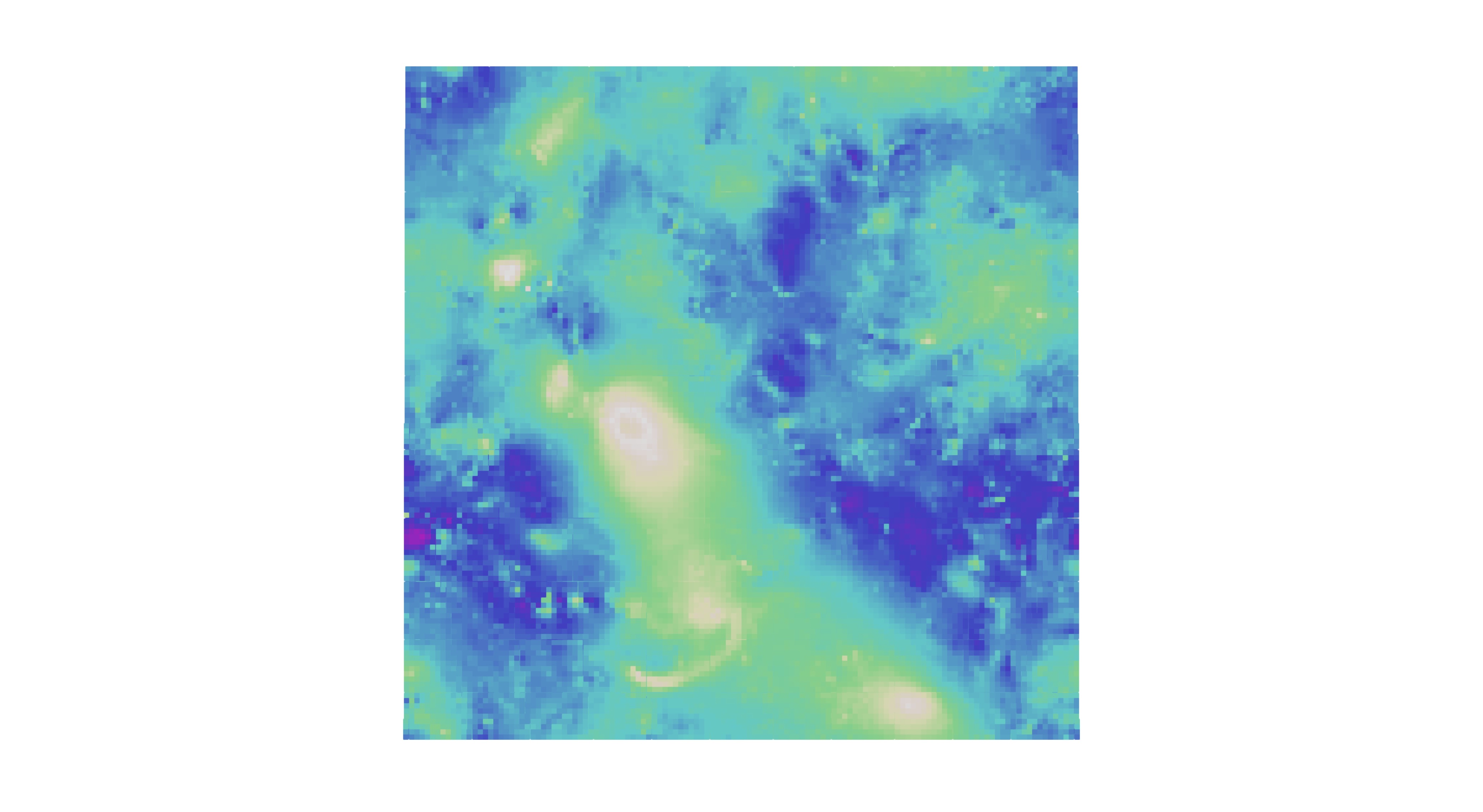}
\includegraphics[width=0.49\linewidth]{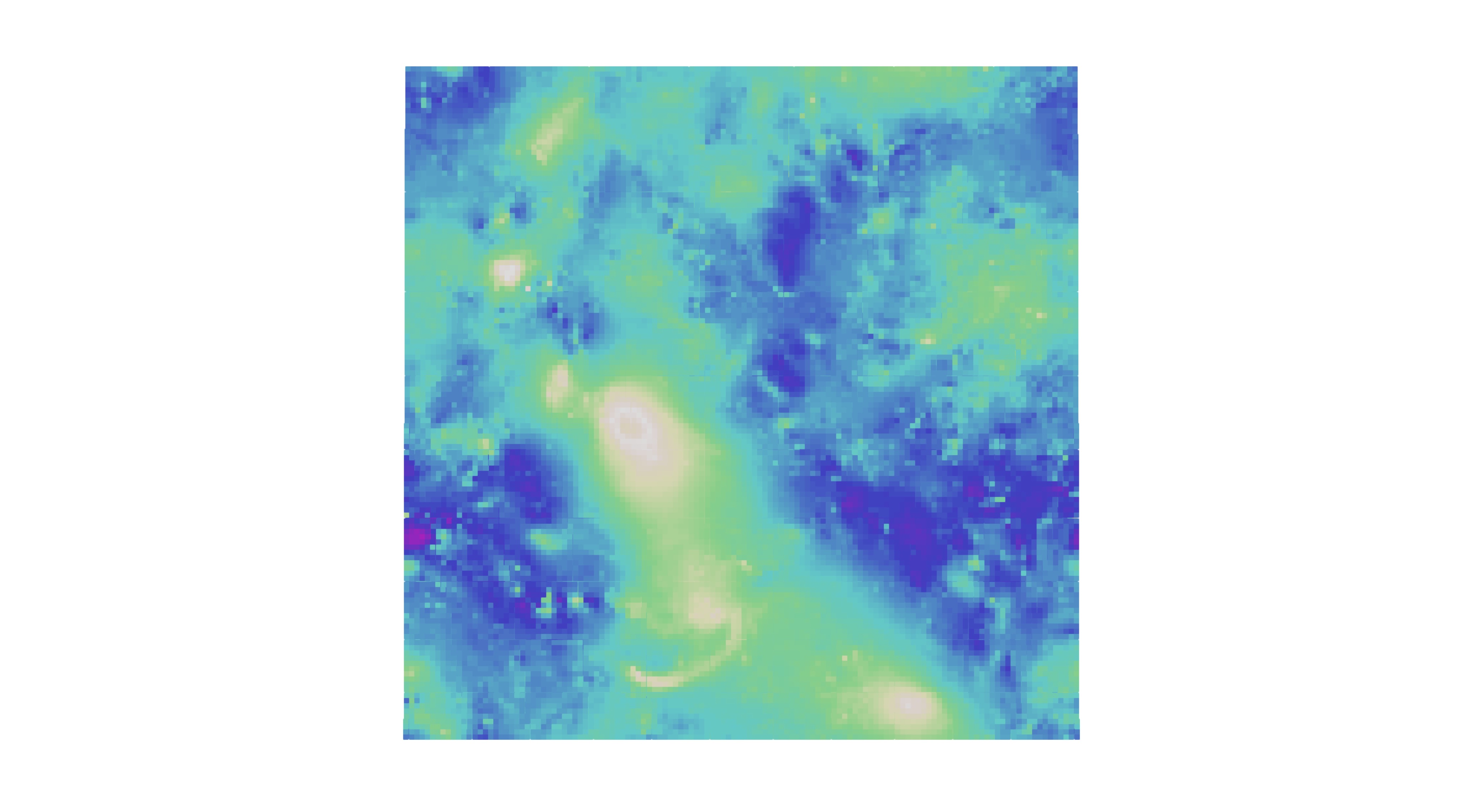}
\includegraphics[width=0.49\linewidth]{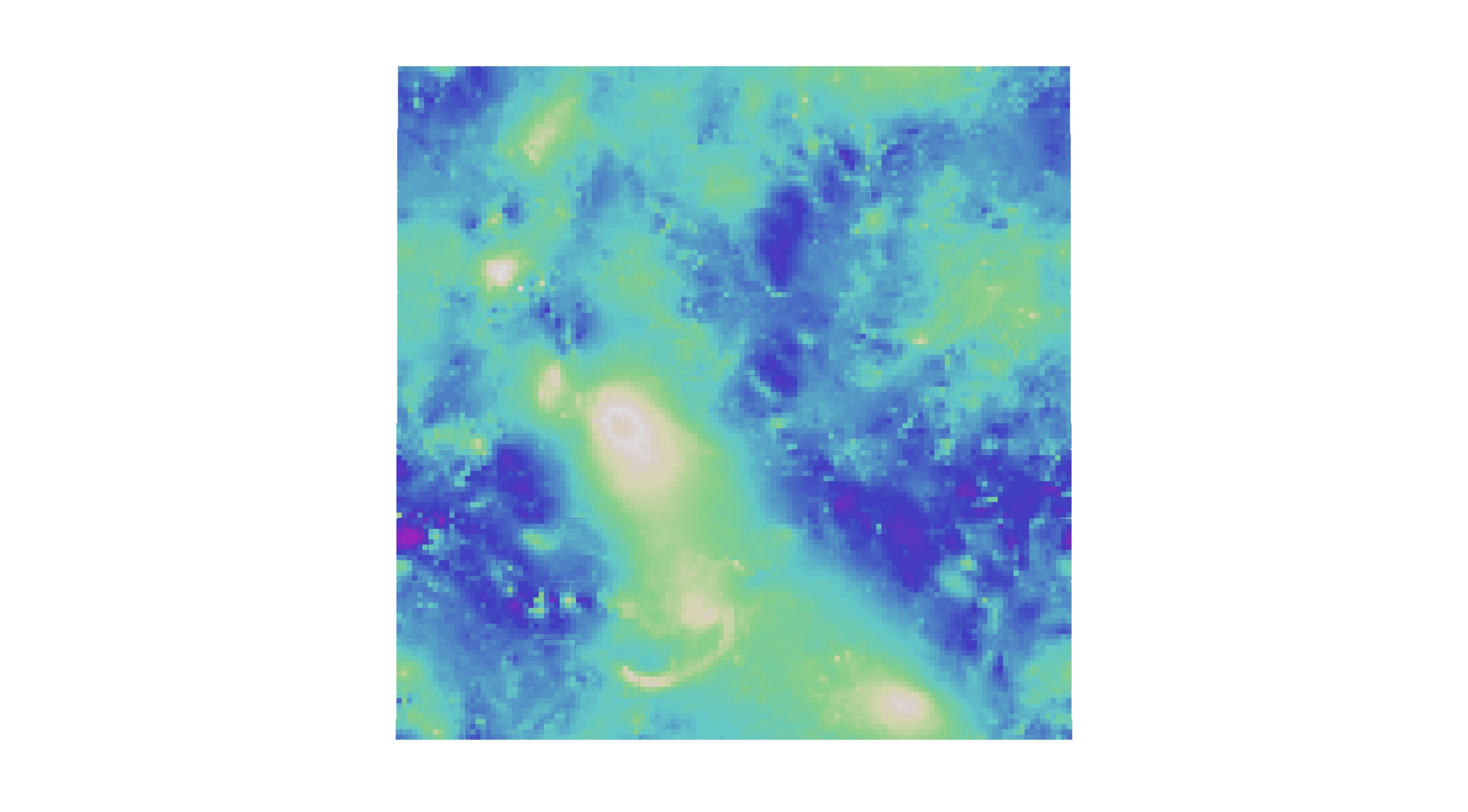}
\includegraphics[width=0.49\linewidth]{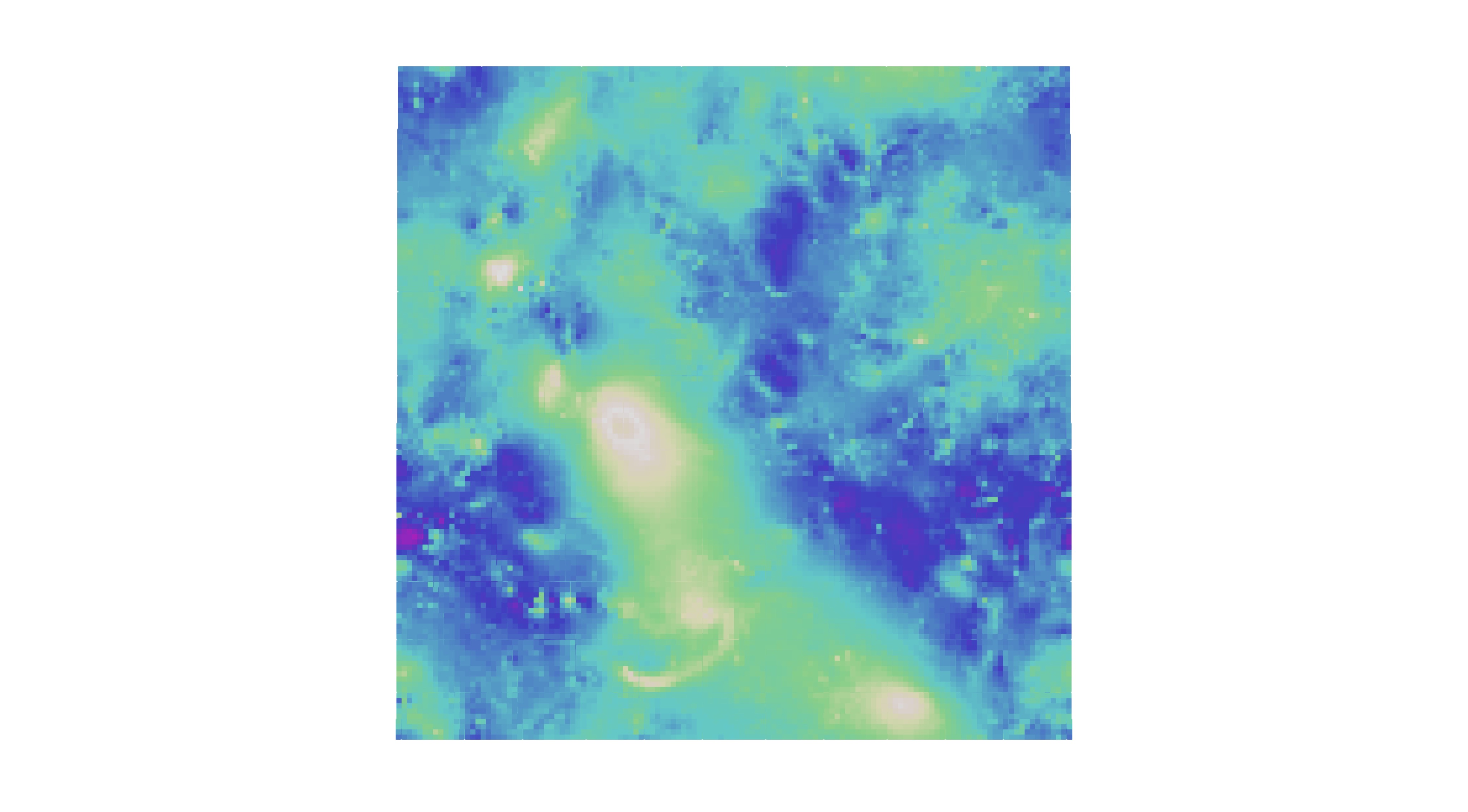}
\includegraphics[width=0.49\linewidth]{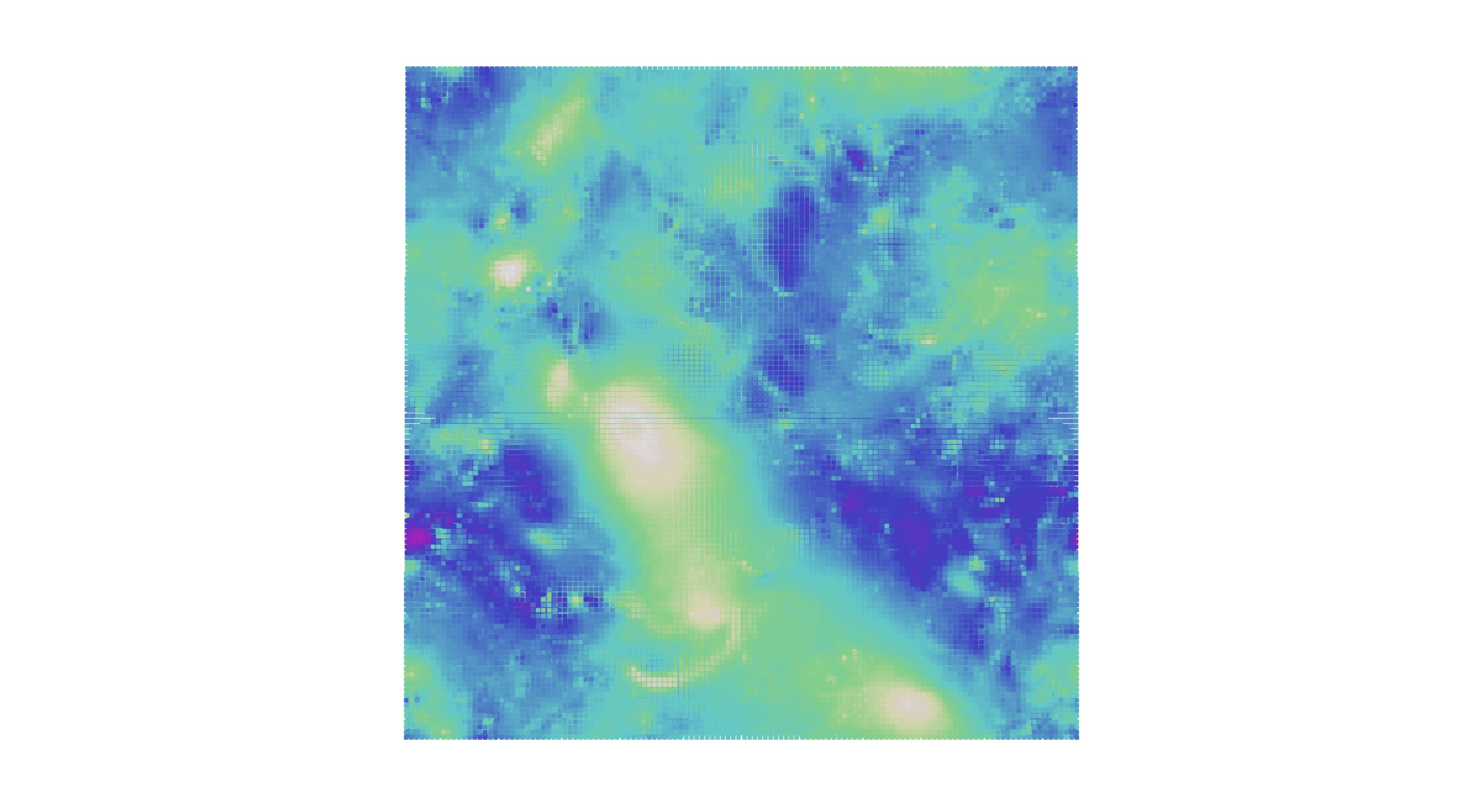}
\includegraphics[width=0.49\linewidth]{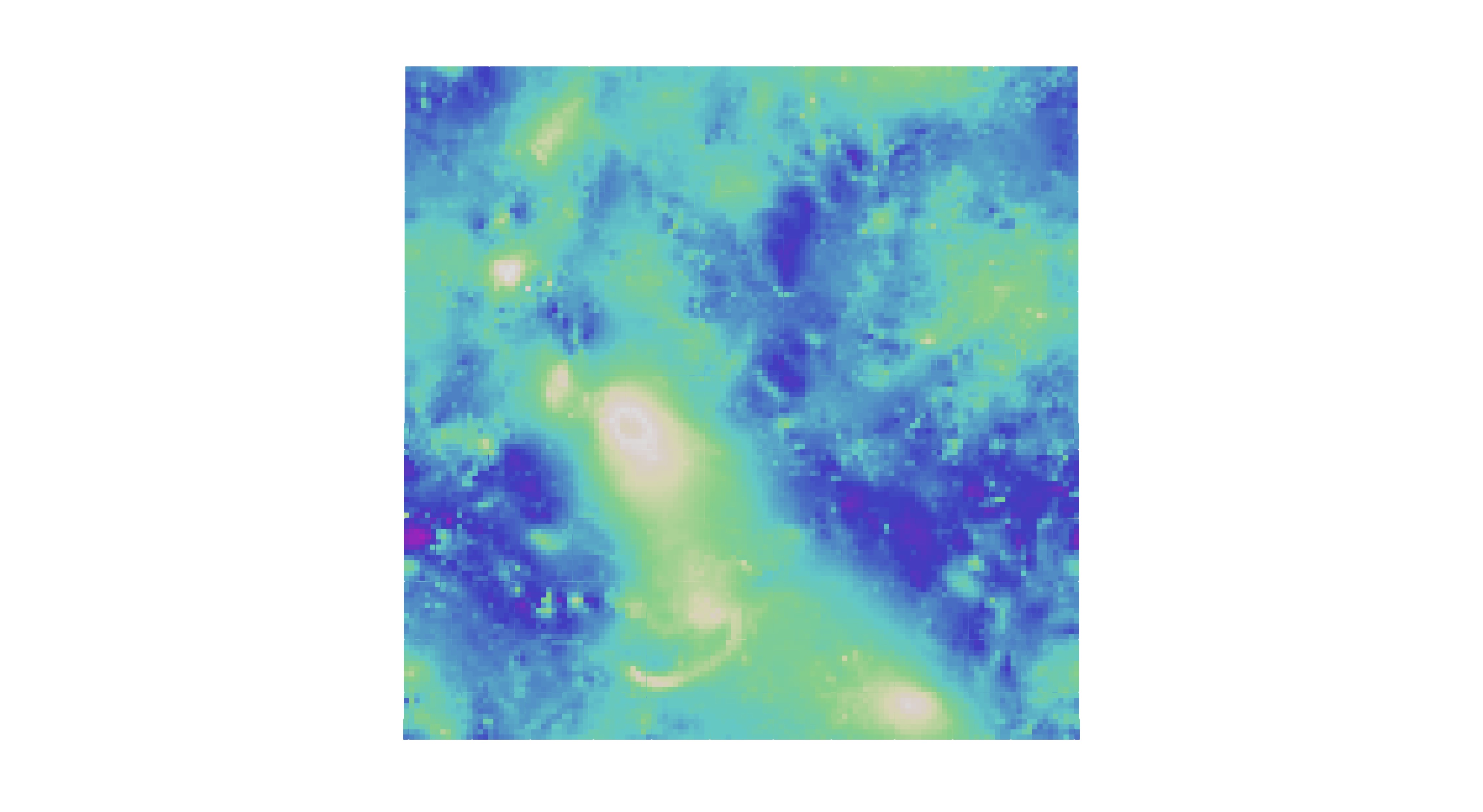}
\caption{Volumetric snapshots of the pressure field taken at the $42^{nd}$ time-step. From top to bottom: FPZIP, ZFP, and SZ. From left to right: ID and SVD. All snapshots are reconstructed from low-rank approximations which are accurate to almost 3 digits in terms of relative Frobenius error.}
 \label{fig:12snapshots}
 \end{figure}
 
 \begin{table}[t]
\centering
\tabcolsep7pt\begin{tabular}{|l|c|c|c|c|c|}
\hline
FPZIP (S) & ZFP (S) & SZ (S) & SPC-SVD (T) & SPC-ID (T)\\
\hline
(6.0,1.6e-3) & ($\bm{7.2}$,$\bm{9.3e}$-$\bm{4}$) & ($\bm{7.2}$,9.4e-4) & ($\bm{25.0}$,$\bm{1.3e}$-$\bm{3}$) & ($\bm{25.0}$,2.0e-3)\\
\hline
\end{tabular}
\caption{Spatial (S) and temporal (T) compression factors and errors for the 100 snapshot JHU isotropic turbulence dataset. Spatial compression factors are computed as the average ratio of the number of bytes required to store each of the 100 snapshots relative to the full size. Temporal compression factors are the ratio of the size of the original matrix to the sum of the sizes of the factor matrices. All errors are given in terms of the relative Frobenius error for the entire 100 snapshot data.} 
\label{tab:spatialcompression}
\end{table}

\begin{table}[t]
\centering
\tabcolsep7pt\begin{tabular}{|l|c|c|c|c|c|}
\hline
SVD+FPZIP & SVD+ZFP & SVD+SZ & ID+FPZIP & ID+ZFP & ID+SZ  \\
\hline
(110.8,3.8e-3) & (131.5,$\bm{1.4e}$-$\bm{3}$) & ($\bm{133.6}$,$\bm{1.4e}$-$\bm{3}$) & (122.3,$\bm{2.0e}$-$\bm{3}$) & ($\bm{185.5}$,2.2e-3) & (107.6,2.2e-3) \\
\hline
\end{tabular}
\caption{Spatio-temporal compression factors and errors for JHU isotropic turbulence dataset. Compression factors are computed as in~(\ref{eq:spatiotemporalCF}), errors are given in terms of the relative Frobenius error.} 
\label{tab:spatiotemporalcompression}
\end{table}

Table~\ref{tab:spatialcompression} provides the relative errors and spatial compression factors achieved using FPZIP, ZFP, and SZ on the 100 snapshots. Spatial compression factors are computed as the average compression factor over the 100 snapshots. Relative errors are reported in terms of the relative Frobenius error in approximation the entire data. FPZIP achieves a spatial compression factor of 6.0 across the 100 timesteps, with a reconstruction accuracy of $1.6 \times 10^{-3}$. ZFP outperforms FPZIP with a compression factor of 7.3 and a smaller reconstruction accuracy of $9.3 \times 10^{-4}$. Finally, SZ achieves a compression factor of 7.2 with a reconstruction error of $9.4 \times 10^{-4}$. Of the three methods, ZFP performs the best, achieving slightly better compression and reconstruction accuracy than SZ in this problem.

In the two right panels of Table~\ref{tab:spatialcompression} the relative errors and temporal compression factors using SPC-SVD and SPC-ID are reported. Both methods achieve a temporal compression factor of $25$, while SPC-SVD outperforms SPC-ID in terms of accuracy. However, SPC-ID produces a decomposition comprised of snapshots from the original data. This gives it an advantage over SPC-SVD in downstream spatial compression.

Table~\ref{tab:spatiotemporalcompression} provides the relative errors of SPC-ID and SPC-SVD used in combination with each of FPZIP, ZFP, and SZ. In Figure~\ref{fig:12snapshots}, the reconstructed snapshots from each of the six combinations are provided. Of all six combinations, the smallest error is achieved using SVD+ZFP, as well as SVD+SZ. SVD+SZ outperforms SVD+ZFP in terms of compression factor as well. When using SPC-ID, ZFP is the best of the methods in terms of both accuracy and compression. The factor matrix $\bm{B}$ in a row ID is comprised of solution snapshots, and therefore physically structured; this is precisely the type of data for which ZFP is designed. Across all cases, significant improvement over the spatial compression factors reported in Table~\ref{tab:spatialcompression} is achieved. This indicates that the spatio-temporal compression approached proposed is a useful tool in simulation data compression. By identifying low-rank structure in the temporal domain of the data, SPC-ID and SPC-SVD enable enhanced, {\it in situ} compression of the turbulence dataset at a minimal loss of accuracy.

\section{Proofs of main theoretical results}
\label{sec:proofs}
This section is devoted to proving Theorems~\ref{thm:tpsvdcoarseerror},~\ref{thm:SPCSVD},~\ref{thm:SPCIDbound}, and~\ref{thm:coarsepwritr}.

\subsection{Proof of Theorem~\ref{thm:tpsvdcoarseerror}} 
\label{sec:prooftpsvdcoarseerror}

Let the SVD of $\bm{A}_c = \bm{U}_{c}\bm{\Sigma}_{c}\bm{V}_{c}^T$, $\tilde{\bm{U}}_k = \bm{U}_c(:,1:k)$, and recall the operator $\bm{M}$ from Definition~\ref{def:interpop}. Then,
\begin{align}
\Vert \bm{A}_f - \tilde{\bm{U}}_k\tilde{\bm{U}}_k^T \bm{A}_f \Vert_2 
&= \Vert (\bm{I} - \tilde{\bm{U}}_k\tilde{\bm{U}}_k^T) \bm{A}_f \Vert_2, \\
&= \Vert (\bm{I} - \tilde{\bm{U}}_k\tilde{\bm{U}}_k^T) (\bm{A}_c\bm{M} + \bm{E}_I) \Vert_2,\\
&\leq \Vert \bm{M} \Vert_2  \sigma_{c,k+1}   + \Vert \bm{E}_I \Vert_2.
\label{eqn:rangeapprox}
\end{align}
The matrix $\tilde{\bm{V}}_k^T$ is then obtained,
\begin{equation}
\tilde{\bm{V}}_k^T = (\tilde{\bm{U}}_k\tilde{\bm{\Sigma}}_k)^+ \bm{A}_f.
\end{equation}
Then, 
\begin{align}
\Vert \bm{A}_f - \tilde{\bm{U}}_k\tilde{\bm{\Sigma}}_k\tilde{\bm{V}}^T_k \Vert_2 
&= \Vert \bm{A}_f - \tilde{\bm{U}}_k\tilde{\bm{\Sigma}}_k (\tilde{\bm{U}}_k\tilde{\bm{\Sigma}}_k)^+ \bm{A}_f \Vert_2, \\
&= \Vert \bm{A}_f - \tilde{\bm{U}}_k\tilde{\bm{U}}_k^T \bm{A}_f \Vert_2, \\
&\leq \Vert \bm{M} \Vert_2  \sigma_{c,k+1}   + \Vert \bm{E}_I \Vert_2. 
\end{align}
\subsection{Proof of Theorem~\ref{thm:SPCSVD}}
The main result of Section~\ref{sec:singlepass} guarantees the approximation accuracy of SPC-SVD. 
\label{sec:proofspcsvd}
Let the QR decomposition of the coarse grid data matrix be $\bm{A}_c = \bm{Q}_c\bm{R}_c$ and $\bm{M}$ be as in Definition~\ref{def:interpop}.
Then,
\begin{equation}
    \bm{A}_f = \bm{Q}_c\bm{R}_c\bm{M} + \bm{E}_I,
\end{equation}
which yields
\begin{subequations}
\begin{align}
    \Vert \bm{A}_f - \bm{Q}_c\bm{Q}_c^T \bm{A}_f \Vert_2 
    &= \Vert (\bm{I} - \bm{Q}_c\bm{Q}^T_c)(\bm{Q}_c\bm{R}_c\bm{M} + \bm{E}_I) \Vert_2,
    \\
    & = \Vert (\bm{I} - \bm{Q}_c\bm{Q}^T_c) \bm{E}_I \Vert_2,
    \\
    & \leq \Vert \bm{E}_I \Vert_2.
\end{align}
\end{subequations}
Let $\hat{\bm{A}}_{f,k}$ be the best rank-$k$ approximation to $\bm{A}_f$. Then,
\begin{subequations}
\begin{align}
 \Vert \bm{A}_f - \hat{\bm{A}}_f \Vert_2 
    & \leq  \Vert \bm{A}_f - \bm{Q}_c\bm{Q}_c^T \bm{A}_f \Vert_2 + \Vert \bm{Q}_c\bm{Q}_c^T \bm{A}_f - \hat{\bm{A}}_f \Vert_2,\\
    & \leq   \Vert \bm{E}_I \Vert_2 + \Vert \bm{Q}_c\bm{Q}^T_c(\bm{A}_f - \hat{\bm{A}}_{f,k}) \Vert_2,\label{ineq:bestkrankproj}\\
    & \leq \Vert \bm{E}_I \Vert_2 + \Vert \bm{A}_f - \hat{\bm{A}}_{f,k} \Vert_2,\\
    & = \sigma_{f,k+1} + \Vert \bm{E}_I \Vert_2.
\end{align}
\end{subequations}
\begin{remark}
The inequality (\ref{ineq:bestkrankproj}) follows from the fact that $\hat{\bm{A}}_f$ is the best rank-$k$ approximation to $\bm{Q}_c\bm{Q}_c^T \bm{A}_f$. $\bm{Q}_c\bm{Q}_c^T \hat{\bm{A}}_{f,k}$, a rank-$k$ matrix, can be no better of an approximation to $\bm{Q}_c\bm{Q}_c^T \bm{A}_f$ than $\hat{\bm{A}}_f$ is.
\end{remark}
\subsection{Proof of Theorem~\ref{thm:SPCIDbound}}
The following lemma is a restatement of a result from~\cite{hampton2018practical}.
\begin{lemma} (Lemma 4 from~\cite{hampton2018practical}.) Let $\tilde{\bm{U}}_r$ have columns comprising the $r$
left singular vectors corresponding to the $r$ largest singular values of $\bm{A}_c$, $\bm{P}_{\tilde{\mathcal{U}}_r}$ be the orthogonal projection onto the range of $\tilde{\bm{U}}_r$, $\bm{P}_{(\tilde{\mathcal{U}}_r)^{\perp}}$ be the orthogonal projection onto the orthogonal complement of the range of $\tilde{\bm{U}}_r$, $\sigma_{c,j}$ be the $j$th largest singular value of $\bm{A}_c$, and $\epsilon(\tau)$ defined as in (\ref{eqn:epstau_main_test}). Then, 
\begin{subequations}
\begin{align}
    \Vert \bm{P}_{(\tilde{\mathcal{U}}_r)^{\perp}} \bm{A}_f  \Vert_2 &\leq \left(\epsilon(\tau) + \tau \sigma_{c,r+1}^2 \right)^{1/2}, \\
\Vert \bm{A}_c^+\bm{P}_{\tilde{\mathcal{U}}_r}\bm{A}_f\Vert_2 &\leq \left(\tau + \epsilon(\tau) \sigma_{c,r}^{-2}\right)^{1/2}. 
\end{align}
\end{subequations}
\label{lem:SPIDlem1}
\end{lemma}
\label{sec:proofSPCID}
With this lemma established, the main result of Section~\ref{sec:SPCID} can be proven.
Given the SVD of $\bm{A}_c$,
\begin{subequations}
\begin{align}
\bm{A}_c &= \bm{U}_c\bm{S}_c\bm{V}_c^T,\\
\tilde{\bm{U}}_r &= \bm{U}_c(:,1:r).
\end{align}
\end{subequations}
The lifting operator $\bm{T}_r$ such that $\bm{A}_c\bm{T}_r \approx \bm{A}_f$ is then defined as follows.
\begin{align}
\bm{T}_r
&= \bm{A}_c^+\tilde{\bm{U}}_r\tilde{\bm{U}}^T_r\bm{A}_f.
\end{align}
Let $\bm{P}_c \bm{A}_c(\mathcal{I}_c,:)$ be the ID approximation of the coarse grid matrix $\bm{A}_c$. Then, 
\begin{align}
    \hat{\bm{A}}_f = \bm{P}_c \bm{A}_c(\mathcal{I}_c,:)\bm{T}_r.
\end{align}
Let $\bm{A}_c\bm{A}_c^+$ be the orthogonal projection onto the range of $\bm{A}_c$. Then,
\begin{subequations}
\begin{align}
    \Vert \bm{A}_f -  \hat{\bm{A}}_f\Vert_2 
    &\leq \Vert \bm{A}_f - \bm{A}_c \bm{T}_r \Vert_2 + \Vert   \bm{A}_c \bm{T}_r - \bm{P}_c  \bm{A}_c(\mathcal{I}_c,:)\bm{T}_r\Vert_2,
    \\
    &\leq \Vert \bm{A}_f - \bm{A}_c \bm{A}_c^+\tilde{\bm{U}}_r\tilde{\bm{U}}^T_r\bm{A}_f \Vert_2 + \Vert   \bm{A}_c \bm{A}_c^+\tilde{\bm{U}}_r\tilde{\bm{U}}^T_r\bm{A}_f - \bm{P}_c  \bm{A}_c(\mathcal{I}_c,:)\bm{A}_c^+\tilde{\bm{U}}_r\tilde{\bm{U}}^T_r\bm{A}_f\Vert_2,
    \\
     &\leq \Vert (\bm{I} - \bm{A}_c \bm{A}_c^+\bm{P}_{\tilde{\mathcal{U}}_r}) \bm{A}_f  \Vert_2 + \Vert \bm{A}_c - \bm{P}_c \bm{A}_c(\mathcal{I}_c,:) \Vert_2 \Vert \bm{A}_c^+\bm{P}_{\tilde{\mathcal{U}}_r} \bm{A}_f\Vert_2, 
    \\
    &\leq \Vert \bm{P}_{(\tilde{\mathcal{U}}_r)^{\perp}} \bm{A}_f  \Vert_2 + \Vert \bm{A}_c - \bm{P}_c \bm{A}_c(\mathcal{I}_c,:) \Vert_2 \Vert \bm{A}_c^+\bm{P}_{\tilde{\mathcal{U}}_r}\bm{A}_f\Vert_2 \label{ineq:thmwithoutlems}.
\end{align}
\end{subequations}

Combining Lemma~\ref{lem:SPIDlem1} with the inequality (\ref{ineq:thmwithoutlems}) yields Theorem~\ref{thm:SPCIDbound}.
\subsection{Proof of Theorem~\ref{thm:coarsepwritr}}
\label{sec:proofcoarsepoweriteration}
The proof begins with the following lemma, which quantifies the error incurred using C-PWR.
\begin{lemma}
\label{lem:powerlemma}
Let $\bm{A}_c$ be a coarse grid sketch of a fine grid data matrix $\bm{A}_f$, $\tau$ be a positive real number, $\rho(\tau)$ be defined as in Definition~\ref{def:rho_tau}, and $q$ be the number of power iterations used in C-PWR. Then,
\begin{align}
    \Vert \left(\bm{A}_f\bm{A}_f^T\right)^q\bm{A}_f - \left(\tau\bm{A}_c\bm{A}_c^T\right)^q\bm{A}_f \Vert_2 \lesssim q \tau^{q-1}\Vert \bm{A}_c \Vert_2^{2q-2} \Vert \bm{A}_f \Vert_2 \rho(\tau)  + \rho^2(\tau).
\end{align}
\end{lemma}
\begin{proof}
Let the matrix $\bm{E} =\bm{A}_f\bm{A}^T_f -  \tau\bm{A}_c\bm{A}^T_c$, with $\Vert \bm{E} \Vert_2 = \rho(\tau)$. Then,
\begin{subequations}
\begin{align}
    \Vert \left(\bm{A}_f\bm{A}_f^T\right)^q\bm{A}_f - \left(\tau\bm{A}_c\bm{A}_c^T\right)^q\bm{A}_f \Vert_2 
    &= \Vert \left(\tau\bm{A}_c\bm{A}_c^T + \bm{E} \right)^q\bm{A}_f - \left(\tau\bm{A}_c\bm{A}_c^T\right)^q\bm{A}_f \Vert_2,\\
    &\leq \Vert \left(\tau\bm{A}_c\bm{A}_c^T + \bm{E} \right)^q - \left(\tau\bm{A}_c\bm{A}_c^T\right)^q \Vert_2 \Vert \bm{A}_f \Vert_2,\\
    &\lesssim  q\tau^{q-1}  \Vert \bm{A}_c \Vert_2^{2q-2}  \Vert \bm{A}_f \Vert_2 \Vert  \bm{E} \Vert_2 + \Vert \bm{E} \Vert_2^2, \label{ineq:coarsespsvd} \\
    &= q\tau^{q-1} \Vert \bm{A}_c \Vert_2^{2q-2}  \Vert \bm{A}_f \Vert_2 \rho(\tau)+ \rho^2(\tau).
\end{align}
\end{subequations}
The asymptotic inequality (\ref{ineq:coarsespsvd}) follows from the binomial theorem and sub-multiplicativity of $\Vert \cdot \Vert_2$.
\end{proof}
\alec{The groundwork to state the desired result is now complete.} 
Following Theorem 9.3 of~\cite{halko2011finding}, let $\bm{P}_{\bm{A}^q_b}$ be a projection onto the range of $\bm{A}^q_b = (\tau\bm{A}_c\bm{A}_c^T)^q \bm{A}_f$ and \alec{let} $\bm{A}^q_f = (\bm{A}_f\bm{A}_f^T)^q \bm{A}_f$. Then,
\begin{subequations}
\begin{align}
    \Vert \bm{A}_f - \hat{\bm{A}}_f \Vert_2 
    &\leq \Vert \bm{A}_f - \bm{P}_{\bm{A}^q_b} \bm{A}_f \Vert_2 + \Vert \bm{P}_{\bm{A}^q_b} \bm{A}_f - \hat{\bm{A}}_f \Vert_2.
\end{align}
\end{subequations}
Let $\hat{\bm{A}}_{f,k}$ denote a best rank-$k$ approximation of the matrix $\bm{A}_f$. Because $\hat{\bm{A}}_f$ is the best rank-$k$ approximation to $\bm{P}_{\bm{A}^q_b} \bm{A}_f$,
\begin{subequations}
\begin{align}
    \Vert \bm{A}_f - \hat{\bm{A}}_f \Vert_2 
    &\leq \Vert \bm{P}_{\bm{A}^q_b} \bm{A}_f - \bm{P}_{\bm{A}^q_b}\hat{\bm{A}}_{f,k} \Vert_2 + \Vert (\bm{I} - \bm{P}_{\bm{A}^q_b}) \bm{A}_f \Vert_2 \label{eq:line3},\\
    &= \Vert \bm{P}_{\bm{A}^q_b}( \bm{A}_f - \hat{\bm{A}}_{f,k}) \Vert_2 + \Vert (\bm{I} - \bm{P}_{\bm{A}^q_b}) \bm{A}_f \Vert_2,\\
    &\leq \Vert  \bm{A}_f - \hat{\bm{A}}_{f,k} \Vert_2 + \Vert (\bm{I} - \bm{P}_{\bm{A}^q_b}) \bm{A}_f \Vert_2,\\
    &\leq \sigma_{f,k+1}  +  \Vert (\bm{I} - \bm{P}_{\bm{A}^q_b}) \bm{A}_f \Vert_2.
\end{align}
\end{subequations}
Now, following the proof of Theorem 9.2 from~\cite{halko2011finding}, 
\begin{subequations}
\begin{align}
    \Vert \bm{A}_f - \hat{\bm{A}}_f \Vert_2 
    &\leq \sigma_{f,k+1} +  \Vert (\bm{I} - \bm{P}_{\bm{A}^q_b})\bm{A}^q_f \Vert_2^{1/(2q+1)},\label{eq:line7}\\
    &= \sigma_{f,k+1} +  \Vert (\bm{I} - \bm{P}_{\bm{A}^q_b})(\bm{A}^q_b + \bm{E}) \Vert_2^{1/(2q+1)},\\
    &= \sigma_{f,k+1} +  \Vert( \bm{P}_{(\bm{A}^q_b)^T}  \bm{A}^q_b + \bm{P}_{(\bm{A}^q_b)^T} \bm{E}) \Vert_2^{1/(2q+1)},\\
    &\leq \sigma_{f,k+1} +  \left(\Vert \bm{P}_{(\bm{A}^q_b)^T}  \bm{A}^q_b \Vert_2 + \Vert \bm{P}_{(\bm{A}^q_b)^T} \bm{E} \Vert_2 \right)^{1/(2q+1)},\\
    &\leq \sigma_{f,k+1} +  \left(\Vert \bm{P}_{(\bm{A}^q_b)^T} \bm{A}^q_b \Vert_2 + \Vert \bm{E} \Vert_2\right)^{1/(2q+1)}, \\
    &\leq \sigma_{f,k+1} +  \Vert \bm{E} \Vert_2^{1/(2q+1)}, \\
 &\lesssim \sigma_{f,k+1} +  \left(q \tau^{q-1}\Vert \bm{A}_c \Vert_2^{2q-2} \Vert \bm{A}_f \Vert_2  \rho(\tau)  + \rho^2(\tau)\right)^{1/(2q+1)} \label{eq:finalineq}.
\end{align}
\end{subequations}
\alec{This bound indicates that the sub-optimality due to C-PWR increases as the number of power iterations $q$, the scaling parameter $\tau$, or the approximation parameter $\rho(\tau)$ increases. On the other hand, for small $\rho(\tau)$, the error is expected to be small and scale with $\rho(\tau)^{1/(2q+1)}$. }
\section{Conclusions}
\label{sec:conclusions}
This work presents pass-efficient algorithms for computing low-rank SVD and \alec{Interpolative Decomposition (ID)} matrix approximations of high dimensional spatio-temporal data matrices. The proposed algorithms use a coarsened (sketched) data matrix to compute a decomposition of fine grid (unsketched) data matrix. A first-of-its-kind single-pass power iteration algorithm is also presented.

\alec{As opposed to the randomized matrix sketching literature, this work is focused on {\it deterministic} sketches. With data assumed to come from physical simulations, properties of these datasets can be exploited to obtain a faster, more memory movement-efficient sketch with deterministic approaches. While the data-agnosticism of randomized methods is an attractive feature, the proposed sketches in this work demonstrate competitiveness with a standard randomized approach (a dense Gaussian matrix). Deterministic approaches also enable the construction of single-pass algorithms for both low-rank approximation and online error estimation for which randomized approaches were observed to be less effective.}

In the single-pass case, the sub-optimality of the SVD algorithms relative to the oracle solution depends on the error incurred by mapping the coarse grid data on the fine grid. In the ID algorithm, existing analysis from~\cite{hampton2018practical} aids in the derivation of error bounds. The single-pass algorithm SPC-ID has a bound competitive with an analogous two-pass algorithm presented in~\cite{dunton2020pass}, while an algorithm based on previous work in~\cite{hampton2018practical} enables online estimation of the approximation error. In the case of coarse grid power iteration, the sub-optimality bound depends on an intuitive relationship between the coarse grid and fine grid data.

Numerical experiments demonstrate that the proposed deterministic sketches enable near-optimal low-rank approximation of data matrices via the singular value decomposition and interpolative decomposition. The proposed coarse-grid power iteration scheme C-PWR reduces approximation error with minimal trade-offs in runtime and importantly, no additional passes over the \alec{high-dimensional} input matrix. 
Combining these low-rank methods with three state-of-the-art compression algorithms FPZIP, ZFP, and SZ on a pressure dataset from the Johns Hopkins University Turbulence Databases enables spatio-temporal compression. 

Extending the algorithms in this work to other low-rank approximations of data matrices, including but not limited to, CUR and Cholesky decompositions, is a straightforward augmentation to this work. In addition, tighter error bounds on coarse grid power iteration are an existing gap not just in this work, but the power iteration literature in general. Finally, generalizing the proposed methods to high-order tensor decomposition algorithms would be a natural next step.


\section*{Acknowledgements}
The authors would like to thank Ryan Skinner for the NACA airfoil dataset, and Llu\'is Jofre for the turbulent channel flow dataset. 
This work was funded by the United States Department of Energy's National Nuclear Security Administration under the Predictive Science Academic Alliance Program (PSAAP) II at Stanford University, Grant DE-NA-0002373. \alec{The work of AD was also supported by the AFOSR grant FA9550-20-1-0138.}

\bibliography{Ref}

\begin{thebibliography}{47}
\expandafter\ifx\csname natexlab\endcsname\relax\def\natexlab#1{#1}\fi
\providecommand{\bibinfo}[2]{#2}
\ifx\xfnm\relax \def\xfnm[#1]{\unskip,\space#1}\fi
\bibitem[{Dunton et~al.(2020)Dunton, Jofre, Iaccarino, and
  Doostan}]{dunton2020pass}
\bibinfo{author}{A.~M. Dunton}, \bibinfo{author}{L.~Jofre},
  \bibinfo{author}{G.~Iaccarino}, \bibinfo{author}{A.~Doostan},
\newblock \bibinfo{title}{Pass-efficient methods for compression of
  high-dimensional turbulent flow data},
\newblock \bibinfo{journal}{Journal of Computational Physics}
  \bibinfo{volume}{423} (\bibinfo{year}{2020}) \bibinfo{pages}{109704}.
\bibitem[{Eckart and Young(1936)}]{eckart1936approximation}
\bibinfo{author}{C.~Eckart}, \bibinfo{author}{G.~Young},
\newblock \bibinfo{title}{The approximation of one matrix by another of lower
  rank},
\newblock \bibinfo{journal}{Psychometrika} \bibinfo{volume}{1}
  (\bibinfo{year}{1936}) \bibinfo{pages}{211--218}.
\bibitem[{Cheng et~al.(2005)Cheng, Gimbutas, Martinsson, and
  Rokhlin}]{cheng2005compression}
\bibinfo{author}{H.~Cheng}, \bibinfo{author}{Z.~Gimbutas},
  \bibinfo{author}{P.-G. Martinsson}, \bibinfo{author}{V.~Rokhlin},
\newblock \bibinfo{title}{On the compression of low rank matrices},
\newblock \bibinfo{journal}{SIAM Journal on Scientific Computing}
  \bibinfo{volume}{26} (\bibinfo{year}{2005}) \bibinfo{pages}{1389--1404}.
\bibitem[{Papadimitriou et~al.(2000)Papadimitriou, Raghavan, Tamaki, and
  Vempala}]{papadimitriou2000latent}
\bibinfo{author}{C.~H. Papadimitriou}, \bibinfo{author}{P.~Raghavan},
  \bibinfo{author}{H.~Tamaki}, \bibinfo{author}{S.~Vempala},
\newblock \bibinfo{title}{Latent semantic indexing: A probabilistic analysis},
\newblock \bibinfo{journal}{Journal of Computer and System Sciences}
  \bibinfo{volume}{61} (\bibinfo{year}{2000}) \bibinfo{pages}{217--235}.
\bibitem[{Achlioptas and McSherry(2007)}]{achlioptas2007fast}
\bibinfo{author}{D.~Achlioptas}, \bibinfo{author}{F.~McSherry},
\newblock \bibinfo{title}{Fast computation of low-rank matrix approximations},
\newblock \bibinfo{journal}{Journal of the ACM (JACM)} \bibinfo{volume}{54}
  (\bibinfo{year}{2007}) \bibinfo{pages}{9}.
\bibitem[{Arora et~al.(2006)Arora, Hazan, and Kale}]{arora2006fast}
\bibinfo{author}{S.~Arora}, \bibinfo{author}{E.~Hazan},
  \bibinfo{author}{S.~Kale},
\newblock \bibinfo{title}{A fast random sampling algorithm for sparsifying
  matrices},
\newblock in: \bibinfo{booktitle}{Approximation, Randomization, and
  Combinatorial Optimization. Algorithms and Techniques},
  \bibinfo{publisher}{Springer}, \bibinfo{year}{2006}, pp.
  \bibinfo{pages}{272--279}.
\bibitem[{Gittens and Tropp(2009)}]{gittens2009error}
\bibinfo{author}{A.~Gittens}, \bibinfo{author}{J.~A. Tropp},
\newblock \bibinfo{title}{Error bounds for random matrix approximation
  schemes},
\newblock \bibinfo{journal}{arXiv preprint arXiv:0911.4108}
  (\bibinfo{year}{2009}).
\bibitem[{Spielman and Srivastava(2011)}]{spielman2011graph}
\bibinfo{author}{D.~A. Spielman}, \bibinfo{author}{N.~Srivastava},
\newblock \bibinfo{title}{Graph sparsification by effective resistances},
\newblock \bibinfo{journal}{SIAM Journal on Computing} \bibinfo{volume}{40}
  (\bibinfo{year}{2011}) \bibinfo{pages}{1913--1926}.
\bibitem[{Woolfe et~al.(2008)Woolfe, Liberty, Rokhlin, and
  Tygert}]{woolfe2008fast}
\bibinfo{author}{F.~Woolfe}, \bibinfo{author}{E.~Liberty},
  \bibinfo{author}{V.~Rokhlin}, \bibinfo{author}{M.~Tygert},
\newblock \bibinfo{title}{A fast randomized algorithm for the approximation of
  matrices},
\newblock \bibinfo{journal}{Applied and Computational Harmonic Analysis}
  \bibinfo{volume}{25} (\bibinfo{year}{2008}) \bibinfo{pages}{335--366}.
\bibitem[{Tropp(2011)}]{tropp2011improved}
\bibinfo{author}{J.~A. Tropp},
\newblock \bibinfo{title}{Improved analysis of the subsampled randomized
  hadamard transform},
\newblock \bibinfo{journal}{Advances in Adaptive Data Analysis}
  \bibinfo{volume}{3} (\bibinfo{year}{2011}) \bibinfo{pages}{115--126}.
\bibitem[{Gu(2015)}]{gu2015subspace}
\bibinfo{author}{M.~Gu},
\newblock \bibinfo{title}{Subspace iteration randomization and singular value
  problems},
\newblock \bibinfo{journal}{SIAM Journal on Scientific Computing}
  \bibinfo{volume}{37} (\bibinfo{year}{2015}) \bibinfo{pages}{A1139--A1173}.
\bibitem[{Liberty et~al.(2007)Liberty, Woolfe, Martinsson, Rokhlin, and
  Tygert}]{liberty2007randomized}
\bibinfo{author}{E.~Liberty}, \bibinfo{author}{F.~Woolfe},
  \bibinfo{author}{P.-G. Martinsson}, \bibinfo{author}{V.~Rokhlin},
  \bibinfo{author}{M.~Tygert},
\newblock \bibinfo{title}{Randomized algorithms for the low-rank approximation
  of matrices},
\newblock \bibinfo{journal}{Proceedings of the National Academy of Sciences}
  \bibinfo{volume}{104} (\bibinfo{year}{2007}) \bibinfo{pages}{20167--20172}.
\bibitem[{Dasgupta and Gupta(2003)}]{dasgupta2003elementary}
\bibinfo{author}{S.~Dasgupta}, \bibinfo{author}{A.~Gupta},
\newblock \bibinfo{title}{An elementary proof of a theorem of johnson and
  lindenstrauss},
\newblock \bibinfo{journal}{Random Structures \& Algorithms}
  \bibinfo{volume}{22} (\bibinfo{year}{2003}) \bibinfo{pages}{60--65}.
\bibitem[{Halko et~al.(2011)Halko, Martinsson, and Tropp}]{halko2011finding}
\bibinfo{author}{N.~Halko}, \bibinfo{author}{P.-G. Martinsson},
  \bibinfo{author}{J.~A. Tropp},
\newblock \bibinfo{title}{Finding structure with randomness: Probabilistic
  algorithms for constructing approximate matrix decompositions},
\newblock \bibinfo{journal}{SIAM review} \bibinfo{volume}{53}
  (\bibinfo{year}{2011}) \bibinfo{pages}{217--288}.
\bibitem[{Ailon and Chazelle(2009)}]{ailon2009fast}
\bibinfo{author}{N.~Ailon}, \bibinfo{author}{B.~Chazelle},
\newblock \bibinfo{title}{The fast johnson--lindenstrauss transform and
  approximate nearest neighbors},
\newblock \bibinfo{journal}{SIAM Journal on computing} \bibinfo{volume}{39}
  (\bibinfo{year}{2009}) \bibinfo{pages}{302--322}.
\bibitem[{Cannings and Samworth(2015)}]{cannings2015random}
\bibinfo{author}{T.~I. Cannings}, \bibinfo{author}{R.~J. Samworth},
\newblock \bibinfo{title}{Random-projection ensemble classification},
\newblock \bibinfo{journal}{arXiv preprint arXiv:1504.04595}
  (\bibinfo{year}{2015}).
\bibitem[{Johnson and Lindenstrauss(1984)}]{johnson1984extensions}
\bibinfo{author}{W.~B. Johnson}, \bibinfo{author}{J.~Lindenstrauss},
\newblock \bibinfo{title}{Extensions of lipschitz mappings into a hilbert
  space},
\newblock \bibinfo{journal}{Contemporary mathematics} \bibinfo{volume}{26}
  (\bibinfo{year}{1984}) \bibinfo{pages}{1}.
\bibitem[{Shitov(2017)}]{shitov2017column}
\bibinfo{author}{Y.~Shitov},
\newblock \bibinfo{title}{Column subset selection is np-complete},
\newblock \bibinfo{journal}{arXiv preprint arXiv:1701.02764}
  (\bibinfo{year}{2017}).
\bibitem[{Frieze et~al.(2004)Frieze, Kannan, and Vempala}]{frieze2004fast}
\bibinfo{author}{A.~Frieze}, \bibinfo{author}{R.~Kannan},
  \bibinfo{author}{S.~Vempala},
\newblock \bibinfo{title}{Fast monte-carlo algorithms for finding low-rank
  approximations},
\newblock \bibinfo{journal}{Journal of the ACM (JACM)} \bibinfo{volume}{51}
  (\bibinfo{year}{2004}) \bibinfo{pages}{1025--1041}.
\bibitem[{Drineas et~al.(2004)Drineas, Frieze, Kannan, Vempala, and
  Vinay}]{drineas2004clustering}
\bibinfo{author}{P.~Drineas}, \bibinfo{author}{A.~Frieze},
  \bibinfo{author}{R.~Kannan}, \bibinfo{author}{S.~Vempala},
  \bibinfo{author}{V.~Vinay},
\newblock \bibinfo{title}{Clustering large graphs via the singular value
  decomposition},
\newblock \bibinfo{journal}{Machine learning} \bibinfo{volume}{56}
  (\bibinfo{year}{2004}) \bibinfo{pages}{9--33}.
\bibitem[{Drineas et~al.(2006)Drineas, Kannan, and Mahoney}]{drineas2006fast}
\bibinfo{author}{P.~Drineas}, \bibinfo{author}{R.~Kannan},
  \bibinfo{author}{M.~W. Mahoney},
\newblock \bibinfo{title}{Fast monte carlo algorithms for matrices ii:
  Computing a low-rank approximation to a matrix},
\newblock \bibinfo{journal}{SIAM Journal on computing} \bibinfo{volume}{36}
  (\bibinfo{year}{2006}) \bibinfo{pages}{158--183}.
\bibitem[{Liberty(2013)}]{liberty2013simple}
\bibinfo{author}{E.~Liberty},
\newblock \bibinfo{title}{Simple and deterministic matrix sketching},
\newblock in: \bibinfo{booktitle}{Proceedings of the 19th ACM SIGKDD
  international conference on Knowledge discovery and data mining},
  \bibinfo{organization}{ACM}, pp. \bibinfo{pages}{581--588}.
\bibitem[{Ghashami et~al.(2016)Ghashami, Liberty, Phillips, and
  Woodruff}]{ghashami2016frequent}
\bibinfo{author}{M.~Ghashami}, \bibinfo{author}{E.~Liberty},
  \bibinfo{author}{J.~M. Phillips}, \bibinfo{author}{D.~P. Woodruff},
\newblock \bibinfo{title}{Frequent directions: Simple and deterministic matrix
  sketching},
\newblock \bibinfo{journal}{SIAM Journal on Computing} \bibinfo{volume}{45}
  (\bibinfo{year}{2016}) \bibinfo{pages}{1762--1792}.
\bibitem[{Clarkson and Woodruff(2009)}]{clarkson2009numerical}
\bibinfo{author}{K.~L. Clarkson}, \bibinfo{author}{D.~P. Woodruff},
\newblock \bibinfo{title}{Numerical linear algebra in the streaming model},
\newblock in: \bibinfo{booktitle}{Proceedings of the forty-first annual ACM
  symposium on Theory of computing}, \bibinfo{organization}{ACM}, pp.
  \bibinfo{pages}{205--214}.
\bibitem[{Woodruff et~al.(2014)}]{woodruff2014sketching}
\bibinfo{author}{D.~P. Woodruff}, et~al.,
\newblock \bibinfo{title}{Sketching as a tool for numerical linear algebra},
\newblock \bibinfo{journal}{Foundations and Trends{\textregistered} in
  Theoretical Computer Science} \bibinfo{volume}{10} (\bibinfo{year}{2014})
  \bibinfo{pages}{1--157}.
\bibitem[{Boutsidis et~al.(2016)Boutsidis, Woodruff, and
  Zhong}]{boutsidis2016optimal}
\bibinfo{author}{C.~Boutsidis}, \bibinfo{author}{D.~P. Woodruff},
  \bibinfo{author}{P.~Zhong},
\newblock \bibinfo{title}{Optimal principal component analysis in distributed
  and streaming models},
\newblock in: \bibinfo{booktitle}{Proceedings of the forty-eighth annual ACM
  symposium on Theory of Computing}, \bibinfo{organization}{ACM}, pp.
  \bibinfo{pages}{236--249}.
\bibitem[{Upadhyay(2016)}]{upadhyay2016fast}
\bibinfo{author}{J.~Upadhyay},
\newblock \bibinfo{title}{Fast and space-optimal low-rank factorization in the
  streaming model with application in differential privacy},
\newblock \bibinfo{journal}{arXiv preprint arXiv:1604.01429}
  (\bibinfo{year}{2016}).
\bibitem[{Tropp et~al.(2017)Tropp, Yurtsever, Udell, and
  Cevher}]{tropp2017practical}
\bibinfo{author}{J.~A. Tropp}, \bibinfo{author}{A.~Yurtsever},
  \bibinfo{author}{M.~Udell}, \bibinfo{author}{V.~Cevher},
\newblock \bibinfo{title}{Practical sketching algorithms for low-rank matrix
  approximation},
\newblock \bibinfo{journal}{SIAM Journal on Matrix Analysis and Applications}
  \bibinfo{volume}{38} (\bibinfo{year}{2017}) \bibinfo{pages}{1454--1485}.
\bibitem[{Yu et~al.(2017)Yu, Gu, Li, Liu, and Li}]{yu2017single}
\bibinfo{author}{W.~Yu}, \bibinfo{author}{Y.~Gu}, \bibinfo{author}{J.~Li},
  \bibinfo{author}{S.~Liu}, \bibinfo{author}{Y.~Li},
\newblock \bibinfo{title}{Single-pass pca of large high-dimensional data},
\newblock \bibinfo{journal}{arXiv preprint arXiv:1704.07669}
  (\bibinfo{year}{2017}).
\bibitem[{Lindstrom(2014)}]{lindstrom2014fixed}
\bibinfo{author}{P.~Lindstrom},
\newblock \bibinfo{title}{Fixed-rate compressed floating-point arrays},
\newblock \bibinfo{journal}{IEEE transactions on visualization and computer
  graphics} \bibinfo{volume}{20} (\bibinfo{year}{2014})
  \bibinfo{pages}{2674--2683}.
\bibitem[{Lindstrom and Isenburg(2006)}]{lindstrom2006fast}
\bibinfo{author}{P.~Lindstrom}, \bibinfo{author}{M.~Isenburg},
\newblock \bibinfo{title}{Fast and efficient compression of floating-point
  data},
\newblock \bibinfo{journal}{IEEE transactions on visualization and computer
  graphics} \bibinfo{volume}{12} (\bibinfo{year}{2006})
  \bibinfo{pages}{1245--1250}.
\bibitem[{Di and Cappello(2016)}]{di2016fast}
\bibinfo{author}{S.~Di}, \bibinfo{author}{F.~Cappello},
\newblock \bibinfo{title}{Fast error-bounded lossy hpc data compression with
  sz},
\newblock in: \bibinfo{booktitle}{2016 ieee international parallel and
  distributed processing symposium (ipdps)}, \bibinfo{organization}{IEEE}, pp.
  \bibinfo{pages}{730--739}.
\bibitem[{Tao et~al.(2017)Tao, Di, Chen, and Cappello}]{tao2017significantly}
\bibinfo{author}{D.~Tao}, \bibinfo{author}{S.~Di}, \bibinfo{author}{Z.~Chen},
  \bibinfo{author}{F.~Cappello},
\newblock \bibinfo{title}{Significantly improving lossy compression for
  scientific data sets based on multidimensional prediction and
  error-controlled quantization},
\newblock in: \bibinfo{booktitle}{2017 IEEE International Parallel and
  Distributed Processing Symposium (IPDPS)}, \bibinfo{organization}{IEEE}, pp.
  \bibinfo{pages}{1129--1139}.
\bibitem[{Liang et~al.(2018)Liang, Di, Tao, Li, Li, Guo, Chen, and
  Cappello}]{liang2018error}
\bibinfo{author}{X.~Liang}, \bibinfo{author}{S.~Di}, \bibinfo{author}{D.~Tao},
  \bibinfo{author}{S.~Li}, \bibinfo{author}{S.~Li}, \bibinfo{author}{H.~Guo},
  \bibinfo{author}{Z.~Chen}, \bibinfo{author}{F.~Cappello},
\newblock \bibinfo{title}{Error-controlled lossy compression optimized for high
  compression ratios of scientific datasets},
\newblock in: \bibinfo{booktitle}{2018 IEEE International Conference on Big
  Data (Big Data)}, \bibinfo{organization}{IEEE}, pp.
  \bibinfo{pages}{438--447}.
\bibitem[{Golub and Van~Loan(2012)}]{golub2012matrix}
\bibinfo{author}{G.~H. Golub}, \bibinfo{author}{C.~F. Van~Loan},
  \bibinfo{title}{Matrix computations}, volume~\bibinfo{volume}{3},
  \bibinfo{publisher}{JHU press}, \bibinfo{year}{2012}.
\bibitem[{Hampton et~al.(2018)Hampton, Fairbanks, Narayan, and
  Doostan}]{hampton2018practical}
\bibinfo{author}{J.~Hampton}, \bibinfo{author}{H.~R. Fairbanks},
  \bibinfo{author}{A.~Narayan}, \bibinfo{author}{A.~Doostan},
\newblock \bibinfo{title}{Practical error bounds for a non-intrusive
  bi-fidelity approach to parametric/stochastic model reduction},
\newblock \bibinfo{journal}{Journal of Computational Physics}
  \bibinfo{volume}{368} (\bibinfo{year}{2018}) \bibinfo{pages}{315--332}.
\bibitem[{Hansen(1987)}]{hansen1987truncatedsvd}
\bibinfo{author}{P.~C. Hansen},
\newblock \bibinfo{title}{The truncatedsvd as a method for regularization},
\newblock \bibinfo{journal}{BIT Numerical Mathematics} \bibinfo{volume}{27}
  (\bibinfo{year}{1987}) \bibinfo{pages}{534--553}.
\bibitem[{Martinsson et~al.(2010)Martinsson, Szlam, Tygert
  et~al.}]{martinsson2010normalized}
\bibinfo{author}{P.-G. Martinsson}, \bibinfo{author}{A.~Szlam},
  \bibinfo{author}{M.~Tygert}, et~al.,
\newblock \bibinfo{title}{Normalized power iterations for the computation of
  svd},
\newblock \bibinfo{journal}{Manuscript., Nov}  (\bibinfo{year}{2010}).
\bibitem[{Skinner et~al.(2019)Skinner, Doostan, Peters, Evans, and
  Jansen}]{skinner2019reduced}
\bibinfo{author}{R.~W. Skinner}, \bibinfo{author}{A.~Doostan},
  \bibinfo{author}{E.~L. Peters}, \bibinfo{author}{J.~A. Evans},
  \bibinfo{author}{K.~E. Jansen},
\newblock \bibinfo{title}{Reduced-basis multifidelity approach for efficient
  parametric study of naca airfoils},
\newblock \bibinfo{journal}{AIAA Journal} \bibinfo{volume}{57}
  (\bibinfo{year}{2019}) \bibinfo{pages}{1481--1491}.
\bibitem[{Spalart and Allmaras(1992)}]{spalart1992one}
\bibinfo{author}{P.~Spalart}, \bibinfo{author}{S.~Allmaras},
\newblock \bibinfo{title}{A one-equation turbulence model for aerodynamic
  flows},
\newblock in: \bibinfo{booktitle}{30th aerospace sciences meeting and exhibit},
  p. \bibinfo{pages}{439}.
\bibitem[{Skinner et~al.(2017)Skinner, Doostan, Peters, Evans, and
  Jansen}]{skinner2017evaluation}
\bibinfo{author}{R.~Skinner}, \bibinfo{author}{A.~Doostan},
  \bibinfo{author}{E.~Peters}, \bibinfo{author}{J.~Evans},
  \bibinfo{author}{K.~E. Jansen},
\newblock \bibinfo{title}{An evaluation of multi-fidelity modeling efficiency
  on a parametric study of naca airfoils},
\newblock in: \bibinfo{booktitle}{35th AIAA Applied Aerodynamics Conference},
  p. \bibinfo{pages}{3260}.
\bibitem[{Esmaily et~al.(2018)Esmaily, Jofre, Mani, and
  Iaccarino}]{esmaily2018scalable}
\bibinfo{author}{M.~Esmaily}, \bibinfo{author}{L.~Jofre},
  \bibinfo{author}{A.~Mani}, \bibinfo{author}{G.~Iaccarino},
\newblock \bibinfo{title}{A scalable geometric multigrid solver for
  nonsymmetric elliptic systems with application to variable-density flows},
\newblock \bibinfo{journal}{Journal of Computational Physics}
  \bibinfo{volume}{357} (\bibinfo{year}{2018}) \bibinfo{pages}{142--158}.
\bibitem[{Moser et~al.(1999)Moser, Kim, and Mansour}]{moser1999direct}
\bibinfo{author}{R.~D. Moser}, \bibinfo{author}{J.~Kim}, \bibinfo{author}{N.~N.
  Mansour},
\newblock \bibinfo{title}{Direct numerical simulation of turbulent channel flow
  up to re $\tau$= 590},
\newblock \bibinfo{journal}{Physics of fluids} \bibinfo{volume}{11}
  (\bibinfo{year}{1999}) \bibinfo{pages}{943--945}.
\bibitem[{Martinsson(2019)}]{martinsson2019randomized}
\bibinfo{author}{P.-G. Martinsson},
\newblock \bibinfo{title}{Randomized methods for matrix computations},
\newblock \bibinfo{journal}{The Mathematics of Data} \bibinfo{volume}{25}
  (\bibinfo{year}{2019}) \bibinfo{pages}{187--231}.
\bibitem[{Li et~al.(2018)Li, Marsaglia, Garth, Woodring, Clyne, and
  Childs}]{li2018data}
\bibinfo{author}{S.~Li}, \bibinfo{author}{N.~Marsaglia},
  \bibinfo{author}{C.~Garth}, \bibinfo{author}{J.~Woodring},
  \bibinfo{author}{J.~Clyne}, \bibinfo{author}{H.~Childs},
\newblock \bibinfo{title}{Data reduction techniques for simulation,
  visualization and data analysis},
\newblock in: \bibinfo{booktitle}{Computer Graphics Forum},
  volume~\bibinfo{volume}{37}, \bibinfo{organization}{Wiley Online Library},
  pp. \bibinfo{pages}{422--447}.
\bibitem[{Perlman et~al.(2007)Perlman, Burns, Li, and
  Meneveau}]{perlman2007data}
\bibinfo{author}{E.~Perlman}, \bibinfo{author}{R.~Burns},
  \bibinfo{author}{Y.~Li}, \bibinfo{author}{C.~Meneveau},
\newblock \bibinfo{title}{Data exploration of turbulence simulations using a
  database cluster},
\newblock in: \bibinfo{booktitle}{Proceedings of the 2007 ACM/IEEE conference
  on Supercomputing}, \bibinfo{organization}{ACM}, p.~\bibinfo{pages}{23}.
\bibitem[{Li et~al.(2008)Li, Perlman, Wan, Yang, Meneveau, Burns, Chen, Szalay,
  and Eyink}]{li2008public}
\bibinfo{author}{Y.~Li}, \bibinfo{author}{E.~Perlman},
  \bibinfo{author}{M.~Wan}, \bibinfo{author}{Y.~Yang},
  \bibinfo{author}{C.~Meneveau}, \bibinfo{author}{R.~Burns},
  \bibinfo{author}{S.~Chen}, \bibinfo{author}{A.~Szalay},
  \bibinfo{author}{G.~Eyink},
\newblock \bibinfo{title}{A public turbulence database cluster and applications
  to study lagrangian evolution of velocity increments in turbulence},
\newblock \bibinfo{journal}{Journal of Turbulence}  (\bibinfo{year}{2008})
  \bibinfo{pages}{N31}.

\end{thebibliography}
\bibliographystyle{model1-num-names}
\biboptions{sort&compress}

\end{document}